\documentclass[preprint]{elsarticle}
\usepackage[utf8]{inputenc}
\usepackage{appendix}
\usepackage{subcaption}
\usepackage[a4paper, margin=0.8in]{geometry} 

\usepackage{graphicx}

\usepackage[most]{tcolorbox}
\usepackage{tikz,tikz-cd}
\usetikzlibrary{decorations.pathmorphing, decorations.pathreplacing, patterns, angles, quotes, arrows.meta, calc}

\usepackage{pgfplots}
\pgfplotsset{compat=1.18}

\usepackage{xspace}
\usepackage{comment}
\usepackage{amsmath,amssymb,amsfonts,bm,bbm,amsthm}
\usepackage{multirow}
\usepackage{derivative}
\usepackage{xcolor}
\usepackage{pifont}
\usepackage{hyperref}

\newcommand{\rewone}[1]{\textcolor{black}{#1}}
\newcommand{\rewthree}[1]{\textcolor{black}{#1}}
\newcommand{\rewfour}[1]{\textcolor{black}{#1}}

\renewcommand\d{\ensuremath{\mathrm{d}}}

\newcommand{\bbR}{\mathbb{R}}

\newcommand{\inpr}[3][]{\ensuremath{( #2, \, #3 )_{#1}}}
\newcommand{\dualpr}[3][]{\ensuremath{\langle #2, \, #3 \rangle_{#1}}}

\newcommand{\mat}[1]{\widehat{\bm{#1}}}
\newcommand{\matcomp}[1]{\widehat{#1}}

\newtheorem{remark}{Remark}
\newtheorem{proposition}{Proposition}

\journal{ArXiv}

\begin{document}

\begin{frontmatter}

\title{Mixed finite
element discretization of intrinsic geometrically exact beams for multiplier-free multibody dynamics}

\author[ICA]{Andrea Brugnoli\corref{cor}}
\cortext[cor]{Corresponding author}
\ead{andrea.brugnoli@isae-supaero.fr}

\author[KIT,CHE,ILM]{Phi \,L. Kinon}
\ead{phi.kinon@math.tu-chemnitz.de}

\author[ISAE]{Francesco Sanfedino}
\ead{francesco.sanfedino@isae-supaero.fr}

\author[KIT]{Peter Betsch}
\ead{peter.betsch@kit.edu}

\author[MARY]{Olivier A. Bauchau}
\ead{obauchau@umd.edu}

\address[ICA]{Univ Toulouse, IMT Mines Albi, INSA Toulouse, ISAE-SUPAERO, CNRS, ICA, Toulouse, France}
\address[KIT]{Institute of Mechanics, Karlsruhe Institute of Technology (KIT), Otto-Ammann-Platz 9, 76131 Karlsruhe, Germany}
\address[CHE]{Institute of Mathematics, Chemnitz University of Technology, Reichenhainer Straße 41, 09126 Chemnitz, Germany}
\address[ILM]{Institute of Mathematics, Technische Universität Ilmenau, Weimarer Straße 25, 98693 Ilmenau, Germany}
\address[ISAE]{Féderation ONERA ISAE ENAC, 10 Avenue Marc Pélegrin, Cedex 4, Toulouse, BP-54032, 31055, France}
\address[MARY]{Department of Aerospace Engineering, University of Maryland, College Park, MD, 20742, United States}

\begin{abstract}
The Reissner–Simo and Hodges models are two equivalent continuous descriptions of finite-strain beam dynamics. The Reissner–Simo formulation uses displacements and rotations, while the Hodges formulation is intrinsic and avoids both variables. Although equivalent in theory, the two approaches behave differently after discretization and offer distinct numerical advantages.\\
In this work, we develop a structure-preserving discretization \rewthree{of the intrinsic planar beam formulation}. Because the intrinsic equations involve linear differential operators, both kinematic and dynamic boundary conditions can be imposed naturally using mixed finite elements. The resulting formulation also enables multibody systems to be assembled without algebraic constraints, avoiding the stiff differential–algebraic equations typically introduced by kinematic constraints.\\
We demonstrate the approach on different examples, also showing that closed kinematic loops can be modeled without algebraic constraints. The resulting interconnected systems retain a port-Hamiltonian structure, with all nonlinearities confined to the interconnection operator. This structure allows exact energy preservation when combined with implicit midpoint time integration. Furthermore the scheme appears to require less Newton iterations compared to existing energy preserving scheme.
\end{abstract}


\end{frontmatter}

\section{Introduction}
Highly flexible beams play a significant role across multiple engineering disciplines. These structures can efficiently absorb vibrations and shocks, reduce structural mass, and allow for compact storage and deployment. Some applications include soft and surgical robotics \cite{albu2008soft,omisore2022}, highly flexible aircrafts \cite{chang2008flight}, wind turbines \cite{chu2020comparative} and deployable space structures \cite{chen2024deployable}. A widely used modeling approach for these structures is the shear-deformable Simo-Reissner formulation \cite{reissner1972,simo1985beam}, which accommodates large rotations and deformations while treating cross sections as rigid. It is well established that this beam model also admits an intrinsic or material representation, where the dynamics can be expressed without explicitly tracking positions or rotations \cite{hodges2003intrinsic}. The interested reader can consult \cite{rodriguez2021control} for a formal proof of the equivalence between the intrinsic and displacement-based formulation. The intrinsic form naturally reveals a port-Hamiltonian (pH) structure \cite{palacios2011nonlinear,artola2021modal,ayala2023energybased}, which can be exploited for state-of-the-art simulation and control methodologies. More recently, a pH formulation that equips the displacement-based approach with independent material strain measures has been proposed \cite{kinon2025energy}. This framework further avoids shear locking and employs a simple midpoint time integration scheme that guarantees exact energy and momentum preservation. Another recent contribution \cite{ljukovac2025reissner} discusses yet another Hamiltonian formulation where spatial velocities and strain are simultaneously considered. 
Although the intrinsic formulation offers advantages---such as avoiding rotational parameterizations---its numerical treatment has received comparatively limited attention in the existing literature, see exception~\cite{khaneh2025intrinsic}. \\

A recurring challenge in flexible multibody dynamics is the presence of stiff differential--algebraic equations arising from kinematic constraints between interconnected components \cite{bauchau2011flexible}. \rewone{In conventional displacement-based multibody formulations, kinematic interconnections are commonly enforced through algebraic constraint equations and associated Lagrange multipliers.} This limitation stems from the nonlinear appearance of the differential operators. In contrast, the intrinsic formulation features differential operators that enter the equations linearly, allowing kinematic or dynamic boundary conditions to be applied naturally through an appropriate mixed finite element discretization. \rewone{This property is primarily a consequence of the intrinsic representation and the chosen weak formulation, rather than a benefit that is unique to the pH framework.} Similarly, in the spatial formulation proposed in \cite{ljukovac2025reissner} the differential operator enter linearly and by using different mixed finite element formulations different sets of natural boundary conditions can be considered.  \\

The objective of this contribution is twofold. First, it presents an overview of several port-Hamiltonian formulations of the Simo--Reissner model: 
\begin{itemize} 
\item the spatial--material pH representation proposed in \cite{kinon2025energy};
\item the spatial pH formulation introduced in \cite{ljukovac2025reissner};
\item the intrinsic (or material) formulation originally discussed in \cite{hodges2003intrinsic}. The pH formulation of the Hodges beam was initially noticed in \cite{wynn2013energy}, and then detailed in \cite{artola2021modal,ayala2023energybased}.
\end{itemize}
The equivalence between the classical Simo and Hodges formulation is proven in the PhD thesis by Charlotte Rodriguez \cite{rodriguez2021control}. The list is non exhaustive as more formulations may be added to the list. These formulations exhibit different types of nonlinearities. Although they are equivalent at the continuous level, they are not equivalent once discretized. 

Second, we demonstrate how the pH intrinsic beam formulation can be discretized in a structure-preserving manner to obtain interconnected multibody systems without algebraic constraints. \rewthree{The intrinsic formulation does not eliminate the nonlinearities but rather relocates them to the algebraic part of the interconnection operator, resulting in a semilinear partial differential system. Since the differential operators enter the equations linearly, different mixed finite element formulations can be employed to discretize the system and naturally accommodate a variety of boundary conditions.}  A key feature common to all these discretizations is the satisfaction of compatibility conditions between the finite element spaces, which eliminates shear-locking effects. \rewone{The mixed finite element discretization is not, in itself, specific to the port-Hamiltonian (pH) formalism; rather, it provides a convenient numerical realization of the intrinsic weak formulation and its boundary ports. The main benefit attributed specifically to the pH framework is instead the representation of subsystem interconnections through power ports rather than through algebraic kinematic constraints. This enables a modular assembly in which individual components can be formulated independently and coupled through their exchanged power, while preserving the underlying energy and passivity structure.}  Two formulations with opposite input-output behavior (e.g. a free and a clamped beam) can be coupled via a feedback interconnection (or in port-Hamiltonian jargon a gyrator interconnection). \rewone{This power-preserving interconnection is a distinctive feature of the pH description and is central to the proposed modular construction.} This provides a great flexibility when assembling multibody systems and produces discrete systems without Lagrange multipliers. The methodology is explained in two simple cases: a cantilever beam (obtained by interconnecting a clamped and a free beam) and a four bar mechanism. With the latter example, we demonstrate that closed kinematic loops can be modeled without Lagrange multipliers. \rewthree{The formulation is primarily developed for flexible beam subsystems. Rigid bodies can be regarded as the infinite-stiffness limit of the formulation, although this introduces algebraic constraints; in practice, they may only be approximated using a large but finite Young's modulus. Flexible links can be incorporated provided that their inertial properties are properly accounted for, consistently with existing port-Hamiltonian approaches to multibody systems \cite{macchelli2009port}. Related work has also demonstrated the modular construction of rigid multibody systems within the port-Hamiltonian framework, including the treatment of algebraic constraints \cite{berger2025rigid}. Actuators can likewise be incorporated through the corresponding force and moment contributions. In particular, their inclusion is simpler than in \cite{kinon2026mixed}, since the forces and moments associated with tendon- or pressure-driven actuation do not need to be rotated into the spatial frame
}

The resulting interconnected systems preserve the pH formulation, where the nonlinearities are restricted to the interconnection operator, whereas the Hamiltonian is quadratic in the state. Thanks to this feature, the energy is preserved by applying the implicit midpoint scheme. \rewone{The relevance of the implicit midpoint is that, when applied to the present quadratic pH systems, it preserves the discrete power/energy balance and therefore retains the energetic structure of the continuous model \cite{kotyczka2019discrete}.} While we mostly restrict ourselves to the planar case, the methodology can be extended to spatial problems. \rewone{This is however not trivial since rotations belong to SO(3) and arbitrary spatial joint orientations lead to additional complexity in the interconnection and compatibility conditions.} The newly proposed method is validated numerically on different benchmarks: the flying spaghetti, a flexible pendulum with a lumped mass attached at the end, an L-shaped frame, and a four-bar mechanism. \rewone{The numerical experiments also indicate favorable nonlinear-solver behavior, as measured by the number of Newton iterations. We emphasize, however, that this measure is used as an indicator of computational behavior rather than as a claim of universal efficiency superiority over conventional geometrically exact beam or multibody formulations.}

The paper is organized as follows. In the remainder of this section, we introduce notation and a problem statement. In Sec. \ref{sec:formulations} the different pH formulations of geometrically exact beams are presented. The spatial finite element discretization is then discussed in Sec. \ref{sec:spatial_discretization}. The assembly of multibody systems by feedback interconnection is illustrated in Sec. \ref{sec:interconnection}. Sec. \ref{sec:time_integration} explains the time integration method. Numerical experiments are presented in Sec. \ref{sec:num_examples}.

\subsection{General Notation}

Vectors are written in boldface lowercase $\bm{x} \in \bbR^n$, whereas matrices are written in boldface uppercase $\bm{A}\in \bbR^{m \times n}$. The $i$-th vector of the canonical Euclidean basis is denoted by 
$\bm{b}_i$. A one dimensional domain parametrized by an arc-length is considered, i.e. $s \in \Omega=[0, L]$. The inner product of scalar or vector-valued quantities is denoted by
$$
\inpr[\Omega]{f}{g} := \int_\Omega f g \, \d{s}, \qquad 
\inpr[\Omega]{\bm{f}}{\bm{g}} := \int_\Omega \bm{f} \cdot \bm{g} \, \d{s}.
$$
Given the one-dimensional nature of the models considered in this work, the inner product on the boundary reads
$$
\dualpr[\partial\Omega]{f}{g} := f(L)g(L) - f(0)g(0), \qquad 
\dualpr[\partial\Omega]{\bm{f}}{\bm{g}} := \bm{f}(L) \cdot \bm{g}(L) - \bm{f}(0) \cdot \bm{g}(0).
$$

Differential matrix operators involving derivatives are denoted by calligraphic uppercase letters $\mathcal{A}$. Matrix transposition is denoted by $\bm{A}^{\top}$ and the formal adjoint of an operator by $\mathcal{A}^*$. Let $\bm{x}, \bm{y} \in C_c^\infty(\Omega)$ be smooth (possibly vector-valued) compactly supported fields. The formal adjoint is defined via integration by parts as follows
$$
\inpr[\Omega]{\mathcal{A} \bm{x}}{\bm{y}} = \inpr[\Omega]{\bm{x}}{\mathcal{A}^* \bm{y}}.
$$
Matrix--vector multiplication and the action of an operator are denoted similarly as $\bm{Ax}$ and $\mathcal{A} \bm{x}$. A short bracket notation is used to express $ (\bm{x}, \bm{y}) := [\bm{x}^\top,\, \bm{y}^\top]^\top$, and analogously for more than two variables. Identity matrices are represented by $\bm I$ and matrices full of zeros by $0$, their dimension being clear from the context. Numerical realization of vector and matrix variables after finite element discretization are denoted using bold math roman fonts $\mathbf{x}, \; \mathbf{A}$. The symbol $\square$ represents a placeholder. The partial derivative of a function $f$ is denoted by $\partial_{\square} f$.
To represent the cross product between two vectors $\bm{x}, \; \bm{y}$ in $\bbR^2$ we introduce the skew-symmetric matrix $\bm{S} \in \mathbb{R}^{2\times 2}$ such that \cite{kinon2025energy}
\begin{equation}
    \label{eq_Skew_mat}
\bm{x} \times \bm{y} = x_1 y_2 - x_2 y_1 = \bm{x}^\top \bm{S} \bm{y}, \qquad 
\bm{S} = \begin{bmatrix}
    0 & 1 \\
    -1 & 0
\end{bmatrix}.
\end{equation}

\subsection{Problem statement}

\paragraph{Material and spatial configurations}
Consider a material beam configuration $\Omega = [0,L]$, where $L \in \mathbb{R}$ denotes the beam length. The term \textit{material} indicates that quantities are expressed through their components with respect to a frame attached to the beam cross-section. For simplicity, the stress-free configuration of the beam is assumed to be a straight line; an initially curved configuration may be incorporated through an appropriate modification of the formulation.

The motion of the centerline is described by the position vector $\bm{r}(s,t) \in \Omega_t \subset \mathbb{R}^2$. The independent variables are the arc-length coordinate $s \in \Omega$, which parameterizes the centerline of the straight, stress-free initial configuration $\Omega_0 = {\bm{r}(s, 0) = s\bm{b}_1 \mid s \in \Omega}$, and time $t \in [0,T]$. The cross-sections remain planar throughout the motion, and their orientation is characterized by an angle $\theta : \Omega \times [0,T] \rightarrow \mathbb{R}$. 
In addition to the material representation, a \textit{spatial} configuration may also be introduced. In this setting, spatial quantities are expressed in an inertial frame. Accordingly, the spatial notation $\bm{x}$ and its material counterpart $\mat{x}$ are vectors in $\mathbb{R}^2$ representing the components of the same geometric entity. These are distinguished by a hat and are related by
$$
\bm{x} = \bm{\Lambda}(\theta) \mat{x}, \qquad \bm{\Lambda}(\theta) \in \mathrm{SO}(2),
$$
where the rotation matrix reads
$$
\bm{\Lambda}(\theta(s, t)) = \begin{bmatrix}
    \cos \theta(s, t) & - \sin \theta(s, t) \\
    \sin \theta(s, t) & \cos \theta(s, t)
\end{bmatrix}.
$$
The time derivative of the rotation matrix gives
\begin{equation}\label{eq:derivative_rot}
    \partial_t\bm{\Lambda} = \omega \bm{S}^\top \bm{\Lambda} = - \omega \bm{S}\bm{\Lambda}, \qquad \omega:=\partial_t {\theta},
\end{equation}
where $\bm S$ has been introduced in \eqref{eq_Skew_mat}.

\paragraph{Geometric and physical parameters}
An isotropic homogeneous material is considered with cross section area $A$ and second moment of area $I$, mass density $\rho$, Young's modulus $E$, shear modulus $G$ and shear correction factor $k$.
Later in this work, we make use of the so-called compliance form, meaning that the dynamics is written directly in terms of the stresses \cite{brugnoli2025nonlinear,dejong2026decomposition,kinon2026mixed,herrmann2026mixed,greco_2024_objective,kim_1998_new,santos_2010_hybrid,thoma_2024_velocitystress}. To this aim let us introduce the translational and rotational compliance 
$$
\bm{C}_t = \begin{bmatrix}
    (EA)^{-1} & 0 \\
  0  & (kGA)^{-1}
\end{bmatrix} = \mathrm{Diag}\begin{bmatrix}
    EA \\ 
    kGA
\end{bmatrix}^{-1}, \qquad {C}_r = (EI)^{-1},
$$
where we have introduced the operator of diagonal matrices. In the following, we use this notation as well for block-diagonal matrices.
For simplicity, the material parameters are considered constant along the beam axis. However, non-constant coefficients would not introduce any additional difficulty.

\paragraph{Kinematic quantities}
The spatial and material linear velocities, denoted by $\bm{v}$ and $\mat{v}$, are related through $\bm{v} = \bm{\Lambda}\mat{v}$. The associated spatial and material linear momenta, $\bm{\pi}_v$ and $\mat{\pi}_v$, follow the same transformation and are linked to the velocities by $\bm{\pi}_v = \rho A \bm{v}, \; \mat{\pi}_v = \rho A \mat{v}$. \\
Furthermore, we consider the angular velocity $\omega$ and the corresponding angular momentum $\pi_\omega$. In the planar setting, spatial and material representations coincide. Their relation is given by $\pi_\omega = \rho I \omega$.

\paragraph{Strains and stresses}
The material axial and shear strain $\mat{\gamma}$ are defined kinematically as
$$
\mat{\gamma} =  \bm{\Lambda}^\top \partial_s \bm{r} - \bm{b}_1,
$$
with their spatial counterparts obtained via the transformation $\bm{\gamma} = \bm{\Lambda}\mat{\gamma}$. The corresponding stress resultants in material and spatial form, denoted by $\mat{n}$ and $\bm{n}$, satisfy $\bm{n} = \bm{\Lambda}\mat{n}$. The constitutive relations linking strains and stress resultants are given by
\[
\mat{\gamma} = \bm{C}_t \mat{n}, \qquad \bm{\gamma} = \bm{\Lambda}\bm{C}_t \bm{\Lambda}^\top \bm{n}.
\]
Finally, we consider the curvature $\kappa$ defined kinematically as $\kappa = \partial_s \theta.$ The bending moment $m$ is then given by $\kappa = C_r m.$

\section{Port-Hamiltonian formulations of planar geometrically exact beams}\label{sec:formulations}

In this section we present different (port-)Hamiltonian formulations for planar geometrically exact beams, that have already been presented in the literature. The purpose is to show how nonlinearities enter the dynamics and how they affect the subsequent discretization strategy. To the best of our knowledge, a comprehensive discussion on the different formulations was not presented before. Three distinct formulations will be examined. Additional variants may be derived by switching the representation of individual variables between spatial and material descriptions, and vice versa. Nevertheless, analyzing these three cases is sufficient to highlight the principal differences in how nonlinearities are handled. Note additionally, that in the continuous setting, i.e., before discretization, those formulations are equivalent. But when discretized in space and/or time, certain advantages and disadvantages can be found. \\

For the subsequent discussion on the spatial discretization in Sec. \ref{sec:spatial_discretization}, only the material (or intrinsic) model is considered. One interesting feature of this model is that it is represented by a semi-linear PDE, meaning that the differential operators enter the formulation in a linear way. The (quadratic) nonlinearities are confined to an algebraic term. This particular structure opens up many possibilities for the spatial discretization. In particular it will be shown that different mixed finite element formulations can be employed, leading to distinct representations of the boundary conditions. 

\subsection{Spatial-material formulation}
In this formulation, the spatial velocities are used whereas the stresses are written in the material formulation. The governing PDEs read (cf. \cite{kinon2025energy} for more details)

\begin{equation}\label{eq:spatial_material}
\begin{aligned}
\pdv{}{t}
\begin{pmatrix}
\bm{r} \\
\theta \\
\end{pmatrix} &=
\begin{pmatrix}
    \bm{v} \\
    \omega
\end{pmatrix}, \\
\mathrm{Diag}
\begin{bmatrix}
\rho A \\
\rho I \\
\bm{C}_t \\
{C}_r
\end{bmatrix}
\pdv{}{t}
\begin{pmatrix}
\bm{v} \\
\omega \\
\mat{n} \\
m
\end{pmatrix}
&=
\begin{bmatrix}
0 & 0 & \partial_s (\bm{\Lambda}(\theta)\, \square) & 0 \\
0 & 0  & \partial_s \bm{r}^\top \bm{S}\bm{\Lambda}(\theta)(\square) & \partial_s(\square) \\
\bm{\Lambda}(\theta)^\top \partial_s (\square) & \bm{\Lambda}(\theta)^\top \bm{S} \partial_s \bm{r}(\square)  & 0 & 0 \\
0 & \partial_s(\square) & 0 & 0
\end{bmatrix}
\begin{pmatrix}
\bm{v} \\
\omega \\
\mat{n} \\
m
\end{pmatrix},
\end{aligned}
\end{equation}
or compactly
\begin{equation} \label{eq_ph_sm_inter}
\begin{aligned}
\partial_t \bm{q} &= \bm{P}\bm{e}_{\rm sm}, \\
\bm{M} \partial_t \bm{e}_{\rm sm} &= \mathcal{J}_{\rm sm}^{\bm{e}}(\bm{q}) \bm{e}_{\rm sm},
\end{aligned}
\end{equation}
where $\bm{M}$ is a symmetric positive definite matrix, i.e., $\bm{M}=\bm{M}^\top \succ 0$. Moreover, $\mathcal{J}_{\rm sm}$ is a formally skew-adjoint operator, satisfying $\mathcal{J}^{\bm{e}}_{\rm sm} = -(\mathcal{J}^{\bm{e}}_{\rm sm})^*$ (we refer to \cite{kinon2025energy} for a formal proof). Matrix $\bm{P}$ is a projection (matrix) on the first two variables given by
$$
\bm{P} = \begin{bmatrix}
    \bm{I} &  0 &  0 &  0 \\
      0 & 1 &  0 & 0
\end{bmatrix} .
$$
Additionally, we call $\bm{q} = (\bm{r}, \; \theta)$ the configuration vector and $\bm{e}_{\rm sm} = (\bm{v}, \; \omega, \; \mat{n}, \; m)$ the energy variables. A (port-)Hamiltonian formulation is then obtained as follows \cite{morandin2019descriptor}
\begin{subequations}\label{eq:pH_sm}
\begin{align}
 \bm{E}\partial_t \bm{x}_{\rm sm} &= \mathcal{J}_{\rm sm}(\bm{x}_{\rm sm}) \bm{z}_{\rm sm}, \label{eq:pH_sm_1} \\
    \bm{E}^\top \bm{z}_{\rm sm} &= \fdv{H_{\rm sm}}{\bm{x}_{\rm sm}} = \bm{E}^\top \bm{Q} \bm{x}_{\rm sm}, \label{eq:pH_sm_2}
\end{align}
\end{subequations}
where the second identity makes use of the variational derivative.
Further, the state vector is $\bm{x}_{\rm sm} = (\bm{q}, \bm{e}_{\rm sm})$ and the matrices $\bm{E}, \bm{Q}$ and operator $\mathcal{J}_{\rm sm}$ read
$$
\bm{E} =\begin{bmatrix}
    \bm{I} & 0 \\
     0 & \bm{M}
\end{bmatrix}, \qquad
\bm{Q} =\begin{bmatrix}
     0 &  0 \\
     0 & \bm{I}
\end{bmatrix},
\qquad 
\mathcal{J}_{\rm sm}(\bm{x}_{\rm sm}) = \begin{bmatrix}
     0 & \bm{P} \\
    -\bm{P}^\top & \mathcal{J}_{\rm sm}^{\bm{e}}(\bm{q})
\end{bmatrix}.
$$
The total energy, or Hamiltonian, of the model is the following quadratic form 
$$
\begin{aligned}
   H_{\rm sm} &= \frac{1}{2}\int_\Omega  (\rho A \bm{v}^\top \bm{v} + \rho I \omega^2 +  \mat{n}^\top \bm{C}_t \mat{n} + C_r m^2) \; \d{s}, \\
   &= \frac{1}{2}\int_\Omega \bm{e}_{\rm sm}^\top \bm{M}\bm{e}_{\rm sm} \; \d{s} = \frac{1}{2}\int_\Omega \bm{x}_{\rm sm}^\top \bm{E}^\top \bm{Q}\bm{x}_{\rm sm} \; \d{s}, \\
\end{aligned}
$$
The time derivative of the Hamiltonian provides
$$
\begin{aligned}
\dot{H}_{\rm sm} = \inpr[\Omega]{\delta_{\bm{x}_{\rm sm}}H_{\rm sm}}{\partial_t \bm{x}_{\rm sm}},
\overset{\eqref{eq:pH_sm_2}}{=} \inpr[\Omega]{\bm{E}^\top \bm{z}_{\rm sm}}{\partial_t \bm{x}_{\rm sm}}, 
= \inpr[\Omega]{\bm{z}_{\rm sm}}{\bm{E} \partial_t \bm{x}_{\rm sm}}, 
\overset{\eqref{eq:pH_sm_1}}{=}  \inpr[\Omega]{\bm{z}_{\rm sm}}{\mathcal{J}_{\rm sm}(\bm{x}_{\rm sm}) \bm{z}_{\rm sm}}.
\end{aligned}
$$
This result is a consequence of the general descriptor form presented in Eq. \eqref{eq:pH_sm}. The algebraic part of the operator $\mathcal{J}_{\rm sm}$ vanishes due to the skew-symmetry. Expanding the differential operator provides
$$
\begin{aligned}
\dot{H}_{\rm sm} &=  \inpr[\Omega]{\bm{z}_{\rm sm}}{\mathcal{J}_{\rm sm}(\bm{x}_{\rm sm}) \bm{z}_{\rm sm}}, \\
&= \inpr[\Omega]{\bm{v}}{\partial_s (\bm{\Lambda}(\theta)\, \mat{n}) } + \inpr[\Omega]{\omega}{\partial_s m} + \inpr[\Omega]{\mat{n}}{\bm{\Lambda}(\theta)^\top\partial_s \bm{v}} + \inpr[\Omega]{m}{\partial_s \omega} , \\
&= \dualpr[\partial\Omega]{\bm{v}}{\bm{\Lambda}(\theta)\mat{n}} + \dualpr[\partial\Omega]{\omega}{m}.    
\end{aligned}
$$
Therein, we have used integration by parts.
The formal skew-adjointness of $\mathcal{J}_{\rm sm}$ becomes clear now: the temporal variation of the Hamiltonian is caused solely by interactions with the external environment at the boundaries.

Note that the parameters in the matrix $\bm M$ can be set to zero. This leads to an infinite-dimensional pH DAE formulation, which allows the inclusion of massless terms as well as infinitely stiff behaviors, including a shear-stiff representation consistent with a nonlinear Kirchhoff beam \cite{meier2019kirchhoff}. In that case, the corresponding stress resultant acts as a Lagrange multiplier enforcing the constraint in time-differentiated form \cite{kinon2026mixed}. Furthermore, it can be seen from relation \eqref{eq:spatial_material} that the interconnection operator $\mathcal{J}_{\rm sm}(\bm{q})$ depends on both the position and rotation angle.

\begin{remark}
In the previous explanations we have assumed that the Hamiltonian does not dependent on the configuration $\bm q$. However,
if a configuration-dependent potential $U(\bm{q})$ has to be considered, its variational derivatives $\delta_{\bm{r}} {U}, \; \delta_{\theta} {U}$ allows expressing the associated forces and torques. This point will be explained in more detail in Sec. \ref{sec:material}.    
\end{remark}

 \subsection{Spatial formulation}

The spatial pH formulation has been recently presented \cite{ljukovac2025reissner}. Since the stress variables are written in the spatial frame, the constitutive law and consequently the Hamiltonian functional feature rotation matrices.  The governing equations for this case read
\begin{equation} \label{eq_pH_s_orig}
\begin{aligned}
\pdv{}{t}
\begin{pmatrix}
\bm{r} \\
\theta \\
\end{pmatrix} &=
\begin{pmatrix}
    \bm{v} \\
    \omega
\end{pmatrix}, \\
\mathrm{Diag}
\begin{bmatrix}
\rho A \\
\rho I \\
\bm{\Lambda}\bm{C}_t\bm{\Lambda}^\top \\
{C}_r
\end{bmatrix}
\begin{pmatrix}
\partial_t \bm{v} \\
\partial_t \omega \\
\overset{\nabla}{\bm{n}} \\
\partial_t m
\end{pmatrix}
&=
\begin{bmatrix}
0 & 0 & \partial_s  & 0 \\
0 & 0  & \partial_s \bm{r}^\top \bm{S} & \partial_s \\
\partial_s & \bm{S}\partial_s \bm{r} & 0 & 0 \\
0 & \partial_s  & 0 & 0
\end{bmatrix}
\begin{pmatrix}
\bm{v} \\
\omega \\
\bm{n} \\
m
\end{pmatrix}.
\end{aligned}
\end{equation}
The notation $\overset{\nabla}{\bm{n}} $ denotes the Lie derivative of the spatial variable $\bm{n}$. Its expression is given by 
\begin{equation}\label{eq:Lie_derivative}
    \overset{\nabla}{\bm{n}} = \bm{\Lambda}\partial_t(\bm{\Lambda}^\top \bm{n})= \partial_t \bm{n} + \omega \bm{S}\bm{n}.   
\end{equation}
The dynamics can be written in a compact form as 
$$
\begin{aligned}
    \partial_t \bm{q} &= \bm{P} \bm{e}_{\rm s}, \\
    \bm{M}_{\rm s}(\theta)\partial_t \bm{e}_{\rm s} &=\mathcal{J}^{\bm{e}}_{\rm s}(\bm{r}) \bm{e}_{\rm s},
\end{aligned}
$$
where, with a slight abuse of notation, the Lie derivative is denoted by a partial one. The total energy, or Hamiltonian, of the model is the following expression
$$
H_{\rm s} = \frac{1}{2} \int_\Omega  (\rho A \bm{v}^\top \bm{v} + \rho I \omega^2 +  \bm{n}^\top \bm{\Lambda} \bm{C}_t \bm{\Lambda}^\top \bm{n} + C_r m^2) \; \d{s} = \frac{1}{2} \int_\Omega \bm{e}_{\rm s}^\top \bm{M}_{\rm s}(\theta) \bm{e}_{\rm s} \; \d{s}.
$$
Notice that it is not a quadratic form in the variable because of the presence of the rotation matrix~$\bm{\Lambda}$.  The time derivative of the Hamiltonian provides

\begin{equation} \label{eq_ham_s}
\dot{H}_{\rm s} =\int_\Omega \bm{v} \cdot \partial_t \bm{\pi}_v + \omega \partial_t\pi_\omega + \bm{n} \cdot \overset{\nabla}{\bm{\gamma}} + m\partial_t \kappa \; \d{s}=\dualpr[\partial\Omega]{\bm{v}}{\bm{n}} + \dualpr[\partial\Omega]{\omega}{m}.    
\end{equation}

Notice that the operator $\mathcal{J}_{\rm s}^{\bm{e}}(\bm{r})$ only depends on the centerline position and not on the angle $\theta$.

\paragraph{Descriptor port-Hamiltonian form}
The previous formulation can be written as follows
\begin{equation}
\mathrm{Diag}
\begin{bmatrix}
\bm{I} \\
1 \\
\rho A \\
\rho I \\
\bm{\Lambda}\bm{C}_t\bm{\Lambda}^\top \\
{C}_r
\end{bmatrix}
\begin{pmatrix}
\partial_t \bm{r} \\
\partial_t \theta \\
\partial_t \bm{v} \\
\partial_t \omega \\
\overset{\nabla}{\bm{n}} \\
\partial_t m
\end{pmatrix}
=
\begin{bmatrix}
0 & 0 & \bm{I} & 0 & 0 & 0 \\
0 & 0 & 0 & 1 & 0 & 0 \\
- \bm{I} & 0 & 0 & 0 & \partial_s  & 0 \\
0 & -1 & 0 & 0  & \partial_s \bm{r}^\top \bm{S} & \partial_s \\
0 & 0 & \partial_s & \bm{S}\partial_s \bm{r} & 0 & 0 \\
0 & 0 & 0 & \partial_s  & 0 & 0
\end{bmatrix}
\begin{pmatrix}
0 \\
0 \\
\bm{v} \\
\omega \\
\bm{n} \\
m
\end{pmatrix}.
\end{equation}
Because of the appearance of the Lie derivative, the formulation is not in standard port-Hamiltonian form, see, e.g., \eqref{eq:pH_sm}. To find the port-Hamiltonian representation we start from the expression of the Hamiltonian \eqref{eq_ham_s}
$$
H_s = \frac{1}{2}\int_\Omega \bm{e}_{\rm s}^\top \bm{M}_{\rm s}(\theta) \bm{e}_{\rm s} \; \d{s},
$$
The variational derivative of the Hamiltonian with respect to the state is given by
$$
    \fdv{H_{\rm s}}{\bm r} = 0, \qquad
    \fdv{H_{\rm s}}{\theta} = \bm{n}^\top \bm{S}^\top \bm{\Lambda}\bm{C}_t\bm{\Lambda}^\top\bm{n}, \qquad
    \fdv{H_{\rm s}}{\bm{e}_{\rm s}} = \bm{M}_{\rm s}(\theta) \bm{e}_{\rm s}, \\
$$
where the second derivative arises from $\partial_\theta (\bm{n}^\top \bm{\Lambda}\bm{C}_t\bm{\Lambda}^\top \bm{n}) = \bm{n}^\top \bm{S}^\top \bm{\Lambda}\bm{C}_t\bm{\Lambda}^\top \bm{n}.$
In the next step,  a corrected skew-adjoint operator has to be defined to account for the contribution of the aforementioned term  and the additional Lie derivative term
 $$\displaystyle \fdv{H_{\rm s}}{\theta} = \bm{n}^\top \bm{S}^\top \bm{\Lambda}\bm{C}_t\bm{\Lambda}^\top\bm{n}, \qquad  \overset{\nabla}{\bm{n}}= \partial_t \bm{n} + \omega \bm{S}\bm{n}.$$
Given these two aspects, the corrected operator reads
$$
\mathcal{J}_{\rm s}(\bm{x}_{\rm s}) =
\begin{bmatrix}
0 & 0 & \bm{I} & 0 & 0 & 0 \\
0 & 0 & 0 & 1 & 0 & 0 \\
- \bm{I} & 0 & 0 & 0 & \partial_s  & 0 \\
0 & -1 & 0 & 0  & \partial_s \bm{r}^\top \bm{S} + \bm{G} & \partial_s \\
0 & 0 & \partial_s & \bm{S}\partial_s \bm{r} - \bm{G}^\top & 0 & 0 \\
0 & 0 & 0 & \partial_s  & 0 & 0
\end{bmatrix}, \qquad\qquad \bm{G} := \bm{n}^\top \bm{S}^\top \bm{\Lambda}\bm{C}_t\bm{\Lambda}^\top.
$$
The operator $\mathcal{J}_{\rm s}$ is again formally skew adjoint. Given the overall state $\bm{x}_{\rm s} = \begin{pmatrix} \bm{r}, \theta, \bm{v}, \omega, \bm{n}, m\end{pmatrix},$ the dynamics for the augmented system can be written in the following canonical Hamiltonian descriptor form \cite{morandin2019descriptor}
$$
\begin{aligned}
\bm{E}_{\rm s}(\bm{x}_{\rm s}) \partial_t \bm{x}_{\rm s} &= \mathcal{J}_{\rm s}(\bm{x}_{\rm s})  \bm{z}_{\rm s}, \\
\bm{E}_{\rm s}^\top \bm{z}_s &= \fdv{H_{\rm s}}{\bm{x}_{\rm s}} .
\end{aligned}
$$
In this case the variational derivative of the Hamiltonian contains nonlinear terms.

\subsection{Material formulation}\label{sec:material}

This formulation is also called intrinsic in the literature. The term intrinsic, used by Hodges in his seminal work \cite{hodges2003intrinsic}, refers to the fact that the displacement and rotation do not appear in the unforced dynamics and can be reconstructed a posteriori from the knowledge of the linear and angular velocity. Since this formulation is the one that we will consider later for discretization and simulation, we will detail the port-Hamiltonian formulation and the incorporation of additional potential like gravity. The equations are given in the following form

\begin{equation} 
\begin{aligned}
\pdv{}{t}
\begin{pmatrix}
    \bm{r} \\
    \theta
\end{pmatrix} &= 
\begin{bmatrix}
    \bm{\Lambda}(\theta) & 0 \\
    0 & 1
\end{bmatrix}
\begin{pmatrix}
    \mat{v} \\
    \omega
\end{pmatrix}, \\
\mathrm{Diag}
\begin{bmatrix}
\rho A \\
\rho I \\
\bm{C}_t\\
C_r
\end{bmatrix}
\pdv{}{t}
\begin{pmatrix}
\mat{v} \\
\omega \\
\mat{n} \\
m
\end{pmatrix}
&=
\left(
\begin{bmatrix}
0 &  0 & \partial_s & 0 \\
0 & 0 & \bm{b}_2^\top  & \partial_s \\
\partial_s & -\bm{b}_2 & 0 & 0 \\
0 & \partial_s  & 0 & 0
\end{bmatrix}
+ 
\begin{bmatrix}
0 & \bm{S} \mat{\pi}_V & \bm{S}^\top \kappa & 0 \\
\mat{\pi}_V^\top \bm{S} & 0 & \mat{\gamma}^\top \bm{S}  & 0 \\
\bm{S}^\top \kappa & \bm{S}\mat{\gamma}  & 0 & 0 \\
0 & 0 & 0 & 0
\end{bmatrix}
\right)
\begin{pmatrix}
\mat{v} \\
\omega \\
\mat{n} \\
m
\end{pmatrix}.
\end{aligned} \label{eq_ph_inter_m}
\end{equation}
where $\bm{b}_2 =(0,1)$ is the second element of the Euclidean canonical basis and it arises from $\bm{S}\bm{b}_1 = -\bm{b}_2$. Unlike in the previous two formulations, here, the differential operator $\mathcal{J}_{\rm m}^{\bm{e}}$ can be split into a constant differential part and a state modulated algebraic (matrix) part 
$$\mathcal{J}_{\rm m}^{\bm{e}}(\bm{e}_{\rm m}) = \mathcal{J}_{\rm m, d}^{\bm{e}} + \bm{J}_{\rm m, a}^{\bm{e}}(\bm{e}_{\rm m}), \qquad \bm{e}_{\rm m}=(\mat{v}, \, \omega, \, \mat{n}, \, m),
$$ 
defined as
\begin{equation}
\mathcal{J}_{\rm m, d}^{\bm{e}}:=\begin{bmatrix}
0 &  0 & \partial_s & 0 \\
0 & 0 & \bm{b}_2^\top   & \partial_s \\
\partial_s & -\bm{b}_2 & 0 & 0 \\
0 & \partial_s  & 0 & 0
\end{bmatrix}, 
\qquad 
\bm{J}_{\rm m, a}^{\bm{e}}(\bm{e}_m):=\begin{bmatrix}
0 & \bm{S} \mat{\pi}_v & \bm{S}^\top \kappa & 0 \\
\mat{\pi}_v^\top \bm{S} & 0 & \mat{\gamma}^\top \bm{S}  & 0 \\
\bm{S}^\top \kappa & \bm{S}\mat{\gamma}  & 0 & 0 \\
0 & 0 & 0 & 0
\end{bmatrix}.
\end{equation}
The system can be compactly written as 
$$
\begin{aligned}
\partial_t \bm{q} &= \bm{P}_{\rm m}(\theta)\bm{e}_{\rm m}, \\
\bm{M} \partial_t \bm{e}_{\rm m} &= \mathcal{J}_{\rm m}^{\bm{e}}(\bm{e}_{\rm m})\bm{e}_{\rm m}.
\end{aligned}
$$
where the different quantities are defined through comparison with \eqref{eq_ph_inter_m}.
This leads to the port-Hamiltonian formulation
\begin{equation}
\begin{aligned}
    \bm{E}\partial_t \bm{x}_{\rm m} &= \mathcal{J}_{\rm m}(\bm{x}_{\rm m}) \bm{z}_{\rm m}, \\
    \bm{E}^\top \bm{z}_{\rm m} &= \fdv{H_{\rm m}}{\bm{x}_{\rm m}} = \bm{E}^\top \bm{Q} \bm{x}_{\rm m},
\end{aligned}
\end{equation}
where the state vector is $\bm{x}_{\rm m} = (\bm{q},\bm{e}_{\rm m})$ and the matrices $\bm{E}, \bm{Q}$ and operator $\mathcal{J}_{\rm m}$ read
$$
\bm{E} =\begin{bmatrix}
    \bm{I} &  0 \\
     0 & \bm{M}
\end{bmatrix}, \qquad
\bm{Q} =\begin{bmatrix}
    0 &  0 \\
     0 & \bm{I}
\end{bmatrix},
\qquad 
\mathcal{J}_{m} = \begin{bmatrix}
     0 & \bm{P}_{\rm m}(\theta) \\
    -\bm{P}_{\rm m}^\top(\theta) & \mathcal{J}_{\rm m}^{\bm{e}}(\bm{e}_{\rm m})
\end{bmatrix}.
$$
Once again the operator $\mathcal{J}_{\rm m}$ can be split into a constant differential part and a state modulated matrix 
$$\mathcal{J}_{\rm m}(\bm{x}_{\rm m}) = \mathcal{J}_{\rm m, d} + \bm{J}_{\rm m, a}(\bm{x}_{\rm m}),$$ defined as follows
$$
\mathcal{J}_{\rm m, d} = \begin{bmatrix}
    0 & 0 \\
    0 & \mathcal{J}_{\rm m, d}^{\bm{e}}
\end{bmatrix}, \qquad 
\bm{J}_{\rm m, a}(\bm{x}_{\rm m}) = \begin{bmatrix}
    0 & \bm{P}_{\rm m}(\theta) \\
    -\bm{P}_{\rm m}^\top(\theta) & \bm{J}_{\rm m, a}^{\bm{e}} (\bm{e}_{\rm m})
\end{bmatrix}.
$$
The total energy of the system is given by
$$
H_{\rm m}= \frac{1}{2}\int_\Omega  (\rho A \mat{v}^\top\mat{v} + \rho I \omega^2 +  \mat{n}^\top \bm{C}_t \mat{n} + C_r m^2) \; \d{s} = \frac{1}{2}\int_\Omega \bm{e}_{\rm m}^\top \bm{M}\bm{e}_{\rm m} \; \d{s}.
$$
Since $\mathcal{J}_{\rm m}$ is formally skew-adjoint, the time derivative of the Hamiltonian provides
$$
\dot{H}_{\rm m} = \dualpr[\partial\Omega]{\mat{v}}{\mat{n}} + \dualpr[\partial\Omega]{\omega}{m}.
$$
One can notice that the differential operator $\mathcal{J}_{\rm m, d}$ is linear. Different mixed variational formulations can be constructed, resulting in different natural boundary conditions. This specificity of the intrinsic model 
distinguishes it from the two previously discussed formulations for the following reasons: 
\begin{itemize}
    \item Because of its quasilinear nature, a weak formulation for the spatial-material formulation necessarily relies on the use of natural Neumann boundary conditions.
    \item For the spatial formulation, the energy directly depends on the geometrical configuration via the rotation angle, thus making it difficult to preserve the energy.
\end{itemize}
For this reason, we prefer the intrinsic formulation and disregard the other two representations in the remainder of this work.

\paragraph{The linear case}
The linear case is governed by the following PDEs
\begin{equation}\label{eq:linear_system}
\mathrm{Diag}
\begin{bmatrix}
\rho A \\
\rho I \\
\bm{C}_t\\
C_r
\end{bmatrix}
\pdv{}{t}
\begin{pmatrix}
\mat{v} \\
\omega \\
\mat{n} \\
m
\end{pmatrix}
=
\begin{bmatrix}
 0 & 0 & \partial_s & 0 \\
 0& 0 & \bm{b}_2^\top & \partial_s \\
 \partial_s & -\bm{b}_2  & 0 & 0 \\
 0 & \partial_s & 0 & 0
\end{bmatrix}
\begin{pmatrix}
\mat{v} \\
\omega \\
\mat{n} \\
m
\end{pmatrix} \quad \Leftrightarrow \quad \bm{M}\partial_t \bm{e}_{\rm m} = \mathcal{J}_{\rm m, d}^{\bm e} \bm{e}_{\rm m}.
\end{equation}
This system is equivalent to the Timoshenko beam model \cite{macchelli2004timoshenko}. 
In the linear regime, the material and spatial version of a vector coincide.

\paragraph{Incorporation of the gravity potential}

The gravity potential, with gravity acceleration $g$, is written using the centerline vertical position $r_y=\bm{r}^\top\bm{b}_2$ as follows
\begin{equation} \label{eq:gravity_potential}
   U = \int_\Omega \rho A g r_y \, \d s.
\end{equation}
The gravity force per unit reference length is then obtained as $\bm{g} = - \fdv{U}{\bm{r}}$. The overall Hamiltonian system is written 
$$
\begin{bmatrix}
    \bm{I} & 0 \\
    0 & \bm{M}
\end{bmatrix} \pdv{}{t}
\begin{pmatrix}
    \bm{q} \\ \bm{e}_{\rm m}
\end{pmatrix} = 
\begin{bmatrix}
    0 & \bm{P}_{\rm m}(\theta) \\
    -\bm{P}_{\rm m}^\top(\theta) & \mathcal{J}_{\rm m}^{\bm{e}}(\bm{e}_{\rm m})
\end{bmatrix}
\begin{pmatrix}
    \delta_{\bm{q}} U \\
    \bm{e}_{\rm m}
\end{pmatrix}.
$$

Note that the effect of gravity is nonlinear in the material formulation as the force has to be rotated in the material frame.

\section{Spatial discretization of the material formulation}\label{sec:spatial_discretization}

The discretization will be explained only for the linear part, as the differential operators in the intrinsic formulation appear only linearly. The discretization of the full nonlinear system \eqref{eq_ph_inter_m} will be briefly explained later on. As all nonlinear terms are of algebraic nature, their finite element discretization boils down to a projection onto the finite element spaces. The weak form of the linear system \eqref{eq:linear_system} is given by
\begin{equation} \label{eq_weak_form}
\begin{aligned}
    \inpr[\Omega]{\bm{\psi}_v}{\rho A \partial_t \mat{v}} &= \inpr[\Omega]{\bm{\psi}_v}{\partial_s \mat{n}}, \\
    \inpr[\Omega]{{\psi}_{\omega}}{\rho I \partial_t \omega} &= \inpr[\Omega]{{\psi}_{\omega}}{\bm{b}_2^\top \mat{n}} + \inpr[\Omega]{{\psi}_{\omega}}{\partial_s m}, \\
    \inpr[\Omega]{\bm{\psi}_n}{\bm{C}_t \partial_t \mat{n}} &= -\inpr[\Omega]{\bm{\psi}_n}{\bm{b}_2 \omega} + \inpr[\Omega]{\bm{\psi}_n}{\partial_s \mat{v}}, \\
    \inpr[\Omega]{{\psi}_m}{C_r \partial_t m} &= \inpr[\Omega]{{\psi}_m}{\partial_s \omega}, \\    
\end{aligned}
\end{equation}
where ${\psi}_\square, \; \bm{\psi}_\square$ are appropriate test functions. Different integration by parts lead to different boundary terms. In particular, performing integration by parts on the first two equations allows for the imposition of prescribed forces and moments. This coincides with the common procedure in computational mechanics. Using the second two equations instead, one is able to naturally impose prescribed velocities on the boundary without using Lagrange multipliers. In pH terminology, one speaks of different causalities in the final system. This methodology has been detailed in \cite{cardoso2020pfem,brugnoli2022dualfield}.

\subsection{Free-free boundary conditions} \label{sec:spatial_discretization_freefree}

If the first two lines of \eqref{eq_weak_form} are integrated by parts then one obtains
\begin{equation} \label{eq_weak_freefree}
\begin{aligned}
    \inpr[\Omega]{\bm{\psi}_v}{\rho A \partial_t \mat{v}} &= - \inpr[\Omega]{\partial_s  \bm{\psi}_v}{\mat{n}} + \dualpr[\partial \Omega]{\bm{\psi}_v}{\mat{n}}, \\
    \inpr[\Omega]{{\psi}_{\omega}}{\rho I \partial_t \omega} &= +\inpr[\Omega]{{\psi}_{\omega}}{\bm{b}_2^\top \mat{n}} - \inpr[\Omega]{\partial_s {\psi}_{\omega}}{m} + \dualpr[\partial \Omega]{{\psi}_{\omega}}{m}, \\
    \inpr[\Omega]{\bm{\psi}_n}{\bm{C}_t \partial_t \mat{n}} &=- \inpr[\Omega]{\bm{\psi}_n}{\bm{b}_2 \omega} + \inpr[\Omega]{\bm{\psi}_n}{\partial_s \mat{v}}, \\
    \inpr[\Omega]{{\psi}_m}{C_r \partial_t m} &= +\inpr[\Omega]{{\psi}_m}{\partial_s \omega}. 
\end{aligned}
\end{equation}

Lagrange polynomial ansatz functions (linear continuous Galerkin "$\mathrm{CG}_1$") are used for the linear and angular velocity and piecewise constant ansatz functions are used for the stress resultants (discontinuous Galerkin ``$\mathrm{DG}_0$")
$$
\mat{v}_h\in [\mathrm{CG}_1]^2, \; \omega_h \in \mathrm{CG}_1, \qquad \mat{n}_h \in [\mathrm{DG}_0]^2, \; m_h \in \mathrm{DG}_0.
$$
In the sense of a Bubnov-Galerkin method analogous ansatz functions for the related test functions are chosen
$$
\bm{\psi}_v\in [\mathrm{CG}_1]^2, \; \psi_\omega \in \mathrm{CG}_1, \qquad \bm{\psi}_n \in [\mathrm{DG}_0]^2, \; \psi_m \in \mathrm{DG}_0.
$$
Given $N_\mathrm{el}$ finite elements and mesh size $h=L/N_{\mathrm{el}}$, 
where $L$ is the beam length,
the finite element expansion is then written as 
\begin{equation}\label{eq:fe_expansion}
\begin{aligned}
    \mat{v}_h(s, t) &= \sum_{i=1}^{N_{\mathrm{el}}+1} \begin{bmatrix}
        \varphi_i(s) & 0 \\
        0 & \varphi_i(s)
    \end{bmatrix}
    \begin{pmatrix}
        \mathrm{v}_{x, i}(t) \\  
        \mathrm{v}_{y, i}(t) \\  
    \end{pmatrix}, \\ 
    \omega_h(s, t) &= \sum_{i=1}^{N_{\mathrm{el}}+1} \varphi_i(s) \mathrm{w}_{i}(t),
\end{aligned}
\qquad 
\begin{aligned}
    \mat{n}_h(s, t) &= \sum_{i=1}^{N_{\mathrm{el}}} \begin{bmatrix}
        \chi_i(s) & 0 \\
        0 & \chi_i(s)
    \end{bmatrix}
    \begin{pmatrix}
        \mathrm{n}_{x, i}(t) \\  
        \mathrm{n}_{y, i}(t) \\  
    \end{pmatrix}, \\ 
     m_h(s, t) &= \sum_{i=1}^{N_{\mathrm{el}}} \chi_i(s) \mathrm{m}_{i}(t),
\end{aligned}
\end{equation}
where $\varphi_i$ represent the $i-$th Lagrange function and $\chi_i(s) = \mathbbm{1}_{\{(i-1)h<s<ih\}}$ is the constant function over the element $i$. An analogously expansion is used for the test functions. Since the spaces satisfy $\partial_s \mathrm{CG}_1 \subset \mathrm{DG}_0$, cf. Fig.~\ref{fig:derivative_Lagrange},
this choice makes some relationships between the variables exact \cite{brugnoli2025nonlinear}. For instance one has the relationship 
$$C_r \partial_t m_h = \partial_s \omega_h$$
holds strongly. Since the test function $\bm{\psi}_n$ lives in a discontinuous space, the methods avoids transverse shear locking phenomena \cite{kinon2025energy}. More formally it holds
$$
\bm{C}_t \partial_t \mat{n}_h = \partial_s \mat{v}_h - \bm{b}_2 \Pi_{\mathrm{DG}_0} \omega_h,
$$
where $\Pi_{\mathrm{DG}_0}$ is the $L^2$ projection onto the $\mathrm{DG}_0$ space.

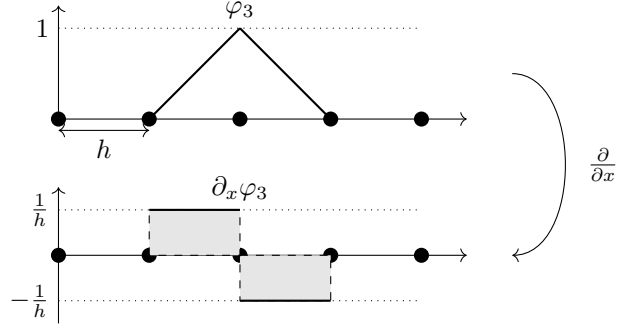
\begin{figure*}[ht]
\centering
\begin{tikzpicture}[scale=0.3]
\draw [->] (0,0) -- (0,5) ;

\draw [->] (0,0) -- (4,0) -- (8,0)  -- (12,0) -- (16,0)  -- (18,0) ;

\draw[<->] (0,-0.5) -- (4,-0.5) node[midway, below] {$h$};

\node[circle, fill=black, inner sep=2pt] at (0,0) {};
\node[circle, fill=black, inner sep=2pt] at (4,0) {};
\node[circle, fill=black, inner sep=2pt] at (8,0) {};
\node[circle, fill=black, inner sep=2pt] at (12,0) {};
\node[circle, fill=black, inner sep=2pt] at (16,0) {};

\draw[thick] (4,0) -- (8,4) node[above]{$\varphi_{3}$} ;
\draw[thick] (8,4) -- (12,0) ;

\draw (0,4) node[left]{$1$} [dotted] -- (16,4) ;

\draw [->] (0,-9) -- (0,-3) ;

\draw [->] (0,-6) -- (4,-6) -- (8,-6)  -- (12,-6) -- (16,-6) -- (18,-6) ;

\node[circle, fill=black, inner sep=2pt] at (0,-6) {};
\node[circle, fill=black, inner sep=2pt] at (4,-6) {};
\node[circle, fill=black, inner sep=2pt] at (8,-6) {};
\node[circle, fill=black, inner sep=2pt] at (12,-6) {};
\node[circle, fill=black, inner sep=2pt] at (16,-6) {};

\draw[fill=gray!20, dashed] (4,-6) -- (4,-4) -- (8,-4) -- (8,-6) -- cycle ;

\draw (0,-4) node[left]{$\frac{1}{h}$} [dotted] -- (16,-4) ;

\draw[fill=gray!20, dashed] (8,-6) -- (12,-6) -- (12,-8) -- (8,-8) -- cycle ;

\draw (0,-8) node[left]{$-\frac{1}{h}$} [dotted] -- (16,-8) ;

\draw[thick] (4,-4) -- (8,-4) node[above]{$\partial_x \varphi_3$} ;
\draw[thick] (8,-8) -- (12,-8) ;

\draw[->, bend right] (20, 2) to[out=90, in=90] (20,-6);

\node at (24,-2) {$\pdv{}{x}$};

\end{tikzpicture}
\caption{Derivative of a Lagrange space $\mathrm{CG}_1$, leading to a piecewise constant function.}
\label{fig:derivative_Lagrange}
\end{figure*}
As customary in port-Hamiltonian systems, the outputs are also included in the weak form. These coincide with nodal evaluations of the velocities at the extremities
$$
\begin{aligned}
    \bm{y}_v &= \mat{v}|_{\partial\Omega}, \\
    \bm{y}_\omega &= \omega|_{\partial\Omega}, \\
\end{aligned}
$$
Inserting the above finite element approximations into weak form \eqref{eq_weak_freefree}, and accounting for the arbitrariness of the test functions, the following set of ordinary differential equations (ODE) is obtained 
\begin{equation} \label{eq_free_free}
\begin{aligned}
    \mathrm{Diag}
\begin{bmatrix}
{\rho A}\mathbf{M}_{\otimes2} \\ 
\rho I\mathbf{M} \\
h \bm{C}_t \otimes \mathbf{I}_{N_e} \\
{C_r} h \mathbf{I}_{N_e}
\end{bmatrix}
    \odv{}{t}
    \begin{pmatrix}
        \mathbf{v} \\
        \mathbf{w} \\
        \mathbf{n} \\
        \mathbf{m}
    \end{pmatrix} &= 
    \begin{bmatrix}
        0 & 0 & -\mathbf{D}_{\otimes2}^\top & 0 \\
        0 & 0 & \mathbf{F}^\top & -\mathbf{D}^\top \\
        \mathbf{D}_{\otimes2} & -\mathbf{F} & 0 & 0 \\
        0 & \mathbf{D} & 0 & 0
    \end{bmatrix}
    \begin{pmatrix}
        \mathbf{v} \\
        \mathbf{w} \\
        \mathbf{n} \\
        \mathbf{m}
    \end{pmatrix} + \begin{bmatrix}
        \mathbf{T}_{\otimes2}^\top & 0 \\
        0 & \mathbf{T}^\top \\
        0 & 0 \\
        0 & 0 \\
    \end{bmatrix}\begin{pmatrix}
        \mathbf{u}_n \\
        \mathbf{u}_m
    \end{pmatrix}, \\
    \begin{pmatrix}
        \mathbf{y}_v \\
        \mathbf{y}_\omega
    \end{pmatrix} &= \begin{bmatrix}
        \mathbf{T}_{\otimes2} & 0 & 0 & 0 \\
        0 & \mathbf{T} & 0 & 0 \\
    \end{bmatrix} \begin{pmatrix}
        \mathbf{v} \\
        \mathbf{w} \\
        \mathbf{n} \\
        \mathbf{m}
    \end{pmatrix}.
\end{aligned}
\end{equation}
Vectors $\mathbf{v} = (\mathbf{v}_x,  \mathbf{v}_y) \in \mathbb{R}^{2N_{\mathrm{el}}+2}, \; \mathbf{w} \in \mathbb{R}^{N_{\mathrm{el}}+1}, \; \mathbf{n} = (\mathbf{n}_x, \mathbf{n}_y)\in \mathbb{R}^{2N_{\mathrm{el}}}$ and $\mathbf{m} \in \mathbb{R}^{N_{\mathrm{el}}}$ contain the finite element degrees of freedom. The notation $\mathbf{A}_{\otimes2}$ indicates that the matrix $\mathbf{A}$ is repeated $i-$times on the diagonal of the resulting matrix and can be written using the Kronecker product as $\mathbf{A}_{\otimes2} = \mathbf{I}_2 \otimes \mathbf{A}$. Furthermore, 
$\mathbf{M} \in \mathbb{R}^{(N_{\mathrm{el}}+1)\times(N_{\mathrm{el}}+1)}$ is the mass matrix obtained when using Lagrange basis and $\mathbf{D} \in \mathbb{R}^{N_{\mathrm{el}}\times(N_{\mathrm{el}}+1)}$ is an approximation matrix of the spatial derivative operator. The component-wise expression of the two is as follows
$$
[\mathbf{M}]_{ij} = \inpr[\Omega]{\varphi_i}{\varphi_j}, \qquad [\mathbf{D}]_{ij} = \inpr[\Omega]{\chi_i}{\partial_s \varphi_j}.
$$
Matrix $\mathbf{F}\in\mathbb{R}^{2N_{\mathrm{el}}\times (N_{\mathrm{el}}+1)}$ is given by
$$
\mathbf{F} = \begin{bmatrix}
   0  \\ \mathbf{\Pi}
\end{bmatrix}, \qquad [\mathbf{\Pi}]_{ij} = \inpr[\Omega]{\chi_i}{\varphi_j}.
$$
Finally $\mathbf{T}\in \mathbb{R}^{2\times N_{\mathrm{el}}+1}$ is a localization matrix that picks the degrees of freedom at the boundary
$$
\mathbf{T} = \begin{bmatrix}
    \mathbf{b}_1^\top \\
    \mathbf{b}_{N_{\mathrm{el}}+1}^\top
\end{bmatrix}, \qquad \mathbf{b}_1 = \begin{pmatrix}
    1 \\ 0 \\ 0 \\ \vdots \\ 0
\end{pmatrix}, \qquad \mathbf{b}_{N_{\mathrm{el}}+1} = \begin{pmatrix}
    0 \\ \vdots \\ 0 \\ 0 \\  1
\end{pmatrix}.
$$
The inputs of the system are the forces and bending moments at the extremities:
$$
\mathbf{u}_n = \begin{pmatrix}
    -\matcomp{n}_x(0, t) \\
    +\matcomp{n}_x(L, t) \\
    -\matcomp{n}_y(0, t) \\
    +\matcomp{n}_y(L, t) \\
\end{pmatrix}, \qquad \mathbf{u}_m = \begin{pmatrix}
    -{m}(0, t) \\
    +{m}(L, t)
\end{pmatrix}, 
$$
whereas the outputs are the linear and angular velocities
$$
\mathbf{y}_v = \begin{pmatrix}
    \matcomp{v}_x(0, t) \\
    \matcomp{v}_x(L, t) \\
    \matcomp{v}_y(0, t) \\
    \matcomp{v}_y(L, t) \\
\end{pmatrix}, \qquad \mathbf{y}_{\omega} = \begin{pmatrix}
    {\omega}(0, t) \\
    {\omega}(L, t)
\end{pmatrix}, 
$$
The system \eqref{eq_free_free} is written more compactly as 
$$
\begin{aligned}
\mathbf{M}_{F}^\mathbf{e}\dot{\mathbf{e}}_{F} &= \mathbf{J}_{F}^\mathbf{e}\mathbf{e}_{F} + \mathbf{B}_{F}^\mathbf{e}\mathbf{u}_{F}, \\
    \mathbf{y}_{F} &= (\mathbf{B}_{F}^\mathbf{e})^\top \mathbf{e}_{F},
\end{aligned}
$$
with positive definite $\mathbf{M}_{F} = \mathbf{M}_{F}^\top \succ 0$ and $\mathbf{J}_{F}^{\mathbf{e}} = - (\mathbf{J}_{F}^{\mathbf{e}})^\top$. 
The superscript of the structure matrix $\mathbf{J}_{F}^{\mathbf{e}}$ refers to the fact that the matrices are related to coenergy part $\mathbf{e}$. Later on, the kinematic relationships between the coordinates and the velocities are also included. Then the overall dynamics is still governed by an overall skew-symmetric matrix $\mathbf{J}$. The power conjugated outputs (the linear and angular velocities) have been added to interconnect systems with different boundary conditions. Note that the port matrix here and also in the following sections is constant. The resulting system can be represented via a block diagram as in Fig. \ref{fig:beam_free}.

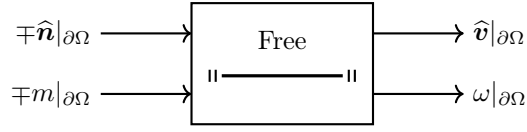
\begin{figure}[htb]
    \centering
    \begin{tikzpicture}[scale=0.8]
    \draw[thick] (0,0) rectangle (3,2);
    
    \node at (1.5,1.4) {Free};
    
    \draw[thick, ->] (-1.5,1.5) -- (0,1.5);
    \node[left] at (-1.5,1.5) {$\mp\mat{n}|_{\partial \Omega}$};
    
    \draw[thick, ->] (-1.5,0.5) -- (0,0.5);
    \node[left] at (-1.5,0.5) {$\mp m|_{\partial \Omega}$};
    
    \draw[thick, ->] (3,1.5) -- (4.5,1.5);
    \node[right] at (4.5,1.5) {$\mat{v}|_{\partial \Omega}$};
    
    \draw[thick, ->] (3,0.5) -- (4.5,0.5);
    \node[right] at (4.5,0.5) {$\omega|_{\partial \Omega}$};

    \draw[very thick] (0.5,0.8) -- (2.5,0.8);
    
    \draw[thick] (0.4,0.7) -- (0.4,0.9);
    \draw[thick] (0.3,0.7) -- (0.3,0.9);
    \draw[thick] (2.6,0.7) -- (2.6,0.9);
    \draw[thick] (2.7,0.7) -- (2.7,0.9);

\end{tikzpicture}
    \caption{Block diagram for the free-free case.}
    \label{fig:beam_free}
\end{figure}

Let us now proceed analogously with a different set of boundary conditions: The case where both extremities of the beam are clamped.

\subsection{Clamped-Clamped boundary conditions} \label{sec:spatial_discretization_clampedclamped}

If the last two lines of \eqref{eq_weak_form} are integrated by parts then one obtains
\begin{equation}
\begin{aligned}
   \inpr[\Omega]{\bm{\psi}_v}{\rho A \partial_t \mat{v}} &= +\inpr[\Omega]{\bm{\psi}_v}{\partial_s \mat{n}}, \\
    \inpr[\Omega]{{\psi}_{\omega}}{\rho I \partial_t \omega} &= +\inpr[\Omega]{{\psi}_{\omega}}{\bm{b}_2^\top \mat{n}} + \inpr[\Omega]{{\psi}_{\omega}}{\partial_s {m}}, \\
    \inpr[\Omega]{\bm{\psi}_n}{\bm{C}_t \partial_t \mat{n}} &= -\inpr[\Omega]{\bm{\psi}_n}{\bm{b}_2 {\omega}} - \inpr[\Omega]{\partial_s \bm{\psi}_n}{\mat{v}} + \dualpr[\partial \Omega]{\bm{\psi}_n}{\mat{v}}, \\
    \inpr[\Omega]{{\psi}_m}{{C}_r \partial_t m} &= - \inpr[\Omega]{\partial_s{\psi}_m}{\omega} + \dualpr[\partial \Omega]{{\psi}_m}{\omega}. \\    
\end{aligned}
\end{equation}
\begin{figure}[htb]
    \centering
    \begin{tikzpicture}
    \draw[thick] (0,0) rectangle (3,2);
    
    \node at (1.5,1.4) {Clamped};
    
    \draw[thick, ->] (-1.5,1.5) -- (0,1.5);
    \node[left] at (-1.5,1.5) {$\mat{v}|_{\partial \Omega}$};
    
    \draw[thick, ->] (-1.5,0.5) -- (0,0.5);
    \node[left] at (-1.5,0.5) {${\omega}|_{\partial \Omega}$};
    
    \draw[thick, ->] (3,1.5) -- (4.5,1.5);
    \node[right] at (4.5,1.5) {$\mp\mat{n}|_{\partial \Omega}$};
    
    \draw[thick, ->] (3,0.5) -- (4.5,0.5);
    \node[right] at (4.5,0.5) {$\mp{m}|_{\partial \Omega}$};

    \draw[very thick] (0.5,0.6) -- (2.5,0.6);
    
    \draw[very thick] (0.5,0.8) -- (0.5,0.4);
    \fill[pattern=north east lines] (0.3,0.8) rectangle (0.5,0.4);
    \draw[thick] (0.3,0.8) -- (0.3,0.4);
    
    \draw[very thick] (2.5,0.8) -- (2.5,0.4);
    \fill[pattern=north east lines] (2.5,0.8) rectangle (2.7,0.4);
    \draw[thick] (2.7,0.8) -- (2.7,0.4);

\end{tikzpicture}
    \caption{Block diagram for the clamped-clamped case.}
    \label{fig:beam_clamped}
\end{figure}
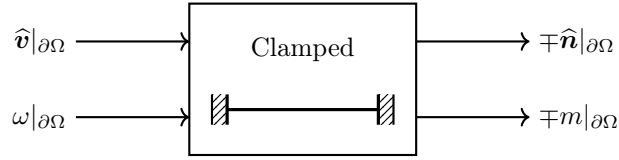
For this formulation continuous Lagrange ansatz functions of order one are used for the stress resultants and piecewise constant ansatz functions are used for the linear and angular velocity, i.e,
$$
\mat{n}_h \in [\mathrm{CG}_1]^2, \; {m}_h \in \mathrm{CG}_1, \qquad \mat{v}_h \in [\mathrm{DG}_0]^2, \; {\omega}_h \in \mathrm{DG}_0.
$$
Introducing the finite element approximations to the weak form, the following ODE is obtained, cf. Fig. \ref{fig:beam_clamped}
\begin{equation}\label{eq:discrete_clamped}
\begin{aligned}
\mathrm{Diag}
\begin{bmatrix}
{\rho A} h\mathbf{I}_{2N_e} \\ 
\rho I h\mathbf{I}_{N_e} \\
\bm{C}_t \otimes \mathbf{M} \\
{C_r} \mathbf{M}
\end{bmatrix}
    \odv{}{t}
    \begin{pmatrix}
        \mathbf{v} \\
        \mathbf{w} \\
        \mathbf{n} \\
        \mathbf{m}
    \end{pmatrix} &= 
    \begin{bmatrix}
        0 & 0 & \mathbf{D}_{\otimes2} & 0 \\
        0 & 0 & \mathbf{F}^\top & \mathbf{D} \\
        -\mathbf{D}_{\otimes2}^\top & -\mathbf{F} & 0 & 0 \\
        0 & -\mathbf{D}^\top & 0 & 0
    \end{bmatrix}
    \begin{pmatrix}
        \mathbf{v} \\
        \mathbf{w} \\
        \mathbf{n} \\
        \mathbf{m}
    \end{pmatrix} + \begin{bmatrix}
        0 & 0 \\
        0 & 0 \\
         \mathbf{T}_{\nu, \otimes2}^\top & 0 \\
        0 & \mathbf{T}_{\nu}^\top \\
    \end{bmatrix}\begin{pmatrix}
        \mathbf{u}_v \\
        \mathbf{u}_{\omega}
    \end{pmatrix}, \\
\begin{pmatrix}
        \mathbf{y}_n \\
        \mathbf{y}_{m}
    \end{pmatrix} &=
    \begin{bmatrix}
        0 & 0 & \mathbf{T}_{\nu, \otimes2} & 0 \\
        0 & 0 & 0 & \mathbf{T}_{\nu} \\
    \end{bmatrix} \begin{pmatrix}
        \mathbf{v} \\
        \mathbf{w} \\
        \mathbf{n} \\
        \mathbf{m}
    \end{pmatrix},
\end{aligned}
\end{equation}
where $\mathbf{T}_{\nu} \in \mathbb{R}^{2\times (N_{\mathrm{el}}+1)}$ is the normal trace matrix, taking values $-1$ or $1$ for the left and right extremity degree of freedom
$$
\mathbf{T}_{\nu} = \begin{bmatrix}
    -\mathbf{b}_1^\top \\
    \mathbf{b}_{N_{\mathrm{el}}+1}^\top
\end{bmatrix}.
$$
Matrix $\mathbf{F}\in\mathbb{R}^{2(N_{\mathrm{el}}+1) \times N_{\mathrm{el}}}$ is given by
$$
\mathbf{F} = \begin{bmatrix}
    0  \\ \mathbf{\Pi}^\top
\end{bmatrix}.
$$
Matrix $\mathbf{F}$ takes a different expression with respect to the previous example as the finite element bases have changed. 
Contrarily to the previous section, here,
the input to the systems are the linear and angular velocities at the extremities,
$$
\mathbf{u}_v = \begin{pmatrix}
    \matcomp{v}_x(0, t) \\
    \matcomp{v}_x(L, t) \\
    \matcomp{v}_y(0, t) \\
    \matcomp{v}_y(L, t) \\
\end{pmatrix}, \qquad 
\mathbf{u}_\omega = \begin{pmatrix}
    \omega(0, t) \\
    \omega(L, t)
\end{pmatrix}.
$$ 
whereas the outputs represent the forces and momenta
$$
\mathbf{y}_n = \begin{pmatrix}
    -\matcomp{n}_x(0, t) \\
    +\matcomp{n}_x(L, t) \\
    -\matcomp{n}_y(0, t) \\
    +\matcomp{n}_y(L, t) \\
\end{pmatrix}, \qquad 
\mathbf{y}_m = \begin{pmatrix}
    m(0, t) \\
    m(L, t)
\end{pmatrix}.
$$ 
As before,
the system \eqref{eq:discrete_clamped} can be written in a more compact form given by
$$
\begin{aligned}
\mathbf{M}_{C}^\mathbf{e}\dot{\mathbf{e}}_{C} &= \mathbf{J}_{C}^\mathbf{e}\mathbf{e}_{C} + \mathbf{B}_{C}^\mathbf{e}\mathbf{u}_{C}, \\
    \mathbf{y}_{C} &= (\mathbf{B}_{C}^\mathbf{e})^\top \mathbf{e}_{C}.
\end{aligned}
$$
In the following, we treat forces differently than moments and linear velocities differently than angular velocities. This leads to two further causality cases.

\subsection{Pinned case} \label{sec:spatial_discretization_pinnedpinned}

If the second and third line of \eqref{eq_weak_form} are integrated by parts then one obtains
\begin{equation}
\begin{aligned}
    \inpr[\Omega]{\bm{\psi}_v}{\rho A \partial_t \mat{v}} &= +\inpr[\Omega]{\bm{\psi}_v}{\partial_s \mat{n}}, \\
    \inpr[\Omega]{{\psi}_{\omega}}{\rho I \partial_t {\omega}} &= +\inpr[\Omega]{{\psi}_{\omega}}{\bm{b}_2^\top\mat{n}} -\inpr[\Omega]{\partial_s {\psi}_{\omega}}{m} + \dualpr[\partial\Omega]{{\psi}_{\omega}}{m}, \\
    \inpr[\Omega]{\bm{\psi}_n}{\bm{C}_t \partial_t \mat{n}} &= -\inpr[\Omega]{\bm{\psi}_n}{\bm{b}_2\omega} - \inpr[\Omega]{\partial_s \bm{\psi}_n}{\mat{v}} + \dualpr[\partial \Omega]{\bm{\psi}_n }{\mat{v}}, \\
    \inpr[\Omega]{{\psi}_m}{\bm{C}_r \partial_t m} &= +\inpr[\Omega]{{\psi}_m}{\partial_s \omega}. \\    
\end{aligned}
\end{equation}
For this formulation we consider
$$
{\omega} \in \mathrm{CG}_1,\; \mat{n} \in [\mathrm{CG}_1]^2, \qquad \mat{v}\in [\mathrm{DG}_0]^2, \; {m} \in \mathrm{DG}_0.
$$
and the following ODE is obtained, cf. Fig. \ref{fig:beam_pinned},
\begin{equation}\label{eq:discrete_pinned}
\begin{aligned}
    \mathrm{Diag}
\begin{bmatrix}
{\rho A} h\mathbf{I}_{2N_e} \\ 
\rho I \mathbf{M} \\
\bm{C}_t \otimes \mathbf{M} \\
{C_r} h \mathbf{I}_{N_e}
\end{bmatrix}
    \odv{}{t}
    \begin{pmatrix}
        \mathbf{v} \\
        \mathbf{w} \\
        \mathbf{n} \\
        \mathbf{m}
    \end{pmatrix} &= 
    \begin{bmatrix}
        0 & 0 & \mathbf{D}_{\otimes2} & 0 \\
        0 & 0 & \mathbf{F}^\top & -\mathbf{D}^\top \\
        -\mathbf{D}_{\otimes2}^\top & -\mathbf{F} & 0 & 0 \\
        0 & \mathbf{D} & 0 & 0
    \end{bmatrix}
    \begin{pmatrix}
        \mathbf{v} \\
        \mathbf{w} \\
        \mathbf{n} \\
        \mathbf{m}
    \end{pmatrix} + \begin{bmatrix}
        0 & 0 \\
         \mathbf{T}^\top & 0 \\
        0 & \mathbf{T}_{\nu, \otimes2}^\top \\
        0 & 0 \\
    \end{bmatrix}\begin{pmatrix}
        \mathbf{u}_{m} \\
        \mathbf{u}_v \\
    \end{pmatrix}, \\
    \begin{pmatrix}
        \mathbf{y}_{\omega} \\
        \mathbf{y}_n \\
    \end{pmatrix} &= \begin{bmatrix}
        0 & \mathbf{T} & 0 & 0 \\
        0 & 0 & \mathbf{T}_{\nu, \otimes2} & 0 \\
    \end{bmatrix}\begin{pmatrix}
        \mathbf{v} \\
        \mathbf{w} \\
        \mathbf{n} \\
        \mathbf{m}
    \end{pmatrix}.
\end{aligned}
\end{equation}
Matrix $\mathbf{F}\in\mathbb{R}^{2(N_{\mathrm{el}}+1) \times (N_{\mathrm{el}}+1)}$ takes a different expression with respect to the previous examples as the finite element bases have changed
$$
\mathbf{F} = \begin{bmatrix}
    0  \\ \mathbf{M}
\end{bmatrix}.
$$
The system can be written in a more compact form 
$$
\begin{aligned}
\mathbf{M}_{P}^\mathbf{e}\dot{\mathbf{e}}_{P} &= \mathbf{J}_{P}^\mathbf{e}\mathbf{e}_{P} + \mathbf{B}_{P}^\mathbf{e}\mathbf{u}_{P}, \\
    \mathbf{y}_{P} &= (\mathbf{B}_{P}^\mathbf{e})^\top \mathbf{e}_{P}.
\end{aligned}
$$

\begin{figure}[htb]
    \centering
    \begin{tikzpicture}
    \draw[thick] (0,0) rectangle (3,2);
    
    \node at (1.5,1.4) {Pinned};
    
    \draw[thick, ->] (-1.5,1.5) -- (0,1.5);
    \node[left] at (-1.5,1.5) {$\pm{m}|_{\partial \Omega}$};
    
    \draw[thick, ->] (-1.5,0.5) -- (0,0.5);
    \node[left] at (-1.5,0.5) {$\mat{v}|_{\partial \Omega}$};
    
    \draw[thick, ->] (3,1.5) -- (4.5,1.5);
    \node[right] at (4.5,1.5) {${\omega}|_{\partial \Omega}$};
    
    \draw[thick, ->] (3,0.5) -- (4.5,0.5);
    \node[right] at (4.5,0.5) {$\pm\mat{n}|_{\partial \Omega}$};

    \draw[very thick] (0.5,0.7) -- (2.5,0.7);
    
    \draw[very thick] (0.5,0.7) -- (0.3,0.4) -- (0.7,0.4) -- cycle;
    \fill[pattern=north east lines] (0.3,0.4) rectangle (0.7,0.3);
    \draw[thick] (0.5,0.7) circle (0.07);
    
    \draw[very thick] (2.5,0.7) -- (2.3,0.4) -- (2.7,0.4) -- cycle;
    \fill[pattern=north east lines] (2.3,0.4) rectangle (2.7,0.3);
    \draw[thick] (2.5,0.7) circle (0.07);

\end{tikzpicture}
    \caption{Block diagram for the pinned case.}
    \label{fig:beam_pinned}
\end{figure}
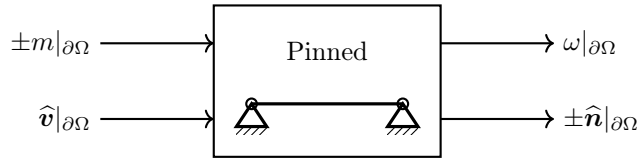

\subsection{Guided} \label{sec:spatial_discretization_guided}

If the first and fourth line of \eqref{eq_weak_form} are integrated by parts then one obtains
\begin{equation}
\begin{aligned}
    \inpr[\Omega]{\bm{\psi}_v}{\rho A \partial_t \mat{v}} &= - \inpr[\Omega]{\partial_s \bm{\psi}_v}{\mat{n}} + \dualpr[\partial\Omega]{\bm{\psi}_v}{\mat{n}}, \\
    \inpr[\Omega]{{\psi}_{\omega}}{\rho I \partial_t \omega} &= +\inpr[\Omega]{{\psi}_{\omega}}{\bm{b}_2^\top \mat{n}} +\inpr[\Omega]{\partial_s {\psi}_{\omega}}{m}, \\
    \inpr[\Omega]{\bm{\psi}_n}{\bm{C}_t \partial_t \mat{n}} &= -\inpr[\Omega]{\bm{\psi}_n}{\bm{b}_2 \omega} + \inpr[\Omega]{\partial_s \bm{\psi}_n}{\mat{v}}, \\
    \inpr[\Omega]{{\psi}_m}{\bm{C}_r \partial_t {m}} &= -\inpr[\Omega]{\partial_s {\psi}_m}{ \omega} + \dualpr[\partial\Omega]{{\psi}_m  }{ \omega}. \\    
\end{aligned}
\end{equation}
For this formulation we consider
$$
\mat{v}\in [\mathrm{CG}_1]^2, \; {m} \in \mathrm{CG}_1, \qquad {\omega} \in \mathrm{DG}_0,\; \mat{n} \in [\mathrm{DG}_0]^2.
$$
such that the following ODE is obtained 
\begin{equation}\label{eq:discrete_guided}
\begin{aligned}
    \mathrm{Diag}
\begin{bmatrix}
{\rho A} h\mathbf{I}_{2N_e} \\ 
\rho I \mathbf{M} \\
\bm{C}_t \otimes \mathbf{M} \\
{C_r} h \mathbf{I}_{N_e}
\end{bmatrix}
    \odv{}{t}
    \begin{pmatrix}
        \mathbf{v} \\
        \mathbf{w} \\
        \mathbf{n} \\
        \mathbf{m}
    \end{pmatrix} &= 
    \begin{bmatrix}
        0 & 0 & -\mathbf{D}_{\otimes2}^\top & 0 \\
        0 & 0 & \mathbf{F}^\top & \mathbf{D} \\
        \mathbf{D}_{\otimes2} & -\mathbf{F} & 0 & 0 \\
        0 & -\mathbf{D}^\top & 0 & 0
    \end{bmatrix}
    \begin{pmatrix}
        \mathbf{v} \\
        \mathbf{w} \\
        \mathbf{n} \\
        \mathbf{m}
    \end{pmatrix} + \begin{bmatrix}
        \mathbf{T}^\top & 0 \\
        0 & 0 \\
        0 & 0 \\
        0 & \mathbf{T}_{\nu, \otimes2}^\top \\
    \end{bmatrix}\begin{pmatrix}
        \mathbf{u}_{n} \\
        \mathbf{u}_\omega \\
    \end{pmatrix}, \\
    \begin{pmatrix}
        \mathbf{y}_v \\
        \mathbf{y}_m \\
    \end{pmatrix} &=
    \begin{bmatrix}
        \mathbf{T} & 0  & 0 & 0\\
        0 & 0 & 0 & \mathbf{T}_{\nu, \otimes2} \\
    \end{bmatrix}
    \begin{pmatrix}
        \mathbf{v} \\
        \mathbf{w} \\
        \mathbf{n} \\
        \mathbf{m}
    \end{pmatrix}.
\end{aligned}
\end{equation}
Matrix $\mathbf{F}\in\mathbb{R}^{2N_{\mathrm{el}} \times N_{\mathrm{el}}}$ is given by
$$
\mathbf{F} = \begin{bmatrix}
    0  \\ h\mathbf{I}_{N_e}
\end{bmatrix}.
$$
The system can be written compactly as
$$
\begin{aligned}
\mathbf{M}_{G}^\mathbf{e}\dot{\mathbf{e}}_{G} &= \mathbf{J}_{G}^\mathbf{e}\mathbf{e}_{G} + \mathbf{B}_{G}^\mathbf{e}\mathbf{u}_{G}, \\
    \mathbf{y}_{G} &= (\mathbf{B}_{G}^\mathbf{e})^\top \mathbf{e}_{G}.
\end{aligned}
$$

\begin{figure}[htb]
    \centering
    \begin{tikzpicture}
    \draw[thick] (0,0) rectangle (3,2);
    
    \node at (1.5,1.4) {Guided};
    
    \draw[thick, ->] (-1.5,1.5) -- (0,1.5);
    \node[left] at (-1.5,1.5) {$\mat{n}|_{\partial \Omega}$};
    
    \draw[thick, ->] (-1.5,0.5) -- (0,0.5);
    \node[left] at (-1.5,0.5) {${\omega}|_{\partial \Omega}$};
    
    \draw[thick, ->] (3,1.5) -- (4.5,1.5);
    \node[right] at (4.5,1.5) {$\mat{v}|_{\partial \Omega}$};
    
    \draw[thick, ->] (3,0.5) -- (4.5,0.5);
    \node[right] at (4.5,0.5) {$m|_{\partial \Omega}$};

    \draw[very thick] (0.5,0.7) -- (2.5,0.7);
    
    \draw[very thick] (0.3,0.6) -- (0.7,0.6);
    \fill[pattern=north east lines] (0.3,0.6) rectangle (0.7,0.5);
    \draw[very thick] (0.3,0.8) -- (0.7,0.8);
    \fill[pattern=north east lines] (0.3,0.8) rectangle (0.7,0.9);

    \draw[very thick] (2.3,0.6) -- (2.7,0.6);
    \fill[pattern=north east lines] (2.3,0.6) rectangle (2.7,0.5);
    \draw[very thick] (2.3,0.8) -- (2.7,0.8);
    \fill[pattern=north east lines] (2.3,0.8) rectangle (2.7,0.9);

\end{tikzpicture}
    \caption{Block diagram for the guided case.}
    \label{fig:beam_guided}
\end{figure}
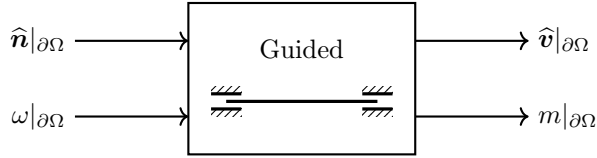

\subsection{Extension to the nonlinear case}\label{sec:nonlinear_formulation}

As we have seen in the previous sections, the discretization of the linear part leads to a system of the form 
\begin{equation}
    \begin{aligned}
        \mathbf{M}^\mathbf{e}_{\square}\dot{\mathbf{e}} &= \mathbf{J}^{\mathbf{e}}_{\square}\mathbf{e}_\square + \mathbf{B}^\mathbf{e}_\square\mathbf{u}_\square, \\
        \mathbf{y} &= (\mathbf{B}_\square^\mathbf{e})^\top \mathbf{e}_\square,
    \end{aligned}
\end{equation}
where $\square$ indicates the boundary conditions for the model, i.e. $\square = \{F,\; C,\; P,\; G\}$.
In the following, due to the structure of the material beam formulation \eqref{eq_ph_inter_m}, nonlinear terms are simply added to the weak formulation. For instance, if one considers a free-free boundary causality \eqref{eq_weak_freefree}, the linear part of the variational formulation reads
\begin{equation*}
\begin{aligned}
    \inpr[\Omega]{\bm{\psi}_v}{\rho A \partial_t \mat{v}} =& - \inpr[\Omega]{\partial_s  \bm{\psi}_v}{\mat{n}}  + \dualpr[\partial \Omega]{\bm{\psi}_v}{\mat{n}}, \\
    \inpr[\Omega]{{\psi}_{\omega}}{\rho I \partial_t {\omega}} =& - \inpr[\Omega]{\partial_s {\psi}_{\omega}}{m} + \inpr[\Omega]{{\psi}_{\omega}}{\bm{b}_1^\top \bm{S}\mat{n}}+\dualpr[\partial \Omega]{{\psi}_{\omega}}{m}, \\
    \inpr[\Omega]{\bm{\psi}_n}{\bm{C}_t \partial_t \mat{n}} =& + \inpr[\Omega]{\bm{\psi}_n}{\partial_s \mat{v}} + \inpr[\Omega]{\bm{\psi}_n}{\bm{S}\bm{b}_1 \omega}, \\
    \inpr[\Omega]{{\psi}_m}{C_r \partial_t m} =& +\inpr[\Omega]{{\psi}_m}{\partial_s \omega}.
\end{aligned}
\end{equation*}
For the nonlinear weak formulation, since the nonlinearity is just algebraic it does not require a specified treatment and it is simply projected onto the appropriate finite element space. The thus completed weak form is given by
\begin{equation}
\begin{aligned}
    \inpr[\Omega]{\bm{\psi}_v}{\rho A \partial_t \mat{v}} =& - \inpr[\Omega]{\partial_s  \bm{\psi}_v}{\mat{n}} + \inpr[\Omega]{\bm{\psi}_v}{\bm{S}\mat{\pi}_v\bm{\omega}}  +  \inpr[\Omega]{\bm{\psi}_v}{\bm{S}^\top \kappa \mat{n}} + \dualpr[\partial \Omega]{\bm{\psi}_v}{\mat{n}}, \\
    \inpr[\Omega]{{\psi}_{\omega}}{\rho I \partial_t {\omega}} =& - \inpr[\Omega]{\partial_s {\psi}_{\omega}}{m} + \inpr[\Omega]{{\psi}_{\omega}}{\mat{\pi}_v^\top \bm{S}\mat{v}} + \inpr[\Omega]{{\psi}_{\omega}}{(\mat{\gamma} + \bm{b}_1)^\top \bm{S}\mat{n}}+\dualpr[\partial \Omega]{{\psi}_{\omega}}{m}, \\
    \inpr[\Omega]{\bm{\psi}_n}{\bm{C}_t \partial_t \mat{n}} =& + \inpr[\Omega]{\bm{\psi}_n}{\partial_s \mat{v}} + \inpr[\Omega]{\bm{\psi}_n}{\bm{S}^\top \kappa \mat{v}} + \inpr[\Omega]{\bm{\psi}_n}{\bm{S}(\mat{\gamma}+\bm{b}_1 )\omega}, \\
    \inpr[\Omega]{{\psi}_m}{C_r \partial_t m} =& +\inpr[\Omega]{{\psi}_m}{\partial_s \omega}.
\end{aligned}
\end{equation}
Once a Galerkin finite element formulation is applied, a final system of the following form is obtained
\begin{equation}
    \begin{aligned}
        \mathbf{M}^\mathbf{e}_{\square}\dot{\mathbf{e}} &= \mathbf{J}^{\mathbf{e}}_{\square}(\mathbf{e}_{\square})\mathbf{e}_\square + \mathbf{B}^\mathbf{e}_\square\mathbf{u}_\square, \\
        \mathbf{y} &= (\mathbf{B}_\square^\mathbf{e})^\top \mathbf{e}_\square.
    \end{aligned}
\end{equation}
Therein, the matrix $\mathbf{J}^{\mathbf{e}}_{\square}(\mathbf{e}_{\square}) = \mathbf{J}^{\mathbf{e}}_{\square} + \mathbf{J}^{\mathbf{e}}_{\square, 1}(\mathbf{e}_{\square})$ 
is decomposed into a constant part $\mathbf{J}^{\mathbf{e}}_{\square}$ and a part $\mathbf{J}^{\mathbf{e}}_{\square, 1}$ that depends linearly on the state. 
The constant part is obtained according to the finite element discretization explained in Secs.~\ref{sec:spatial_discretization_freefree}-\ref{sec:spatial_discretization_guided}. The part that depends linearly on the state is due to the geometric nonlinearity. After Galerkin projection the nonlinear part depends on the considered formulation as different finite element spaces are used in each.

\subsection{Gravity effects \rewone{and nodal loading}}\label{sec:gravity_forcing}

\paragraph{Gravity incorporation}
The incorporation of gravity is achieved by considering the projection of the associated terms on the finite element basis. An important point is the preservation of the skew-symmetry of the interconnection operator. To achieve so, the centerline position and the cross-section angle are discretized using the same finite elements as the material velocity and angular velocity, respectively. Furthermore, the gravity acceleration is also discretized using the same finite element basis as the centerline position. 
Omitting the symbol $\square$, the dynamics incorporating the gravity effect is given by
\begin{equation}\label{eq:discrete_nonlinear}
\begin{aligned}
    \begin{bmatrix}
    \mathbf{M}^{\mathbf{q}} & 0 \\
    0 & \mathbf{M}^{\mathbf{e}}
\end{bmatrix} \odv{}{t}  
\begin{pmatrix}
    \mathbf{q} \\
    \mathbf{e}
\end{pmatrix} &= 
\begin{bmatrix}
    0 & \mathbf{P}(\bm{\theta}) \\
    -\mathbf{P}^\top(\bm{\theta}) & \mathbf{J}^{\mathbf{e}}(\mathbf{e})
\end{bmatrix}
\begin{pmatrix}
    \mathbf{z}_{\mathbf{q}} \\
    \mathbf{e}
\end{pmatrix} + 
\begin{bmatrix}
    0 \\
    \mathbf{B}^\mathbf{e}
\end{bmatrix}\mathbf{u}, \\
\mathbf{y} &= \begin{bmatrix}
    0 & (\mathbf{B}^\mathbf{e})^\top
\end{bmatrix}
\begin{pmatrix}
    \mathbf{z}_{\mathbf{q}} \\
    \mathbf{e}
\end{pmatrix},
\end{aligned}
\end{equation}
where $\mathbf{q}=(\mathbf{r}_x, \mathbf{r}_y, \bm{\theta})$ and the gravity force is given by the vector
\begin{equation}\label{eq:z_vector}
\mathbf{z}_{\mathbf{q}} := \rho A g\begin{pmatrix}
    0 \\ 
    \mathbf{1} \\
    0
\end{pmatrix}, \qquad \text{where} \quad \mathbf{1} := \begin{pmatrix}
    1 \\ \vdots \\ 1
\end{pmatrix}.    
\end{equation}

For details on the construction of matrix $\mathbf{P}(\theta)$ the reader can consult~\ref{sec:p_matrix}. The system \eqref{eq:discrete_nonlinear} can then be written as a port-Hamiltonian descriptor system \cite{morandin2019descriptor}
$$
\begin{aligned}
\mathbf{E}\dot{\mathbf{x}} &= \mathbf{J}(\mathbf{x}) \mathbf{z} + \mathbf{B} \mathbf{u}, \\
\mathbf{y} &= \mathbf{B}^\top \mathbf{z},
\end{aligned}
$$
with $\mathbf{E}=\mathbf{E}^\top \succ 0, \: \mathbf{J}(\mathbf{x}) = -\mathbf{J}^\top$. All matrices and vectors are defined by comparison with \eqref{eq:discrete_nonlinear}. The gradient of the Hamiltonian $H(\mathbf{x}) = \frac{1}{2}\mathbf{e}^\top \mathbf{M}^{\mathbf{e}}\mathbf{e} + U(\mathbf{q})$ satisfies
\begin{equation}\label{eq:ham_approximation}
\mathbf{E}^\top \mathbf{z} = \nabla_{\mathbf{x}} H.
\end{equation}
Therein, the gravity potential introduced in \eqref{eq:gravity_potential} is given by $U(\mathbf{q}) = \int_{\Omega} \rho A g {r}_y \d{s}$, where $g$ is the gravity acceleration and $r_y$ is the vertical position of the centerline. Consider the finite element expansion of the centerline position 
$$\bm{r}_h(s,t) = \begin{bmatrix}
\bm{\xi}_r(s) & 0 \\
0 &  \bm{\xi}_r(s)
\end{bmatrix}
\begin{pmatrix}
    \mathbf{r}_x(t) \\
    \mathbf{r}_y(t)
\end{pmatrix},
$$
where $\bm{\xi}_r$ is a row vector collecting the finite element basis for the components of $\bm{r}_h$. Depending on the formulation, it is  given by $\bm{\xi}_r = [\varphi_1 \dots \varphi_{N_e+1}] \in \mathbb{R}^{1\times N_e+1}$ or $\bm{\xi}_r = [\chi_1 \dots \chi_{N_e}]\in \mathbb{R}^{1\times N_e}$, where $\phi_i$ and $\chi_i$ have been introduced in Eq. \eqref{eq:fe_expansion}. Plugging this expansion in the gravity potential gives $U(\mathbf{r}_y) = \rho A g \mathbf{a}^\top \mathbf{r}_y$, where $\mathbf{a}=\int_{\Omega} \bm{\xi}_r^\top \d{s}$. In particular the rows related to $\mathbf{r}_y$ in Eq. \eqref{eq:ham_approximation} give the identity
\begin{equation}\label{eq:gradient_potential}
    \mathbf{M}^{\mathbf{r}_y} \mathbf{1} = \mathbf{a}, \qquad \text{where} \quad \mathbf{M}^{\mathbf{r}_y}=\int_{\Omega} \bm{\xi}_r^\top \bm{\xi}_r \d{s}.
\end{equation}
Note that we have taken into account the finite element basis in the latter expression. Although not recommended from a numerical standpoint, the system can in principle be written in canonical port-Hamiltonian form \cite{vanderschaft2014overview} as
$$
\begin{aligned}
    \dot{\mathbf{x}} &= \tilde{\mathbf{J}}(\mathbf{x}) \nabla_{\mathbf{x}} H + \tilde{\mathbf{B}} \mathbf{u}, \\
\mathbf{y} &= \tilde{\mathbf{B}}^\top \nabla_{\mathbf{x}} H.
\end{aligned}
$$
in which $\tilde{\mathbf{J}}(\mathbf{x}) := \mathbf{E}^{-1}\mathbf{J}(\mathbf{x}) \mathbf{E}^{-\top}$ and $\tilde{\mathbf{B}} := \mathbf{E}^{-1}\mathbf{B}$. This is due to the fact that $\mathbf{E} \succ 0$ can be inverted.

\rewone{
\paragraph{Nodal loading}
The proposed method relies on a variational formulation, so arbitrary point loads can be incorporated into the weak form as follows. Consider a nodal load written in the material formulation
$$
\bm{f}_{\rm nodal}(s, t) = \bm{f}(t)\delta(s-s_p)
$$
Consider a finite element discretization of the velocity test function
$$
\bm{\psi}_v(s) \in \mathrm{range}\{\bm\Xi(s)\},  \qquad \bm\Xi(s)=\begin{bmatrix}
\bm\xi_v & 0 \\
0 & \bm\xi_v 
\end{bmatrix},
$$
where $\bm\xi_i$ can be either a Lagrange or a piecewise constant finite element basis (depending on the mixed formulation used). The forcing term in the weak formulation takes the form 
$$
({\bm{\psi}}_v, \, \bm{f}_{\rm nodal})_{\Omega} = \bm\Xi(s_p)^\top \bm{f}(t)
$$
The inclusion of an associated port modifies the dynamics as follows
$$
\begin{aligned}
\mathbf{E}\dot{\mathbf{x}} &= \mathbf{J}(\mathbf{x}) \mathbf{z} + \mathbf{B} \mathbf{u} + \mathbf{B}_f\mathbf{f}, \\
\mathbf{y} &= \mathbf{B}^\top \mathbf{z}, \\
\mathbf{y}_f &= \mathbf{B}_f^\top \mathbf{z},
\end{aligned} \qquad \qquad \mathbf{B}_f = \begin{bmatrix}
    0 \\
    \bm\Xi(s_p)^\top \\
    0
\end{bmatrix}.
$$
}

\section{Interconnection of subsystems}\label{sec:interconnection}
In this section, the different discrete formulations (free, clamped, and pinned) are interconnected to assemble systems without the need for Lagrange multipliers. \rewthree{This does not imply that every flexible multibody system can be formulated without algebraic constraints; rather, the proposed methodology provides greater flexibility in the assembly of beam structures.} For the considered examples the guided model is not employed. This can be used for two different purposes:
\begin{itemize}
    \item We obtain mixed boundary conditions. This point is explained in a more general and abstract framework in \cite{dejong2026decomposition} but here it is made concrete by the model choice;
    \item We construct multibody systems with less or in some cases even no algebraic constraints;
\end{itemize}

Therefore, the methodology shares similarities with the two-input two-output approach \cite{finozzi2022parametric,sanfedino2022advances}. However, the presented approach guarantees the preservation of the overall Hamiltonian structure without the need of inverting a dynamical system. The interconnection of systems embodies the power preservation across interconnection points and is an analogous procedure to the choice of flux terms in discontinuous Galerkin methods \cite{hesthaven2008dg}. These two aspects are illustrated by two examples respectively: a cantilever beam and a four bar mechanism. \\

\begin{remark}
    \rewone{
    The present framework imposes equal velocities and forces in a weak manner meaning that point of connection in multibody systems do not have the same velocity but rather exchange forces and velocity fluxed. The idea is analogous to that of a numerical flux in discontinuous Galerkin method. Therefore there is no strong enforcement of velocity constraints and drifts can indeed occur. We did not observe drift issues in the numerical examples presented in the paper. Penalization strategies can be included to guarantee that the equality of velocities of forces is not violated. 
    Within this framework multibody system can be assembled without algebraic constraints by exchanging flux information. This is a novel perspective with respect to standard position based finite element discretization. However, it is still possible to interconnected components in a standard way, by imposing equality of velocities.}
    This way of interconnecting systems corresponds to using a feedback interconnection (or gyrator in port-Hamiltonian jargon \cite{vanderschaft2014overview}). Another way, that is analogous to finite element assembly, is to interconnect system via a transformer interconnection. This transformer interconnection is used when systems arising from the same formulation are used together. The interested reader can consult \cite{brugnoli2021multibody} for an application of the latter approach to the floating frame of reference formulations of beams. 
\end{remark}

\subsection{Mixed boundary conditions: a cantilever beam}
To obtain a dynamical system representing a cantilever beam, cf Fig. \ref{fig:cantilever_beam}, the clamped-clamped representation presented in Sec. \ref{sec:spatial_discretization_clampedclamped} is interconnected with the free-free one of Sec. \ref{sec:spatial_discretization_freefree} by means of a feedback interconnection.

\begin{figure}[htb]
    \centering
    \begin{minipage}{0.38\textwidth}
        \centering
        \includegraphics[width=\linewidth]{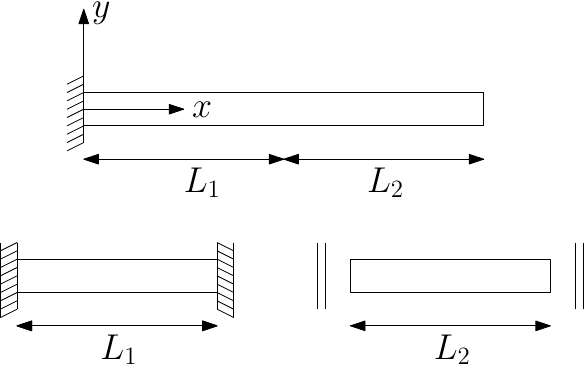}
        \caption{Cantilever beam split into a clamped and a free beam}
        \label{fig:cantilever_beam}
    \end{minipage}\hfill
    \begin{minipage}{0.58\textwidth}
        \centering
        \includegraphics[width=\linewidth]{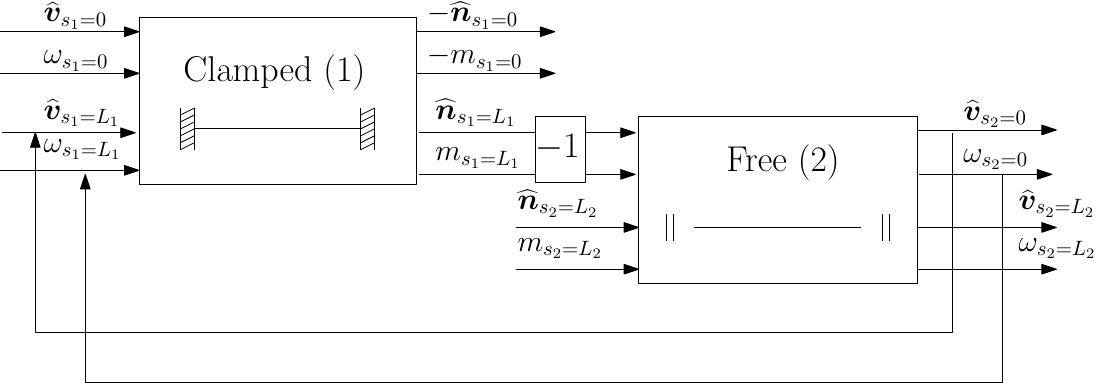}
        \caption{A cantilever beam as interconnected system}
        \label{fig:cantilever_blockdiagram}
    \end{minipage}
\end{figure}

The overall cantilever beam can be split at an arbitrary point with coordinate $x=L_1$. Here the midpoint is chosen for simplicity. To express the interconnection, the left (l) and right (r) inputs and outputs are explicitly shown. The left half is described by the clamped-clamped formulation.

\begin{equation*}
\text{Beam 1} \quad
\begin{cases}
\mathbf{E}_{C} \dot{\mathbf{x}}_{C} = \mathbf{J}_{C}(\mathbf{x}_{C}) \mathbf{z}_{C} + \mathbf{B}_{C,l}\mathbf{u}_{C, l} + \mathbf{B}_{C, r}{\mathbf{u}_{C, r}}, \\
\mathbf{y}_{C, l} = \mathbf{B}_{C, l}^\top {\mathbf{z}}_{C}, \\
{\mathbf{y}_{C, r}} = \mathbf{B}_{C, r}^\top {\mathbf{z}}_{C}.
\end{cases}
\end{equation*}
The right half is described by the free formulation
\begin{equation*}
\text{Beam 2} \quad 
\begin{cases}
\mathbf{E}_{F} \dot{\mathbf{x}}_{F} = \mathbf{J}_{F}(\mathbf{x}_{F}) {\mathbf{z}}_{F} + \mathbf{B}_{F, l}{\mathbf{u}_{F, l}} + \mathbf{B}_{F, r} \mathbf{u}_{F, r}, \\
{\mathbf{y}_{F, l}}= \mathbf{B}_{F, l}^\top {\mathbf{z}}_{F}, \\
\mathbf{y}_{F, r} = \mathbf{B}_{F, r}^\top {\mathbf{z}}_{F}.
\end{cases}
\end{equation*}
The gyrator (or feedback) interconnection 
\begin{equation*}
\begin{aligned}
    {\mathbf{u}_{C, r}} &= {\mathbf{y}_{F, l}}, \\ 
    {\mathbf{u}_{F, l}} &= -{\mathbf{y}_{C, r}},
\end{aligned} 
\end{equation*}
represents essentially Newton's third law: The linear and angular velocity is the same on both sides and the forces and torque are equal and opposite.
The corresponding block-diagram is represented in Fig. \ref{fig:cantilever_blockdiagram}. The resulting interconnected system can be written as 
\begin{equation*}
\begin{aligned}
\begin{bmatrix}
\mathbf{E}_{C}  & 0 \\
0 & \mathbf{E}_{F} 
\end{bmatrix} 
\begin{pmatrix}
\dot{\mathbf{x}}_{C} \\
\dot{\mathbf{x}}_{F} \\
\end{pmatrix} &= 
\begin{bmatrix}
\mathbf{J}_{C}(\mathbf{x}_{C})  & +\mathbf{B}_{C, r} \mathbf{B}_{F, l}^\top \\
-\mathbf{B}_{F, l} \mathbf{B}_{C, r}^\top & \mathbf{J}_{F}(\mathbf{x}_F) 
\end{bmatrix} 
\begin{pmatrix}
\mathbf{z}_C \\
\mathbf{z}_F \\
\end{pmatrix} + 
\begin{bmatrix}
    \mathbf{B}_{C,l} & 0 \\
    0 & \mathbf{B}_{F, r}
\end{bmatrix}
\begin{pmatrix}
    \mathbf{u}_{C, l} \\
    \mathbf{u}_{F, r}
\end{pmatrix}, \\
\begin{pmatrix}
    \mathbf{y}_{C, l} \\
    \mathbf{y}_{F, r}
\end{pmatrix} &= 
\begin{bmatrix}
    \mathbf{B}_{C, l}^\top & 0 \\
    0 & \mathbf{B}_{F, r}^\top
\end{bmatrix}
\begin{pmatrix}
\mathbf{z}_C \\
\mathbf{z}_F \\
\end{pmatrix}.
\end{aligned}
\end{equation*}
Note that this is again a port-Hamiltonian system with skew-symmetric structure matrix. To this end, all beneficial properties of this system class are inherited by larger multibody systems constructed with our approach. In the case of homogeneous boundary conditions, meaning that the beam is clamped at at the left extremity and free and unforced at the right, the inputs are simply set to zero, i.e.
\begin{equation*}
\begin{aligned}
    \mathbf{u}_{C, l}&=0, \qquad \text{The beam is clamped at the left extremity,} \\
    \mathbf{u}_{F, r}&=0, \qquad \text{No forces and torques are applied at the right extremity.}    
\end{aligned}
\end{equation*}

\subsection{A kinematic loop: four bar mechanism}\label{sec:four_bar}
The four bar mechanism \cite{kinon2025energy} is a simple example of a closed kinematic chain.  Beams 1 and 3 are connected to the ground through revolute support joints, cf. Fig.~\ref{fig:four_bar}. The proposed methodology makes it possible to represent this system without making use of algebraic multipliers. For this systems, as the bending moments are zero at the joints, the bending moment input is set to zero for all models meaning $\mathbf{u}_m=0$. Then, the inputs and outputs will either represent linear velocities or forces. To explain the interconnection, consider the dynamics of three beams separately. By considering the pinned case for beam 1 and beam 3 as well as the free-free case for beam 2 we obtain the representation
\begin{figure}[htb]
    \centering
    \includegraphics[width=0.5\textwidth]{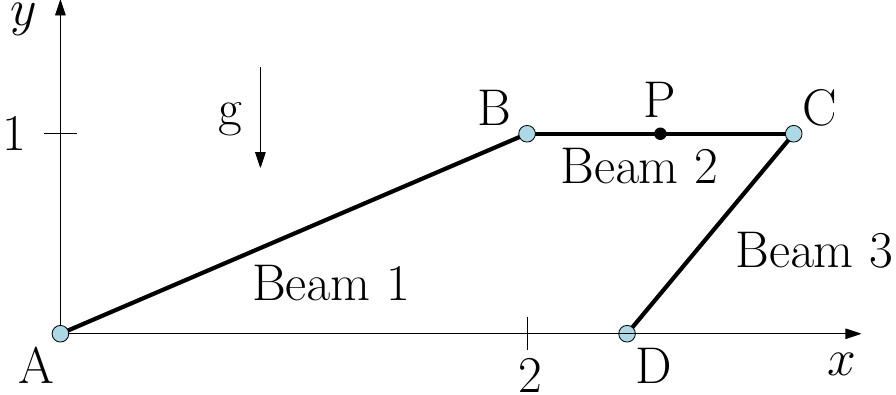}
    \caption{Four bar mechanism}
    \label{fig:four_bar}
\end{figure}
\begin{figure}[htb] 
\centering 
\includegraphics[width=0.95\textwidth]{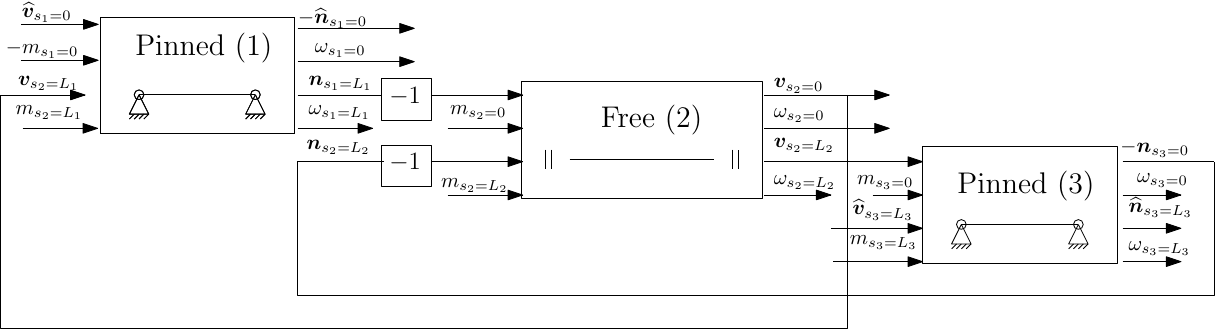} 
\caption{Block diagram for the four bar mechanism}
\label{fig:four_bar_blockdiagram} 
\end{figure} 

\begin{equation*}
\begin{aligned}
&\text{Beam 1} \quad 
\begin{cases}
\mathbf{E}_{P1} \dot{\mathbf{x}}_{P1} = \mathbf{J}_{P1}(\mathbf{x}_{P1}) {\mathbf{z}}_{P1} + \mathbf{B}_{P1, vl}
{\mathbf{u}_{P1, vl}} + \mathbf{B}_{P1, vr}{\mathbf{u}_{P1, vr}}, \\
\mathbf{y}_{P1, nl} = \mathbf{B}_{P1, vl}^\top {\mathbf{z}}_{P1}, \\
{\mathbf{y}_{P1, nr}} = \mathbf{B}_{P1, vl}^\top {\mathbf{z}}_{P1}, 
\end{cases} \\
&\text{Beam 2} \quad 
\begin{cases}
\mathbf{E}_{F2} \dot{\mathbf{x}}_{F2} = \mathbf{J}_{F2}({\mathbf{x}}_{F2}){\mathbf{z}}_{F2} + \mathbf{B}_{F2, nl}{\mathbf{u}_{F2, nl}} + \mathbf{B}_{F2, nr}{\mathbf{u}_{F2, nr}}, \\
{\mathbf{y}_{F2, vl}} = \mathbf{B}_{F2, nl}^\top {\mathbf{z}}_{F2}, \\
{\mathbf{y}_{F2, vr}} = \mathbf{B}_{F2, nr}^\top {\mathbf{z}}_{F2}, \\
\end{cases} \\
&\text{Beam 3} \quad 
\begin{cases}
\mathbf{E}_{P3} \dot{\mathbf{x}}_{P3} = \mathbf{J}_{P3}(\mathbf{x}_{P3}) {\mathbf{z}}_{P3} + \mathbf{B}_{P3, vl}
{\mathbf{u}_{P3, vl}} + \mathbf{B}_{P3, vr}{\mathbf{u}_{P3, vr}}, \\
{\mathbf{y}_{P3, nl}} = \mathbf{B}_{P3, vl}^\top {\mathbf{z}}_{P3}, \\
\mathbf{y}_{P3, nr} = \mathbf{B}_{P3, vl}^\top {\mathbf{z}}_{P3}.
\end{cases}
\end{aligned}
\end{equation*}
where $\mathbf{u}_{Pi, vl},\; \mathbf{u}_{Pi, vr}$ denote the left, respectively right, velocity input for the pinned $i$-th beam ($i={1,3}$) and $\mathbf{u}_{F2, nl},\; \mathbf{u}_{F2, nr}$ denote the left, respectively right, force input for the free $2$nd free beam. The conjugated outputs $\mathbf{y}_{Pi, vl},\; \mathbf{y}_{Pi, vr}$ represent the left, respectively right, force output for the pinned $i$-th beam and $\mathbf{y}_{F2, nl},\; \mathbf{y}_{F2, nr}$ represent the left, respectively right, velocity output for the $2$nd free beam. The corresponding block-diagram is shown in Fig. \ref{fig:four_bar_blockdiagram}.
Pinning the beams 1 and 3 to the ground at point A and D, respectively, is then expressed in terms of vaninishing velocities as
\begin{equation}
\begin{aligned}
    {\mathbf{u}_{P1, vl}} &= 0, \\
    {\mathbf{u}_{P3, vr}} &= 0. \\
\end{aligned}
\end{equation}
The revolute joint connections at point B and C require the orientation of the adjacent beams such that
\begin{equation*}
\begin{aligned}
    \bm{\Lambda}(\theta_{P1}(L_1, t)){\mathbf{u}_{P1, vr}} &= +\bm{\Lambda}(\theta_{F2}(0, t)){\mathbf{y}_{F2, vl}}, \\ 
    \bm{\Lambda}(\theta_{F2}(0, t)){\mathbf{u}_{F2, nl}} &= -\bm{\Lambda}(\theta_{P1}(L_1, t)){\mathbf{y}_{P1, nr}}, \\
    \bm{\Lambda}(\theta_{P3}(0, t)){\mathbf{u}_{P3, vl}} &= +\bm{\Lambda}(\theta_{F2}(L_2, t)){\mathbf{y}_{F2, vr}}, \\ 
    \bm{\Lambda}(\theta_{F2}(L_2, t)){\mathbf{u}_{F2, nr}} &= -\bm{\Lambda}(\theta_{P3}(0, t)){\mathbf{y}_{P3, nl}},
\end{aligned} 
\end{equation*}
While the first and third equation constitute equal spatial velocities in the points B and C, respectively, the second and fourth row correspond to opposite spatial forces. The overall interconnected system can eventually be written as
\begin{equation*}
\mathrm{Diag}
\begin{bmatrix}
\mathbf{E}_{P1} \\
\mathbf{E}_{F2} \\
\mathbf{E}_{P3} \\
\end{bmatrix} 
\begin{pmatrix}
\dot{\mathbf{x}}_{P1} \\
\dot{\mathbf{x}}_{F2} \\
\dot{\mathbf{x}}_{P3} \\
\end{pmatrix} = 
\begin{bmatrix}
\mathbf{J}_{P1}({\mathbf{x}}_{P1})  & +\mathbf{B}_{P1, vr} \bm{\Lambda}(\alpha) \mathbf{B}_{F2, nl}^\top & 0 \\
-\mathbf{B}_{F2, nl} \bm{\Lambda}^\top(\alpha) \mathbf{B}_{P1, vr}   & \mathbf{J}_{F2}({\mathbf{x}}_{F2}) & -\mathbf{B}_{F2, nr} \bm{\Lambda}^\top(\beta) \mathbf{B}_{P3, vl}^\top \\
0 & \mathbf{B}_{P3, vl} \bm{\Lambda}(\beta) \mathbf{B}_{F2, nr}^\top & \mathbf{J}_{P3}({\mathbf{x}}_{P3})
\end{bmatrix} 
\begin{pmatrix}
\mathbf{z}_{P1} \\
\mathbf{z}_{F2} \\
\mathbf{z}_{P3} \\
\end{pmatrix},
\end{equation*}
where
$$
\alpha(t) = \theta_{F2}(0, t) - \theta_{P1}(L_1, t), \qquad  
\beta(t) = \theta_{F2}(L_2, t) - \theta_{P3}(0, t),
$$
denote relative rotation angles.

\section{Time integration}\label{sec:time_integration}
After interconnection, the final multibody system is written in a similar form to Eq. \ref{eq:discrete_nonlinear}, with the important difference that the matrix $\mathbf{J}$ now may depend on the rotation angles $\bm{\theta}$ because of the presence of joints, as
\begin{equation} \label{eq:discrete_nonlinear_multibody}
\begin{bmatrix}
    \mathbf{M}^{\mathbf{q}} & 0 \\
    0 & \mathbf{M}^{\mathbf{e}}
\end{bmatrix} 
\odv{}{t}
\begin{pmatrix}
    \mathbf{q}\\
    \mathbf{e}
\end{pmatrix} = 
\begin{bmatrix}
    0 & \mathbf{P}(\bm{\theta}) \\
    -\mathbf{P}^\top(\bm{\theta}) & \mathbf{J}(\mathbf{e}, \bm{\theta})
\end{bmatrix}
\begin{pmatrix}
    \mathbf{z}_{\mathbf{q}} \\
    \mathbf{e}
\end{pmatrix} + 
\begin{bmatrix}
    0 \\
    \mathbf{B}^\mathbf{e}
\end{bmatrix}\mathbf{u},
\end{equation}
where now the matrix $\mathbf{J}(\mathbf{e}, \bm{\theta}) = - \mathbf{J}^\top(\mathbf{e}, \bm{\theta})$ can be decomposed into the local dynamics of each subsystem and a part due to the interconnection
$$
\mathbf{J}(\mathbf{e}, \bm{\theta}) = \mathbf{J}^\mathbf{e}(\mathbf{e}) + \mathbf{J}^{\mathrm{int}}(\bm{\theta}), \qquad \mathbf{J}^\mathbf{e} = \mathrm{Diag}(\mathbf{J}^\mathbf{e}_1, \mathbf{J}^\mathbf{e}_2, \dots, \mathbf{J}^\mathbf{e}_{N_\mathrm{systems}}), 
$$
where $\mathbf{J}^\mathbf{e}$ is the block-diagonal concatenation of the matrices of each subsystem whereas $\mathbf{J}^{\mathrm{int}}$ contains all the coupling terms. All other matrices are obtained by concatenation of matrices of the single subsystems just like $\mathbf{J}^\mathbf{e}$. When including the gravity effect, the Hamiltonian is quadratic in $\mathbf{e}$ and linear in $\mathbf{q}$.

Therefore the energy can be exactly preserved using the implicit midpoint scheme. This leads to the following system, where $\Delta t$ is the constant time step size,

\begin{equation}\label{eq:discrete_time}
\begin{bmatrix}
    \mathbf{M}^{\mathbf{q}} & 0 \\
    0 & \mathbf{M}^{\mathbf{e}}
\end{bmatrix} 
\begin{pmatrix}
    \mathbf{q}_{n+1} - \mathbf{q}_{n} \\
    \mathbf{e}_{n+1} - \mathbf{e}_{n}
\end{pmatrix} = 
\Delta t
\begin{bmatrix}
    0 & \mathbf{P}(\bm{\theta}_{n+1/2}) \\
    -\mathbf{P}^\top(\bm{\theta}_{n+1/2}) & \mathbf{J}(\mathbf{e}_{n+1/2}, \bm{\theta}_{n+1/2})
\end{bmatrix}
\begin{pmatrix}
    \mathbf{z}_{\mathbf{q}} \\
    \mathbf{e}_{n+1/2}
\end{pmatrix} + 
\Delta t
\begin{bmatrix}
    0 \\
    \mathbf{B}^\mathbf{e}
\end{bmatrix}\mathbf{u}_{n+1/2},
\end{equation}
where $\mathbf{z}_{\mathbf{q}} = \nabla_{\mathbf{q}} U$ is a constant gravity force. The notation $\mathbf{x}_{n} \approx \mathbf{x}(t_n), \; \mathbf{x}_{n+1} \approx \mathbf{x}(t_{n+1})$ denotes sampling approximations at the current or next time instance, whereas $\mathbf{x}_{n+1/2} := \frac{1}{2}(\mathbf{x}_{n} + \mathbf{x}_{n+1})$. Notice that the evolution of the centerline position can be reconstructed a posteriori from the knowledge of $\bm{\theta}, $ and $\bm{v}$ by integrating
\begin{equation}\label{eq:discrete_time_positions}
\mathbf{M}^{\mathbf{r}} (\mathbf{r}_{n+1} - \mathbf{r}_{n}) = \Delta t \mathbf{R}(\theta_{n+1/2}) \mathbf{v}_{n+1/2}.    
\end{equation}

This is the time discrete version of the first block of differential equations in system \eqref{eq:discrete_nonlinear_multibody}. The matrix $\mathbf{R}(\bm{\theta})$ can be interpreted as a discretization of the rotation matrix and it is given by 
$$
\mathbf{R}(\bm{\theta}) = \int_\Omega \begin{bmatrix}
    \bm{\xi}_r^\top & 0 \\
    0 & \bm{\xi}_r^\top
\end{bmatrix} \bm{\Lambda}(\theta_h) \begin{bmatrix}
    \bm{\xi}_r & 0 \\
    0 & \bm{\xi}_r
\end{bmatrix} \d\Omega,
$$
where $\bm{\xi}_r$ is the finite element ansatz matrix for the centerline position and the linear velocity.  

\begin{proposition}
When $\mathbf{u}_{n+1/2}=0$ the discrete energy $H(\mathbf{x}_{n}) = \frac{1}{2}\mathbf{e}_{n}^\top \mathbf{M}^{\mathbf{e}} \mathbf{e}_{n}  + \rho g A \mathbf{a}^\top \mathbf{r}_{y, n}$ is preserved. If instead the system is forced by an external input, the following power balance is obtained
$$
\frac{H(\mathbf{x}_{n}) - H(\mathbf{x}_{n+1})}{\Delta t} =  \mathbf{y}_{n+1/2}^\top \mathbf{u}_{n+1/2}, \qquad  \mathbf{y}_{n+1/2}:= (\mathbf{B}^{\mathbf{e}})^\top \mathbf{e}_{n+1/2}.
$$
 
\end{proposition}
\begin{proof}
     Premultipling the left and right side of Eq. \eqref{eq:discrete_time} by $(\mathbf{z}_{\mathbf{q}}, \; \mathbf{e}_{n+1/2}$), it is obtained that
\begin{equation}
    \begin{pmatrix}
    \mathbf{z}_{\mathbf{q}} \\
    \mathbf{e}_{n+1/2}
\end{pmatrix}^\top \begin{bmatrix}
    \mathbf{M}^{\mathbf{q}} & 0 \\
    0 & \mathbf{M}^{\mathbf{e}}
\end{bmatrix} 
\begin{pmatrix}
    \mathbf{q}_{n+1} - \mathbf{q}_{n} \\
    \mathbf{e}_{n+1} - \mathbf{e}_{n}
\end{pmatrix} = 0, \label{eq_proof_1}
\end{equation}
as the first matrix on the right hand side of \eqref{eq:discrete_time} is skew-symmetric. Multiplying Eq.~\eqref{eq:gradient_potential} by $(\mathbf{r}_{y, n+1} - \mathbf{r}_{y, n})$ and using the symmetry of $\mathbf{M}^{\mathbf{r}_y}$ yields 
\begin{equation}
    \bm{1}^\top \mathbf{M}^{\mathbf{r}_y}(\mathbf{r}_{y, n+1} - \mathbf{r}_{y, n}) = \mathbf{a}^\top (\mathbf{r}_{y, n+1} - \mathbf{r}_{y, n}). \label{eq_proof_2}
\end{equation}

Using the relation \eqref{eq:z_vector} together with \eqref{eq_proof_2}, one can rewrite \eqref{eq_proof_1} as the following discrete energy conservation balance 
$$
\frac{1}{2}\mathbf{e}_{n}^\top \mathbf{M}^{\mathbf{e}} \mathbf{e}_{n}  + \rho g A \mathbf{a}^\top \mathbf{r}_{y, n} = \frac{1}{2}\mathbf{e}_{n+1}^\top \mathbf{M}^{\mathbf{e}} \mathbf{e}_{n+1}  + \rho g A \mathbf{a}^\top \mathbf{r}_{y, n+1}.
$$
That is, $H(\mathbf{x}_{n}) = H(\mathbf{x}_{n+1})$.  The power balance above is obtained by an analogous computation. 
\end{proof}
\begin{remark}
In the material formulation, the angular momentum becomes a nonlinear expression because of the rotation matrix. Its expression is 
$$
\begin{aligned}
J_n &= \inpr[\Omega]{\bm{r}_{h, n}}{\rho A \bm{S}\bm{\Lambda}(\theta_{h, n})\mat{v}_{h, n}} + \inpr[\Omega]{1}{\rho I \omega_n}, 
= \rho A \mathbf{r}_n^\top \mathbf{S}\mathbf{R}(\bm{\theta}_n)\mathbf{v}_n + \rho I \mathbf{1}^\top\mathbf{M}^\omega \mathbf{w}_n, \qquad \mathbf{S}:= \begin{bmatrix}
    0 & \mathbf{I} \\
    -\mathbf{I} & 0
\end{bmatrix}.    
\end{aligned}
$$
Therefore, the implicit midpoint scheme does not preserve the angular momentum. In the spatial-material formulation, however, this can be achieved, cf.~\cite{kinon2025energy}. 

\end{remark}

\begin{remark}
The recursion \eqref{eq:discrete_time} can be equivalently written as
$$
\begin{pmatrix}
    \mathbf{q}_{n+1} \\
    \mathbf{e}_{n+1}
\end{pmatrix} = \mathrm{Cay}\left(\frac{\Delta t}{2}\mathbf{E}^{-1}\mathbf{J}(\mathbf{x}^{n+1/2})\right)
\begin{pmatrix}
    \mathbf{z}_{\mathbf{q}} \\
    \mathbf{e}_{n}
\end{pmatrix} +  \left(\mathbf{E} - \frac{\Delta t}{2}\mathbf{J}(\mathbf{x}^{n+1/2})\right)^{-1} \Delta t\mathbf{B}\mathbf{u}_{n+1/2},
$$
where $\mathrm{Cay}(\mathbf{A}) := (\mathbf{I}-\mathbf{A})^{-1}(\mathbf{I}+\mathbf{A})$ denotes the Cayley transform of the matrix $\mathbf{A}$.  
\end{remark}

\section{Numerical examples}\label{sec:num_examples}

In this section, the performance of our approach is assessed on four different examples:
\begin{itemize}
    \item The flying spaghetti benchmark \cite{simo1986flying}. This example is used to verify the proposed mixed finite element discretization and its properties.
    \item A pendulum with a lumped mass attached \cite{ljukovac2025reissner}. This simple example shows how the overall model can be constructed by interconnecting the pendulum bar and the mass as two subsystems. We compare our results with the literature.
    \item An L-shaped frame \cite{betsch2001conservation}.
    This example demonstrates how the proposed interconnection strategy combines a free-free model with a cantilever beam, cf. Sec \ref{sec:interconnection}.
    \item A four bar mechanism. This example demonstrates that the proposed interconnection strategy can be used on closed loop kinematic chains as well. 
\end{itemize}

For all examples we check energy conservation as discussed in the previous Section. This is a common feature of all models discussed in this paper. The code is implemented in \textsc{python}. To implement the Newton method for the implicit midpoint method, the \textsc{jax} library \cite{jax2018} is used to compute the Jacobian using automatic differentiation. 

\subsection{Flying spaghetti}

The first example is concerned with the free motion of a floating beam. This is the so-called \textit{flying spaghetti} problem original proposed in \cite{simo1986flying}. The objective of the test is to validate the novel mixed finite element discretization strategy and verify the energy conservation properties induced by the time integration approach, see Sec \ref{sec:time_integration}. Note that we make use of the free-free model according to Sec. \ref{sec:spatial_discretization_freefree}. Initially in an inclined configuration (see Fig. \ref{fig:flying_spaghetti_0}), the beam is subject to the following boundary conditions: a free end at $s = 0$ and a nodal force and torque at $s =L$, $\bm{n}|_{s = L} = \bm{n}_0, \;  m|_{s = L} = m_0$. These are applied for $2.5 \; \mathrm{[s]}$ and given by the expressions
$$
\bm{n}_0(t) = \begin{cases}
    8 \bm{b}_1 \quad &\text{for } t\le 2.5 \; \mathrm{[s]}, \\
    0 \quad &\text{for } t > 
    2.5 \; \mathrm{[s]}, \\
\end{cases} \qquad 
{m}_0(t) = \begin{cases}
    -80 \quad &\text{for } t\le 2.5 \; \mathrm{[s]}, \\
    0 \quad &\text{for } t > 
    2.5 \; \mathrm{[s]}. \\
\end{cases}
$$

After $2.5 \; \mathrm{[s]}$, the system is isolated and the beam floats freely. The physical parameters, reported in Tab. \ref{tab:parameters_spaghetti}, are the same as in \cite{simo1986flying}, albeit non physical since $EA/EI = \rho A/ \rho I$ does not hold true.
The configurations of the beam at different time instants are shown together with a colorcoding for the bending moment in Fig. \ref{fig:configurations_flying_spaghetti} and they agree with the results reported in \cite[Fig. 4]{simo1986flying}. \rewfour{The initial force and torque applied at one boundary during the first $2.5 \; \mathrm{[s]}$ impart a net linear and angular momentum to the structure. Because the system is completely isolated after this transient forcing, the beam undergoes continuous large-scale rigid-body translation and rotation across the $x$-$y$ plane as it drifts freely. The applied loads simultaneously excite the highly flexible dynamic modes of the structure, causing significant continuous bending. The color-coding in Figure \ref{fig:configurations_flying_spaghetti} specifically maps the internal bending moment along the arc length. This visually captures the magnitude of the local elastic deformation as it evolves dynamically over time, reflecting the ongoing internal exchange between kinetic and elastic potential energy long after the external forcing ceases at $t = 2.5 \; \mathrm{[s]}$ s.} The trend of the energy and its increments are shown in Fig. \ref{fig:energy_flying_spaghetti}. It can be noticed that for $t>2.5 \; \mathrm{[s]}$  the energy is preserved up to machine precision. The results qualitatively coincide with the ones reported by \cite{kinon2025energy}. 

\begin{figure}[htb]
    \centering
    \includegraphics[width=0.35\linewidth]{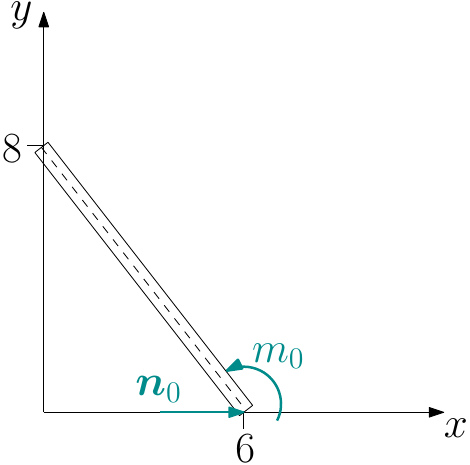}
    \caption{Initial configuration for the flying spaghetti}
    \label{fig:flying_spaghetti_0}
\end{figure}

\begin{table}[htb]
    \centering
    \begin{tabular}{ccccccccc}
        \hline
        $\Delta t$ [s] & $T$ [s] & $N_{\mathrm{el}}$ [-] & $L$ [m] & $\rho A$ [kg/m] & $\rho I$ [kg\,m] & $EA=GA$ [N] & $EI$ [N\,m$^2$] \\
        \hline
        0.1 & 15 & 10 & 10 & 1 & 10 & $10^4$ & 100 \\
        \hline
    \end{tabular}
    \caption{Parameters for flying spaghetti example}
    \label{tab:parameters_spaghetti}
\end{table}

\begin{figure}[htb]
    \centering
    \includegraphics[width=0.8\linewidth]{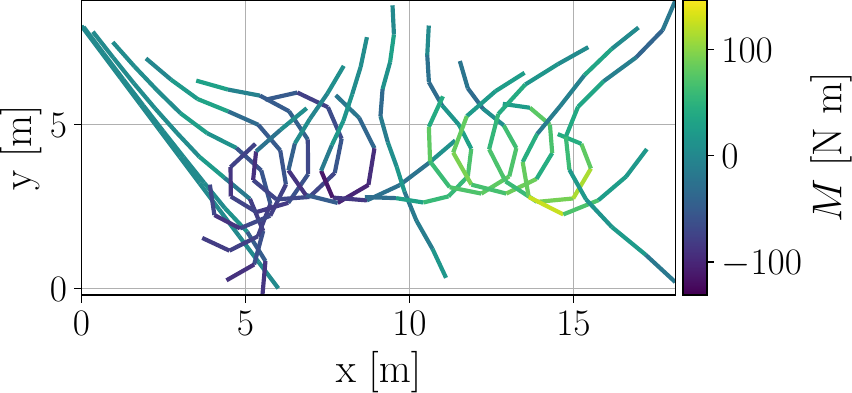}
    \caption{Configuration of the beam at different instants for $t \le 7.5$, after each 5 time increments}
    \label{fig:configurations_flying_spaghetti}
\end{figure}

\begin{figure}[htb]
    \centering
    \begin{subfigure}[b]{0.48\textwidth}
        \centering
        \includegraphics[width=\textwidth]{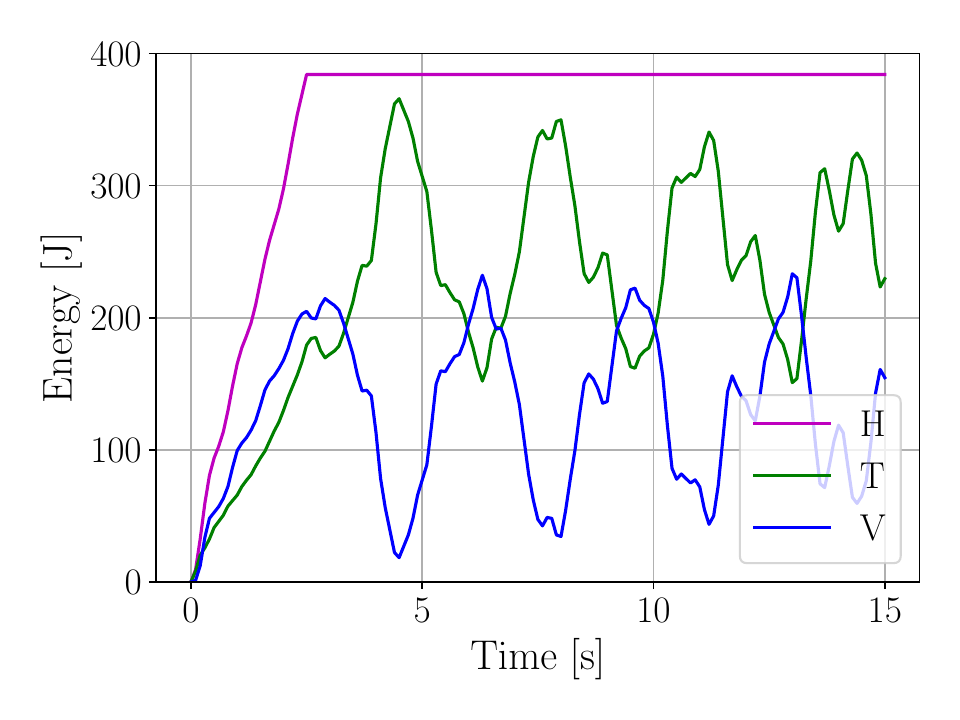}
    \end{subfigure}
    \hfill
    \begin{subfigure}[b]{0.48\textwidth}
        \centering
        \includegraphics[width=\textwidth]{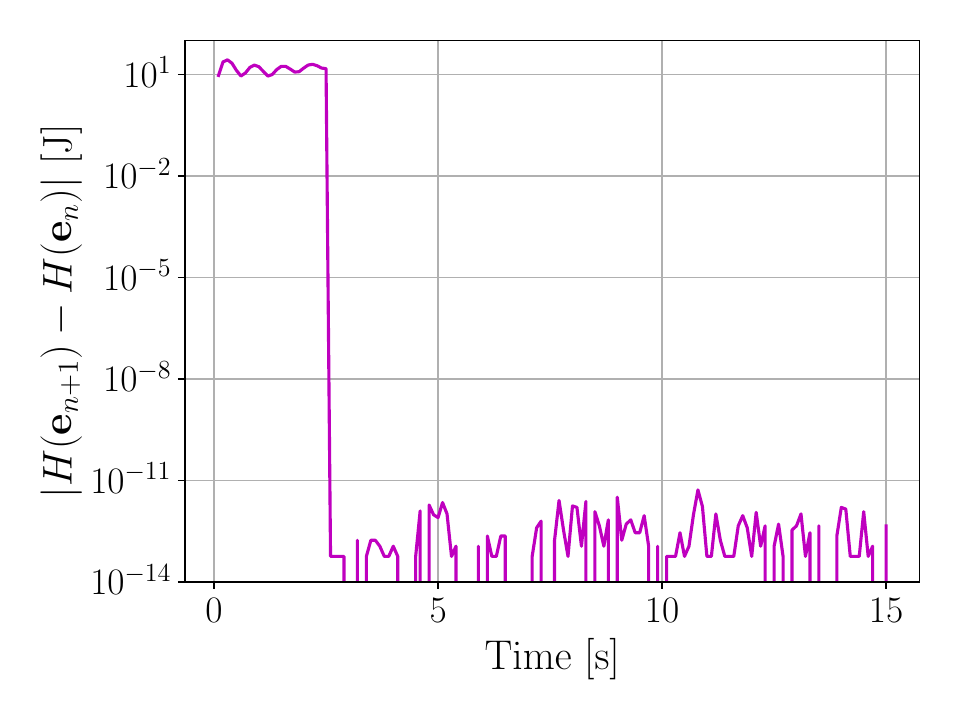}
    \end{subfigure}
    \caption{Evolution of the energy and its increments for the flying spaghetti}
    \label{fig:energy_flying_spaghetti}
\end{figure}

\subsection{Flexible pendulum}
The next example is a simple pendulum under gravity,  cf. Fig. \ref{fig:sketch_pendulum}. The setup and parameters are the same as the ones used in \cite{ljukovac2025reissner}. The pendulum has length $L$ with a concentrated mass $m$ positioned at the free end which is represented with a single finite element of length $h=3.04$.  To obtain the system the pinned formulation is interconnected with the mass as explained in Sec. \ref{sec:interconnection}, cf Fig. \ref{fig:pendulum_blockdiagram}. The geometric and material properties are given in Tab. \ref{tab:parameters_pendulum}. For the simulation the initial angle is taken to be 
$\theta_0 = \pi/2$ whereas all the other parameters are set to zero. The results in terms of angle and angular velocity are shown in Fig. \ref{fig:theta_omega_pendulum} and they coincide with those obtained in \cite{ljukovac2025reissner}. The energy evolution and its increments are reported in Fig. \ref{fig:energy_pendulum_pendulum}. It can be noticed that the overall energy is preserved up to machine precision.

\begin{figure}[htb]
    \centering
    \includegraphics[width=0.8\linewidth]{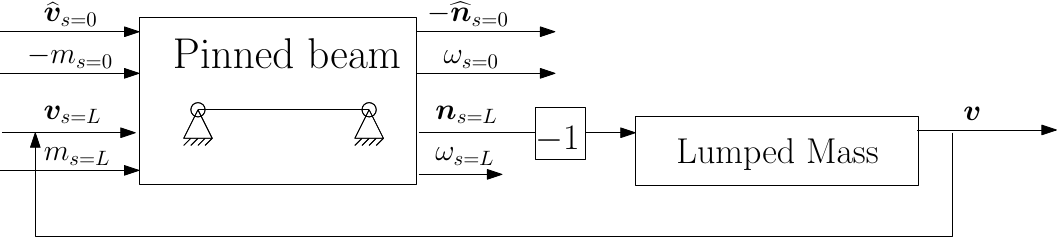}
    \caption{Pendulum with lumped mass as interconnected system}
    \label{fig:pendulum_blockdiagram}
\end{figure}

\begin{figure}[htb]
    \centering
    \includegraphics[width=0.25\linewidth]{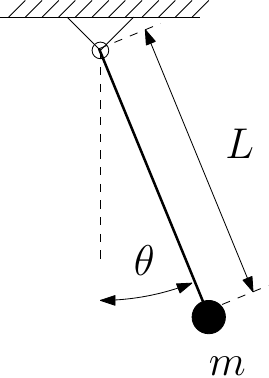}
    \caption{Pendulum}
    \label{fig:sketch_pendulum}
\end{figure}

\begin{table}[htb]
    \centering
    \begin{tabular}{ccccccccc}
        \hline
        $\Delta t$ [s] & $T$ [s] & $N_{\mathrm{el}}$ [-] & $L$ [m] & $\rho A$ [kg/m] & $\rho I$ [kg m] & $EA=GA$ [N] & $EI \; \mathrm{[N \; m^2]}$ & $m$ [kg] \\
        \hline
        0.1 & 5 & 1 & 3.04 & 0 & 0 & $10^{10}$ & $10^{10}$ & 10 \\
        \hline
    \end{tabular}
    \caption{Parameters for the flexible pendulum}
    \label{tab:parameters_pendulum}
\end{table}

\begin{figure}[htb]
    \centering
    \begin{subfigure}[b]{0.48\textwidth}
        \centering
        \includegraphics[width=\textwidth]{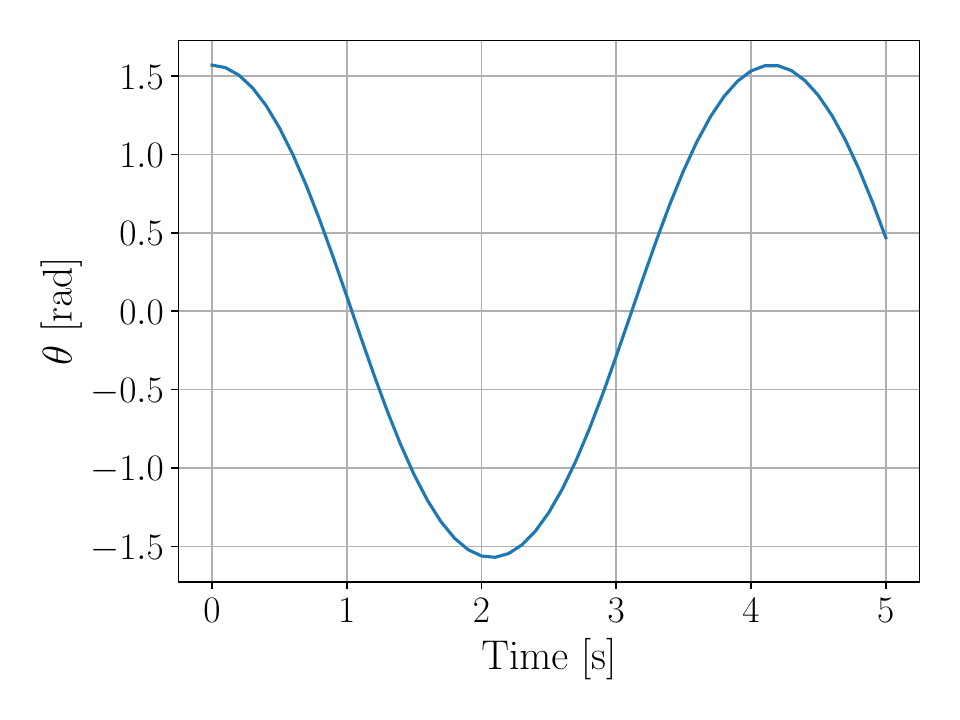}
    \end{subfigure}
    \hfill
    \begin{subfigure}[b]{0.48\textwidth}
        \centering
        \includegraphics[width=\textwidth]{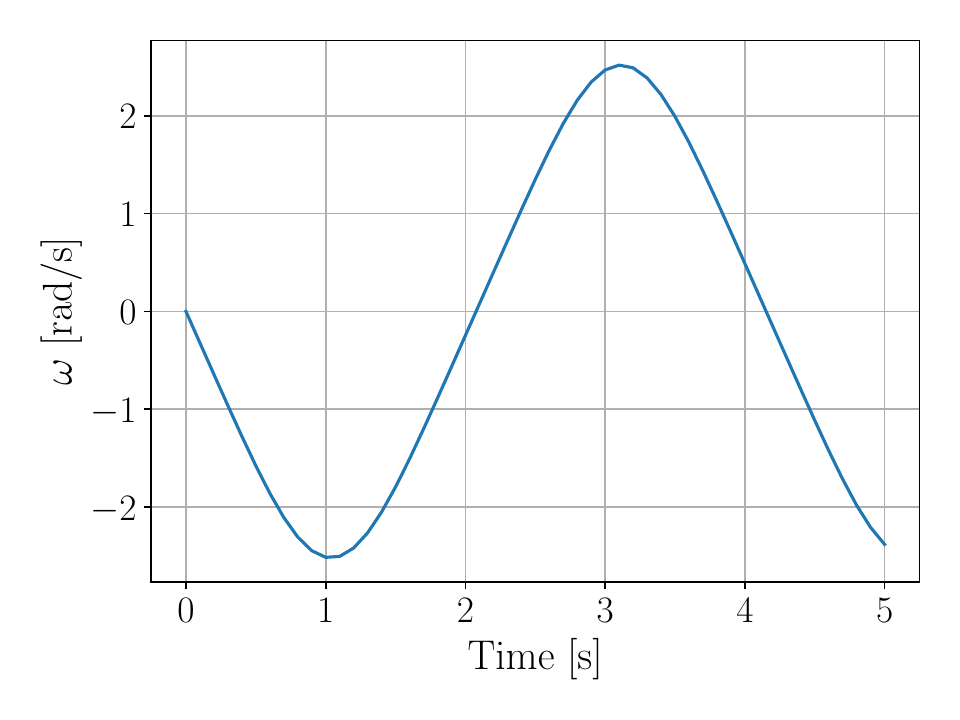}
    \end{subfigure}
    \caption{Evolution of rotation and of the angular velocity of the pendulum.}
    \label{fig:theta_omega_pendulum}
\end{figure}

\begin{figure}[htb]
    \centering
    \begin{subfigure}[b]{0.48\textwidth}
        \centering
        \includegraphics[width=\textwidth]{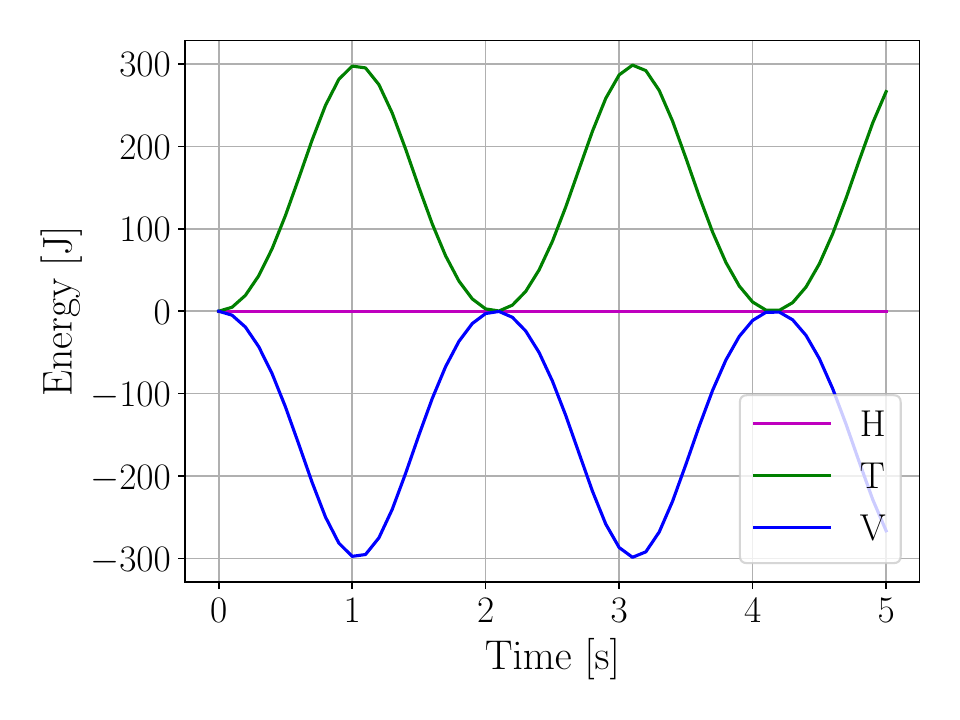}
    \end{subfigure}
    \hfill
    \begin{subfigure}[b]{0.48\textwidth}
        \centering
        \includegraphics[width=\textwidth]{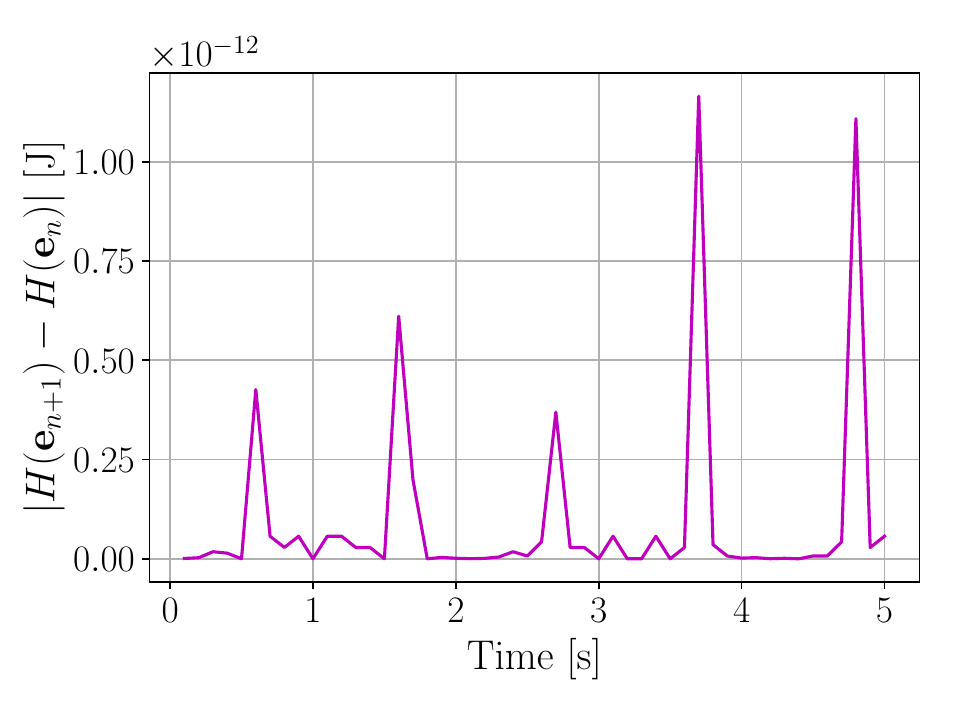}
    \end{subfigure}
    \caption{Evolution of energy and energy increments for the pendulum.}
    \label{fig:energy_pendulum_pendulum}
\end{figure}

\subsection{L-shaped frame}

We now consider the motion of a large frame. An analogous setting using geometrically exact beams has been considered in \cite{ljukovac2025reissner}, whereas in \cite{betsch2001conservation} the same problem is treated using Q4 elements (bilinear four node quadrilateral finite element). This system can be obtained via the interconnection of a free beam and a cantilever one, as illustrated in Fig. \ref{fig:Lshaped_blockdiagram}. The configuration is reported in Fig. \ref{fig:sketch_Lframe}, whereas the geometric and physical parameters are reported in Tab. \ref{tab:parameters_Lshaped_frame}. As in \cite{ljukovac2025reissner} two different scenarios are considered: a rigid frame and a more flexible one. For both cases a time step of $\Delta t = 0.1$ and $10$ finite elements per beam are considered. For the rigid case, the frame configurations at different instants are reported in Fig. \ref{fig:configurations_Lframe_stiff} and the energy signals and the incremental variation of the energy are reported in Fig. \ref{fig:energy_Lframe_stiff}. We realize the rigidity by setting compliance parameters in the compliance matrix to zero. The corresponding stress type quantities immediately act as Lagrange multipliers. For the soft case, the frame configurations at different instants are reported in Fig. \ref{fig:configurations_Lframe_soft} and the energy signals and the incremental variation of the energy are reported in Fig. \ref{fig:energy_Lframe_soft}. The results are well in accordance with the findings of \cite{ljukovac2025reissner}.

\begin{figure}
    \centering
    \includegraphics[width=0.8\linewidth]{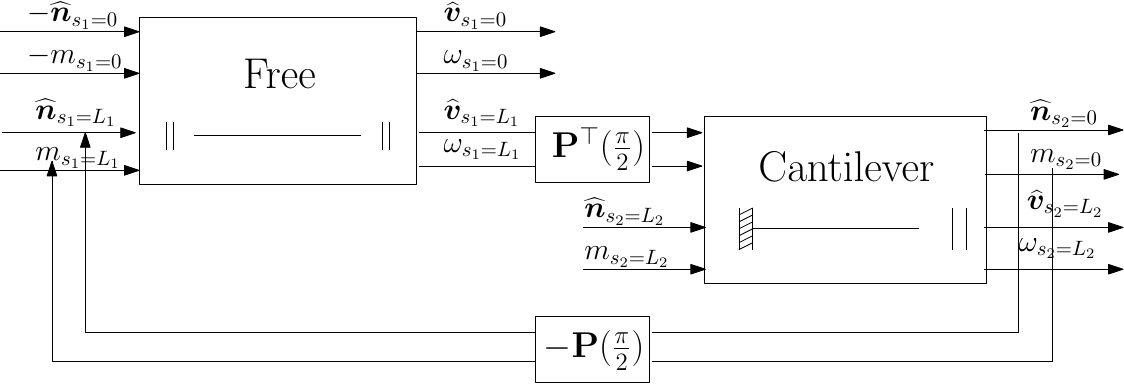}
    \caption{L-shaped frame obtained as interconnected of a free beam and a cantilever one.}
    \label{fig:Lshaped_blockdiagram}
\end{figure}

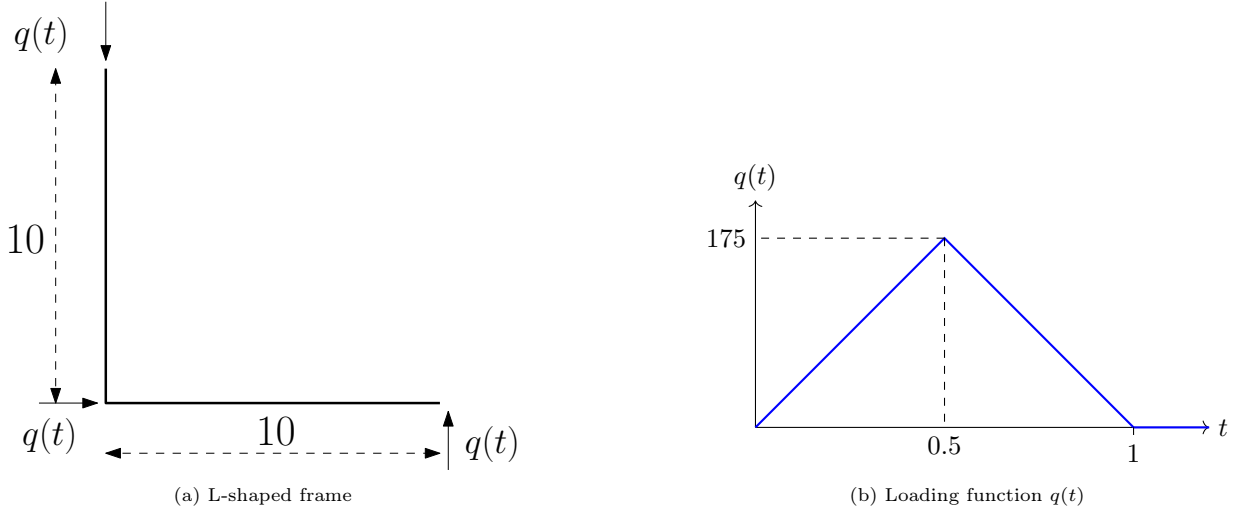
\begin{figure}[htb]
    \centering

    \begin{subfigure}[b]{0.45\linewidth}
        \centering
        \includegraphics[width=0.9\linewidth]{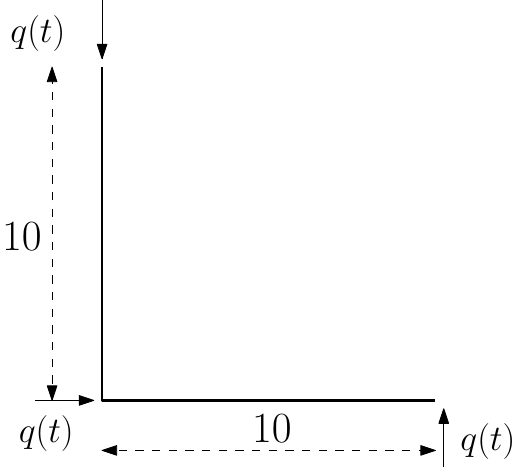}
        \caption{L-shaped frame.}
        \label{fig:sketch_Lframe}
    \end{subfigure}
    \hfill
    \begin{subfigure}[b]{0.45\linewidth}
        \centering
        \begin{tikzpicture}[scale=1]
            \draw[->] (0,0) -- (6,0) node[right] {$t$};
            \draw[->] (0,0) -- (0,3) node[above] {$q(t)$};
            
            \coordinate (O) at (0,0);
            \coordinate (A) at (2.5,2.5);
            \coordinate (B) at (5,0);
            
            \draw[thick, blue] (O) -- (A) -- (B);
            \draw[thick, blue] (5,0) -- (6,0);
            
            \draw[dashed] (A) -- (2.5,0) node[below] {$0.5$};
            \draw[dashed] (A) -- (0,2.5) node[left] {$175$};
            \draw[dashed] (B) -- (5,-0.1) node[below] {$1$};
        \end{tikzpicture}
        \caption{Loading function $q(t)$}
        \label{fig:Lframe_loading}
    \end{subfigure}

    \caption{Geometry and loading definition for the L-shaped frame.}
    \label{fig:Lframe_geometry_loading}
\end{figure}

\begin{table}[htb]
    \centering
    \begin{tabular}{ccccccccc}
    \hline
    & $\Delta t$ [s] & $T$ [s] & $N_{\mathrm{el}}$ [-] & $\rho A$ [kg/m] & $\rho I$ [kg\,m] & $EA$ [N] & $GA$ [N] & $EI$ [N\,m$^2$] \\
    \hline
    Rigid & $0.01$ & 5 & 10 & $1$ & $0.1$ & $10^7$ & $5 \cdot 10^6$ & $10^6$\\
    Soft  & $0.01$ & 5 & 10 & $0.5$ & $0.05$ & $10^{5}$ & $5 \cdot 10^4$ & $10^4$ \\
    \hline
\end{tabular}
    \caption{Parameters for the L-shaped frame}
    \label{tab:parameters_Lshaped_frame}
\end{table}

\begin{figure}[htbp]
    \centering
    \begin{subfigure}[b]{0.48\linewidth}
        \centering
        \includegraphics[width=\linewidth]{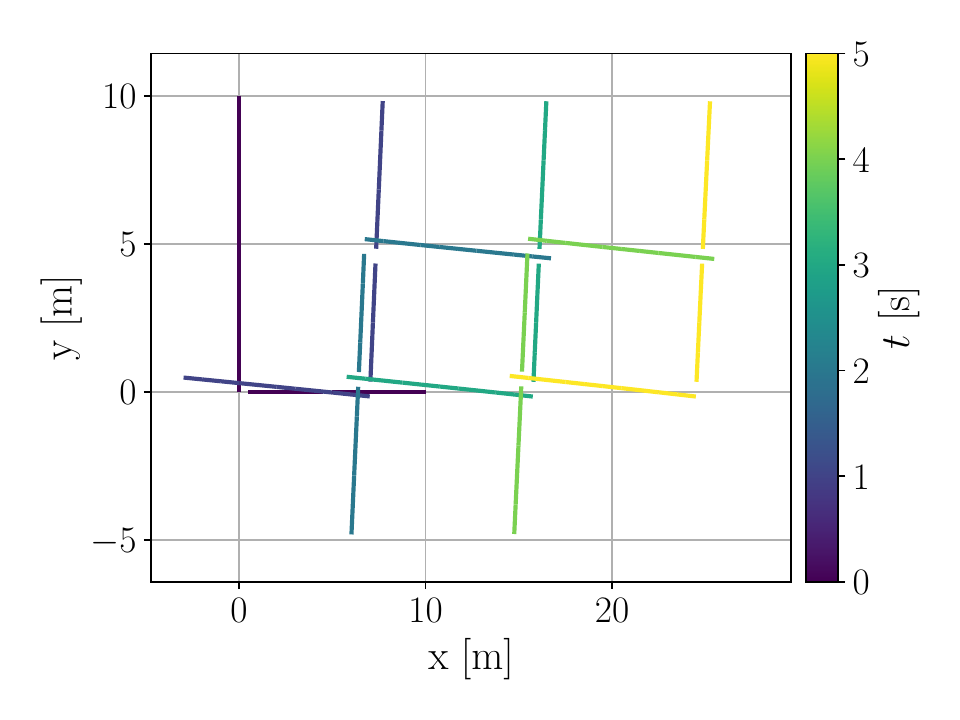}
        \caption{Rigid L-shaped frame.}
        \label{fig:configurations_Lframe_stiff}
    \end{subfigure}
    \hfill
    \begin{subfigure}[b]{0.48\linewidth}
        \centering
        \includegraphics[width=\linewidth]{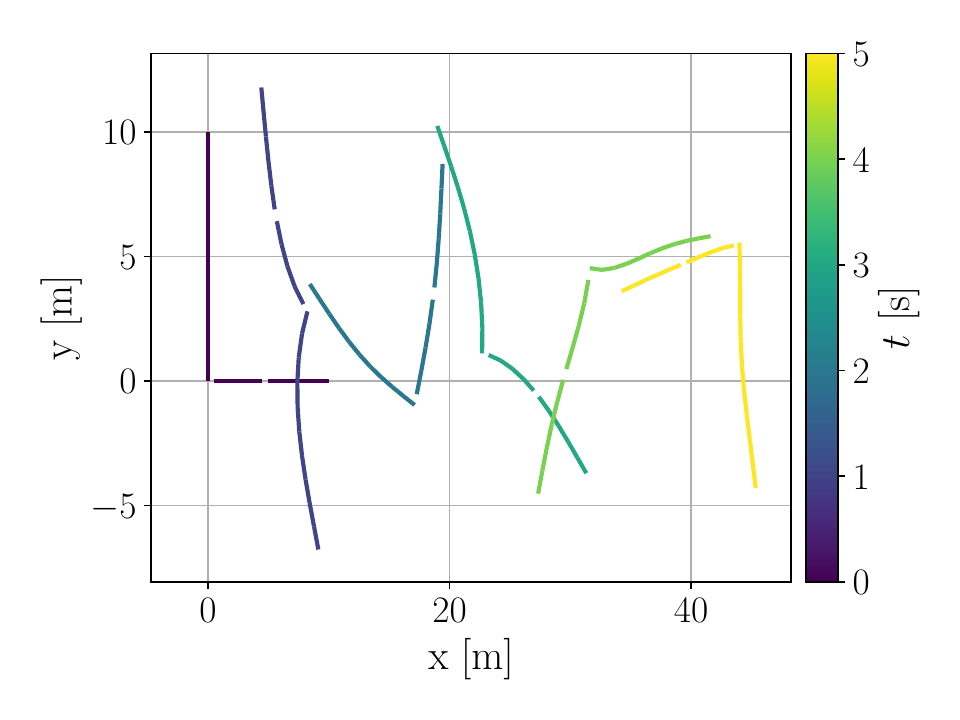}
        \caption{Soft L-shaped frame.}
        \label{fig:configurations_Lframe_soft}
    \end{subfigure}
    \caption{Output configurations for the L-shaped frames at time instants $t = \{1,2,3,4,5\} \; \mathrm{[s]}$.}
    \label{fig:configurations_Lframe}
\end{figure}

\begin{figure}[htb]
    \centering
    \begin{subfigure}[b]{0.48\textwidth}
        \centering
        \includegraphics[width=\textwidth]{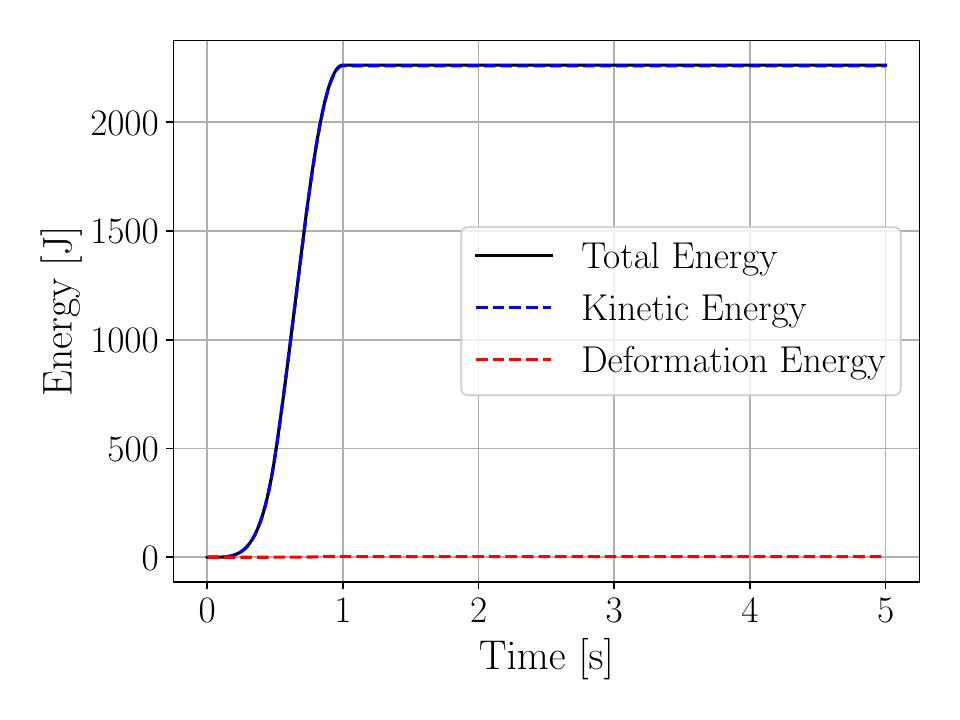}
    \end{subfigure}
    \hfill
    \begin{subfigure}[b]{0.48\textwidth}
        \centering
        \includegraphics[width=\textwidth]{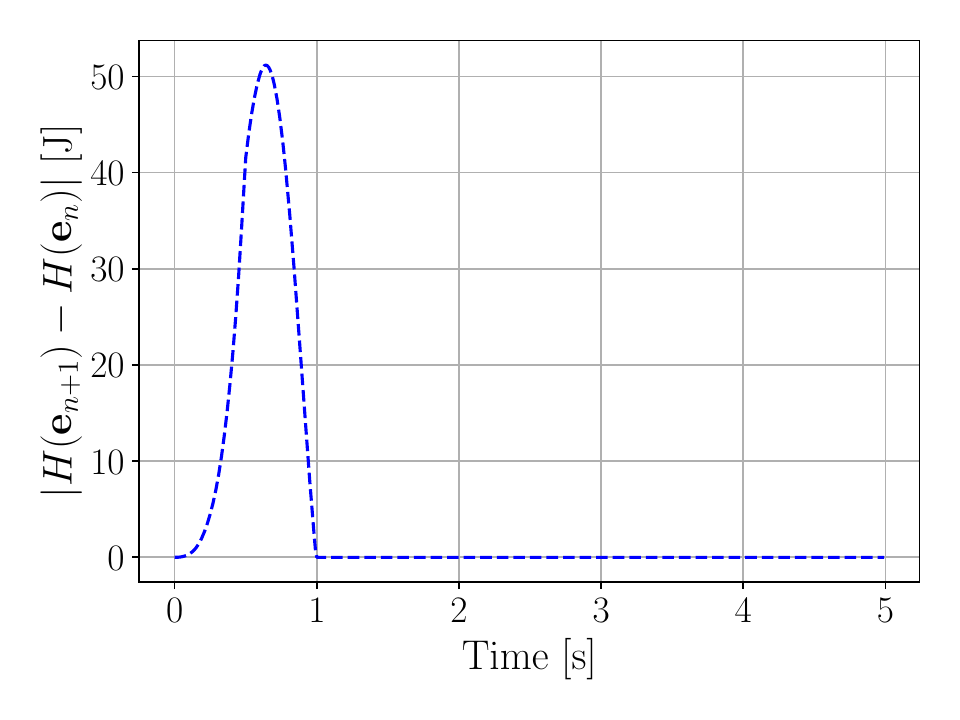}
    \end{subfigure}
    \caption{Rigid L-frame: evolution of energy and energy increments.}
    \label{fig:energy_Lframe_stiff}
\end{figure}

\begin{figure}[htb]
    \centering
    \begin{subfigure}[b]{0.48\textwidth}
        \centering
        \includegraphics[width=\textwidth]{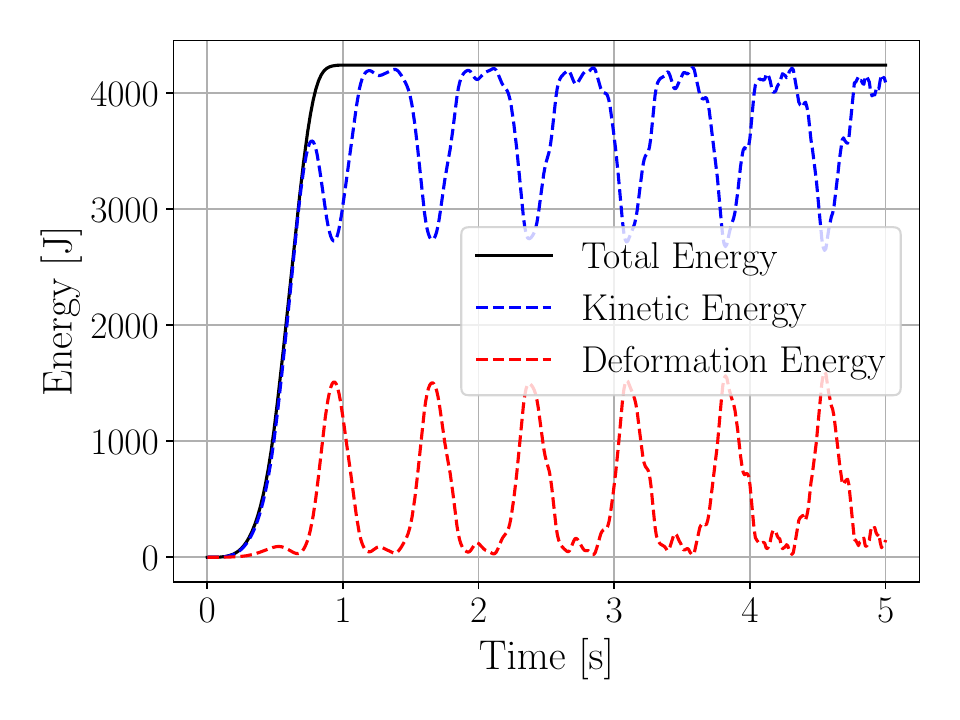}
    \end{subfigure}
    \hfill
    \begin{subfigure}[b]{0.48\textwidth}
        \centering
        \includegraphics[width=\textwidth]{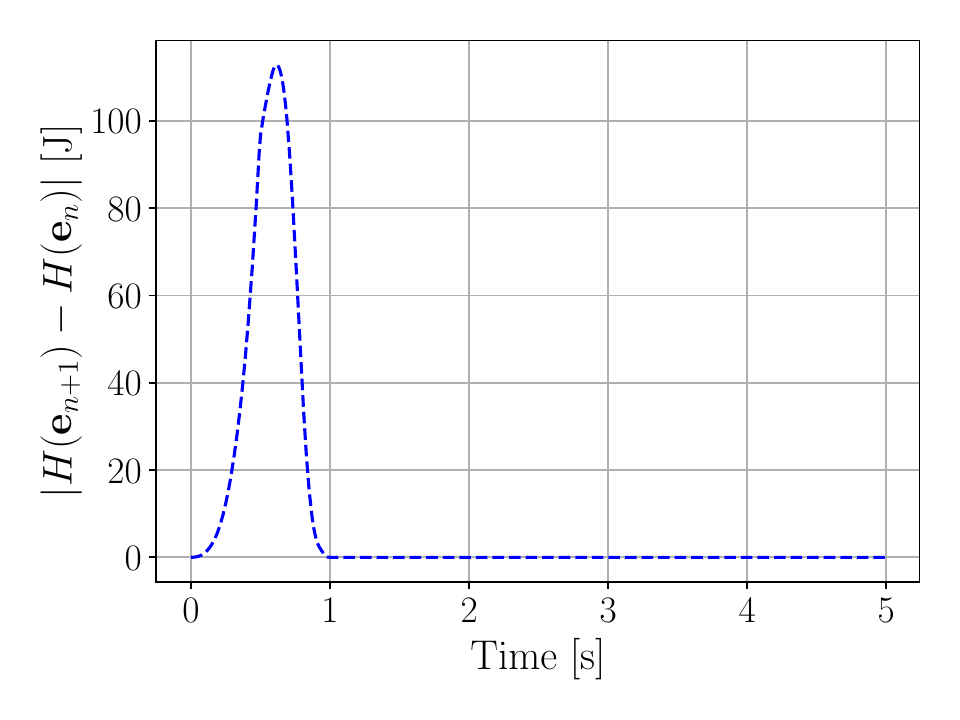}
    \end{subfigure}
    \caption{Soft L-frame: evolution of energy and energy increments.}
    \label{fig:energy_Lframe_soft}
\end{figure}

\subsection{Four bar mechanism}
\begin{table}[ht]
\centering
\begin{tabular}{ccccccccc}
\hline
$\Delta t$ [s] & $T$ [s] & $N_{\mathrm{el}}$ [-] & $L_i$ [m] & $\rho$ [kg/m$^3$] & $A$ [m$^2$] & $I$ [m$^4$] & $E$ [Pa] & $\nu$ [-] \\
\hline
0.02 & 10 & $\{10, 6, 6\}$ & $\{\sqrt{5}, 1.5, \sqrt{2}\}$ & 2710 & $0.05^2$ & $0.05^4/12$ & $2.1 \cdot 10^{11}$ & 0.3 \\
\hline
\end{tabular}
\caption{Parameters for the multibody system.}
\label{tab:parameters_fourbars}
\end{table}
The four bar mechanism is a common benchmark for multibody dynamics problems \cite{shabana2020}. The geometry of the problem and the construction of the system is explained in Sec. \ref{sec:four_bar} and reported in Figs. \ref{fig:four_bar} and \ref{fig:four_bar_blockdiagram}. The system is subjected to gravity acting in the negative $y$-direction, which is incorporated into the proposed formulation by including distributed input forces with constant value $\bm{n}_{\rm ext} = - g\rho A \bm{b}_2$ as explained in Sec. \ref{sec:gravity_forcing}. The geometric and physical parameters are reported in Table \ref{tab:parameters_fourbars}. \rewfour{The initial configuration of the mechanism, is reported in Fig. \ref{fig:fourbar_0}.} The analysis focuses on the motion of the center point $\mathsf P$ of the middle beam, whose initial position is given by $\bm{r}_{\mathsf P}(t=0) = (2.75,\, 1.5)$. The position of point $\mathsf P$ is shown in Fig.~\ref{fig:r_P_four_bar} and the configurations of the system at different time instants are shown in Fig.~\ref{fig:configurations_four_bar}. The results coincide with the findings in~\cite{kinon2025energy}. Fig.~\ref{fig:results_four_bar_energy} verifies the discrete-time energy-preservation of the proposed integration approach. \rewone{To assess the accuracy of the scheme, we also verify the convergence behavior for the scheme for different refinements level in time and space.  As a measure of accuracy, the error of the position of point $\mathsf{P}$ with respect to a reference solution at time $t=1 \; \mathrm{[s]}$ is considered, namely
$$
\mathrm{Error} = \frac{||\mathbf{r}_\mathsf{P}(t=1) - \mathbf{r}_{\mathrm{ref}, \mathsf{P}}(t=1)||}{\mathbf{r}_{\mathrm{ref}, \mathsf{P}}(t=1)}
$$
The reference solution is computed with a time step of $\Delta t_{\mathrm{ref}} = 0.005 \; \mathrm{[s]}$ for the temporal convergence. For the spatial convergence different refinements are considered with respect to the number of elements reported in Table \ref{tab:parameters_fourbars}. For the reference overkill solution, a refinement level of $16$ is chosen. The results are shown in Fig. \ref{fig:convergence_four_bar}. As expect from the implicit midpoint, the convergence rate is of order two in time. The error decreases with order two by increasing the refinement level. 
}

\begin{figure}
    \centering
\begin{tikzpicture}[
  >=Stealth,
  beam/.style        ={line width=1.4pt},
  joint/.style       ={circle, fill=blue!20, draw=black, inner sep=0pt, minimum size=6pt},
  tracept/.style     ={circle, fill=black, inner sep=0pt, minimum size=3pt},
  anglearc/.style    ={draw=red!70!black, -{Stealth[length=2mm]}, thin},
  refline/.style     ={dashed, gray},
  scale=1.5
]

\coordinate (A) at (0,0);
\coordinate (B) at (2.696,1.315);
\coordinate (C) at (4.296,1.315);
\coordinate (D) at (3.236,0);
\coordinate (P) at (3.496,1.315);

\draw[->] (A) -- ++(5.4,0) node[right] {$x$};
\draw[->] (A) -- ++(0,2.2) node[above] {$y$};

\draw[refline] (B) -- ++(1.3,0);
\draw[refline] (C) -- ++(1.3,0);

\draw[beam] (A) -- (B) node[midway, above left] {Beam 1};
\draw[beam] (B) -- (C) node[midway, below, xshift=-3pt] {Beam 2};
\draw[beam] (C) -- (D) node[midway, right]      {Beam 3};

\draw[anglearc] ($(A)+(0.9,0)$) arc (0:26:0.9);
\node[red!70!black] at ($(A)+(1.5,.3)$) {$\theta_1=26^\circ$};

\node[red!70!black, above] at ($(B)+(0,0.1)$) {$\theta_2=0^\circ$};

\draw[anglearc] ($(C)+(0.75,0)$) arc (0:225:0.75);
\node[red!70!black] at ($(C)+(1.25, 0.4)$) {$\theta_3=225^\circ$};

\node[joint] at (A) {};
\node[joint] at (B) {};
\node[joint] at (C) {};
\node[joint] at (D) {};
\node[tracept] at (P) {};

\node[below left]  at (A) {$A$};
\node[left, xshift=-5pt]        at (B) {$B$};
\node[above right] at (C) {$C$};
\node[below right]   at (D) {$D$}; 
\node[above]       at (P) {$P$};

\end{tikzpicture}
    \caption{Initial configuration for the four bar mechanism.}
    \label{fig:fourbar_0}
\end{figure}

\begin{figure}[tbh] 
\centering 
\begin{subfigure}[b]{0.48\textwidth} 
\centering \includegraphics[width=\textwidth]{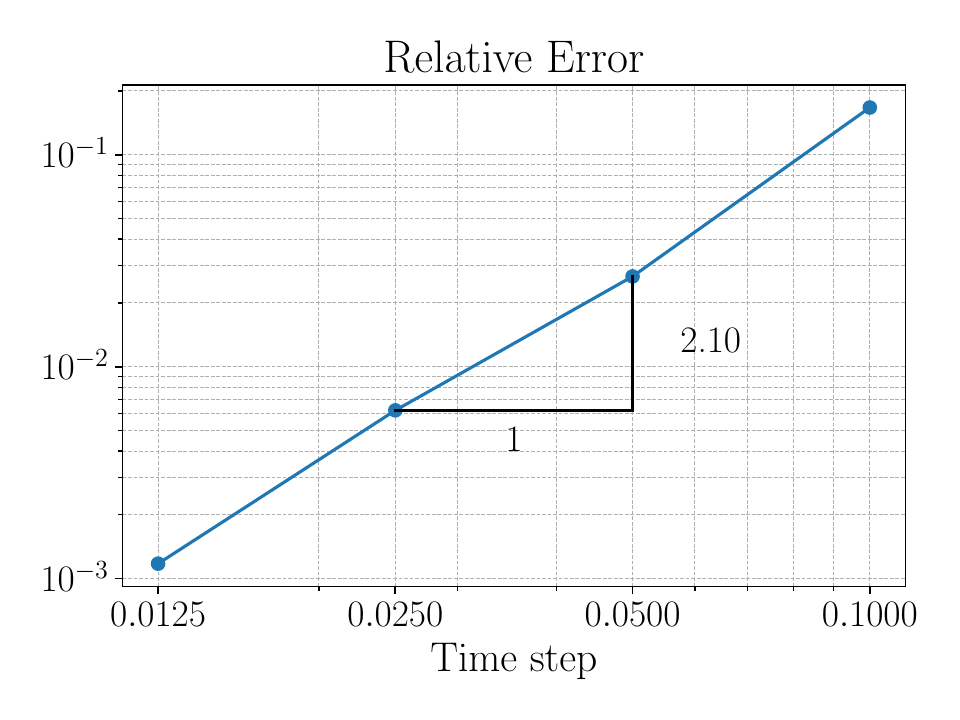} 
\caption{Time convergence} 
\end{subfigure} 
\hfill 
\begin{subfigure}[b]{0.48\textwidth} 
\centering 
\includegraphics[width=\textwidth]{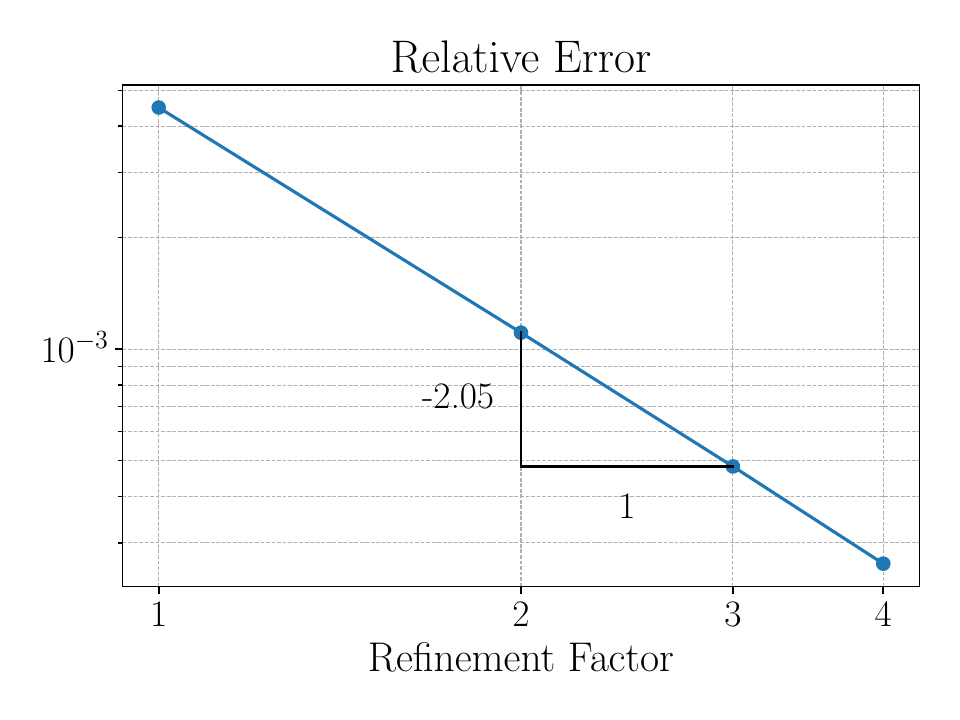} \caption{Spatial convergence} 
\end{subfigure} 
\caption{Convergence results for the error on point $\mathsf P$.} \label{fig:convergence_four_bar} 
\end{figure}

\begin{figure}[tbh] 
\centering 
\begin{subfigure}[b]{0.48\textwidth} 
\centering \includegraphics[width=\textwidth]{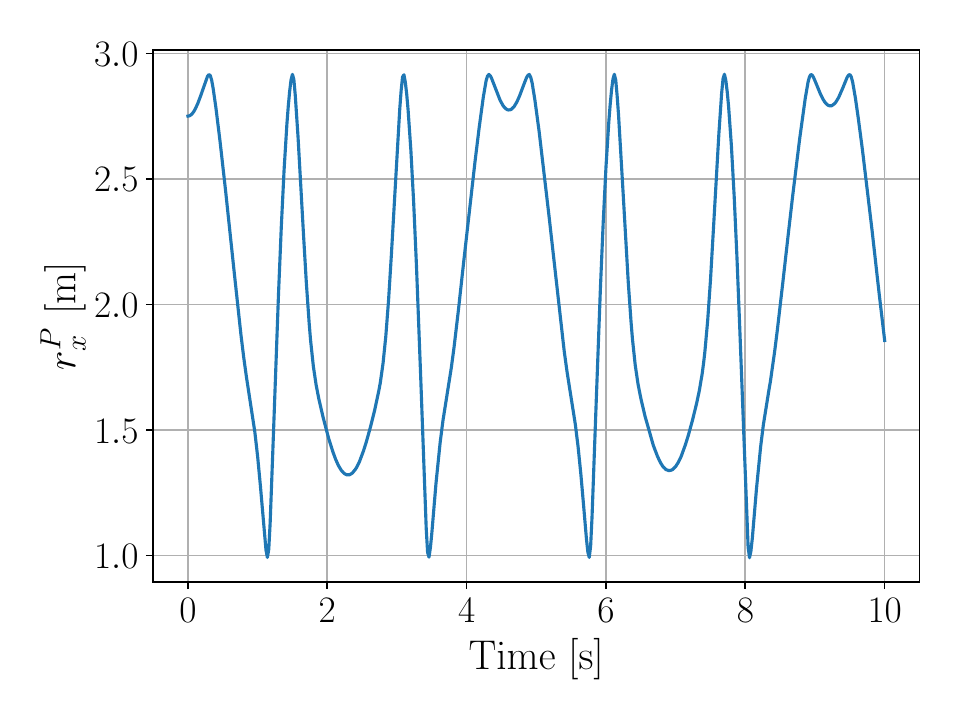} \caption{$r_{\mathsf{P},x}(t)$} 
\end{subfigure} 
\hfill 
\begin{subfigure}[b]{0.48\textwidth} 
\centering 
\includegraphics[width=\textwidth]{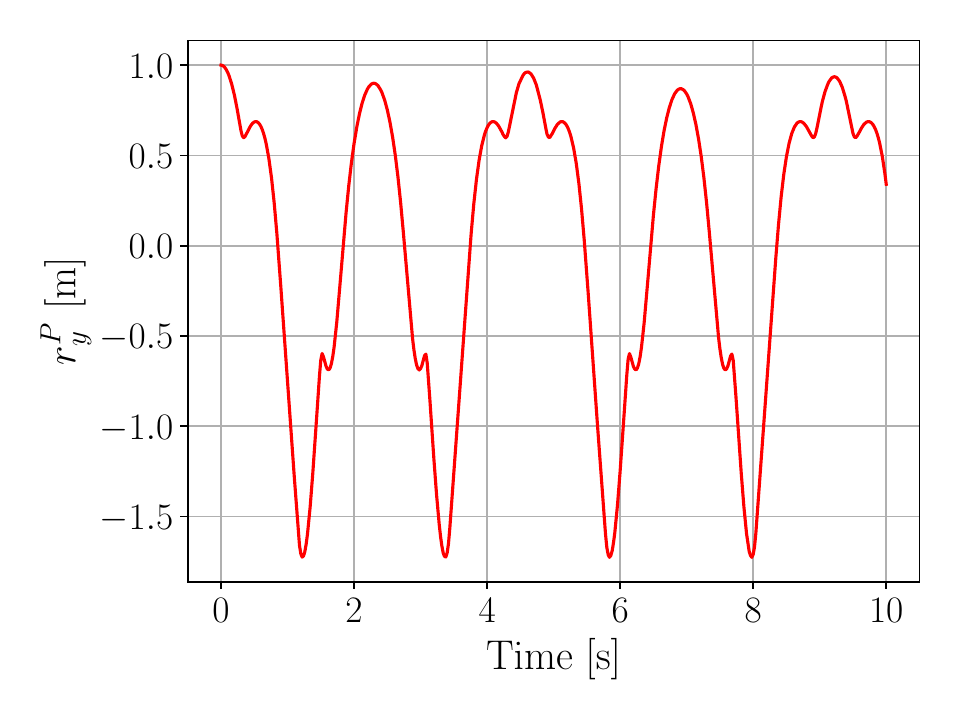} \caption{$r_{\mathsf{P},y}(t)$} 
\end{subfigure} 
\caption{Position of point $\mathsf P$ over time.} \label{fig:r_P_four_bar} 
\end{figure} 

\begin{figure}[tbh]
\centering

\begin{minipage}{0.48\textwidth}
    \centering
    \includegraphics[width=\textwidth]{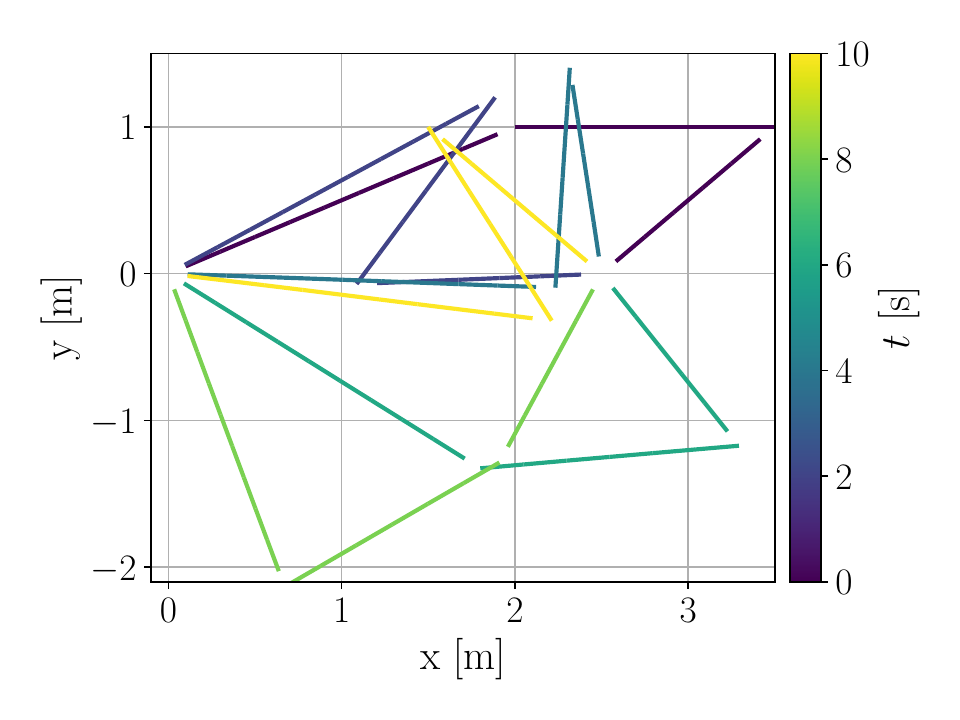}
    \caption{Configurations for the closed loop mechanism.}
    \label{fig:configurations_four_bar}
\end{minipage}
\hfill
\begin{minipage}{0.48\textwidth}
    \centering
    \includegraphics[width=\textwidth]{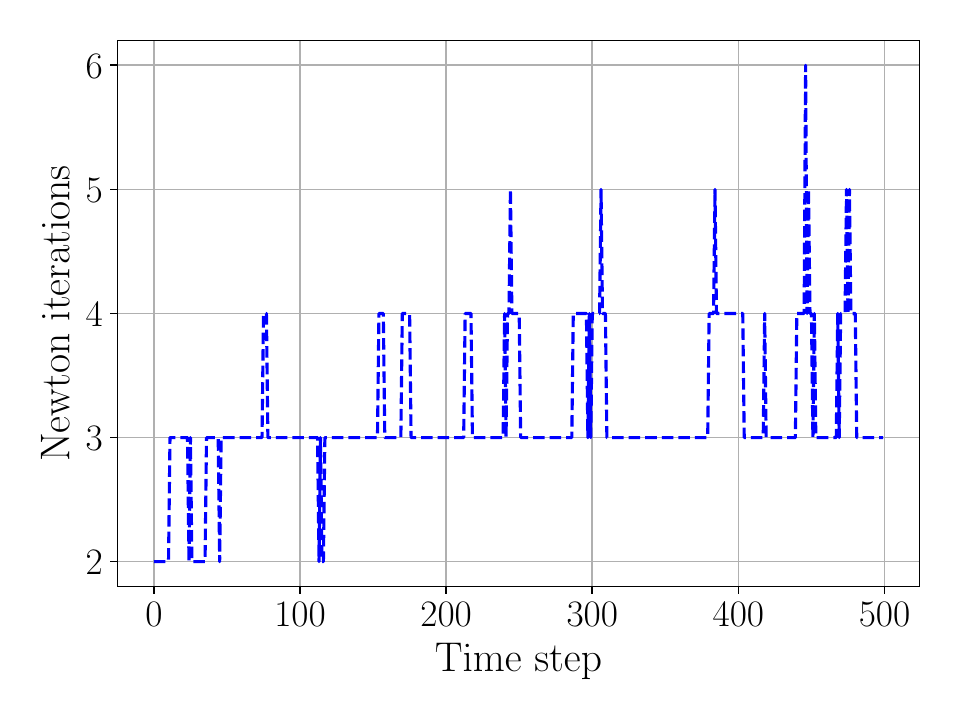}
    
    \caption{Newton iterations for the four bar mechanism }
\label{fig:newton_iterations_four_bar}
\end{minipage}

\end{figure}

\begin{figure}
    \centering
\input{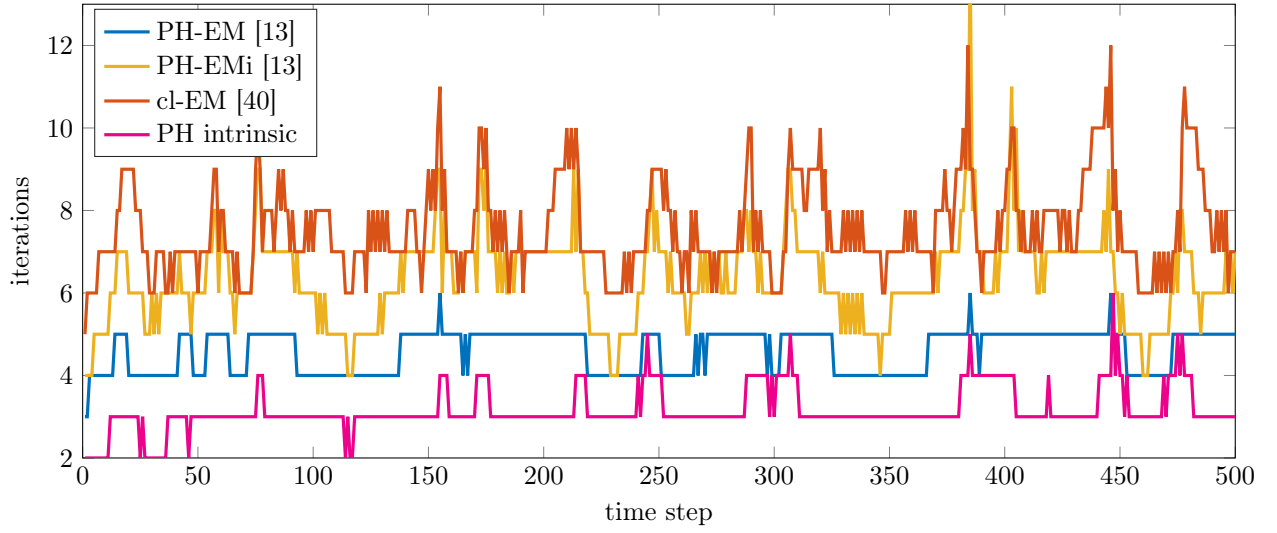}
    \caption{Comparison of Newton iterations for different integrators}
    \label{fig:comparison_newton_iterations}
\end{figure}

\begin{figure}[htb]
    \centering
    \begin{subfigure}[b]{0.45\textwidth}
        \centering
        \includegraphics[width=\textwidth]{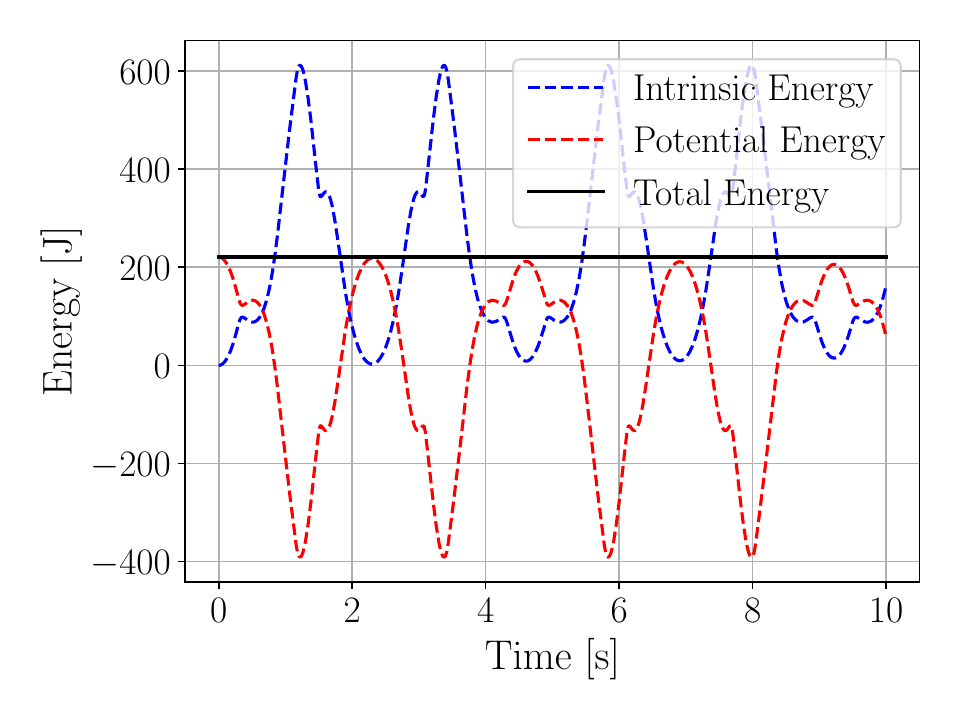}
    \end{subfigure}
    \hfill
    \begin{subfigure}[b]{0.45\textwidth}
        \centering
        \includegraphics[width=\textwidth]{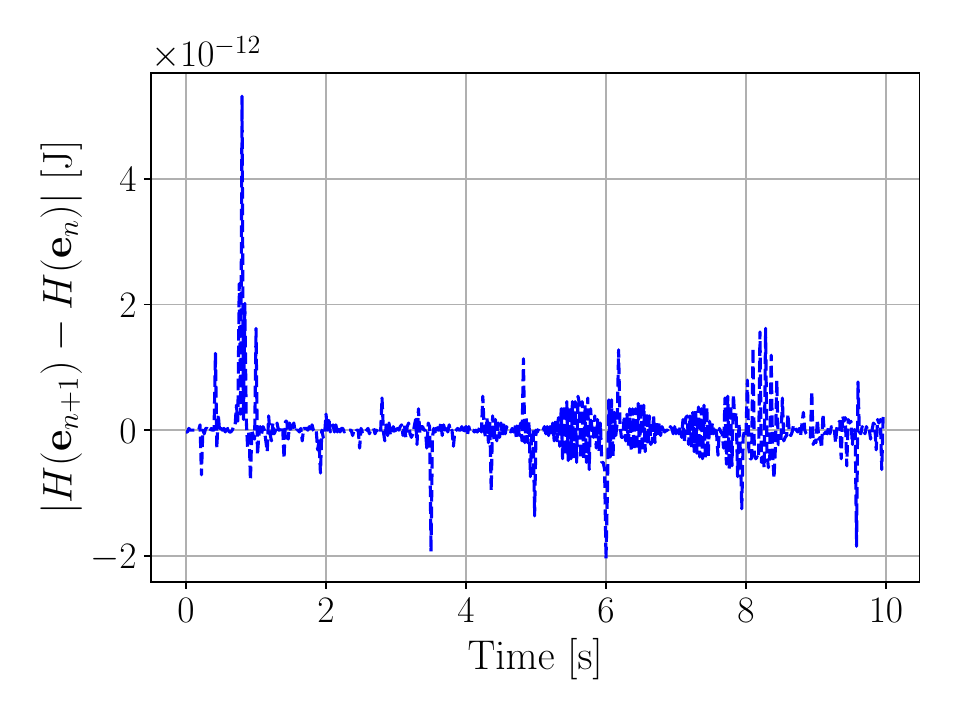}
    \end{subfigure}
    \caption{The intrinsic energy, i.e. the sum of kinetic and deformation energies $H=E_{\rm kin} + E_{\rm def}$, potential energy and total energy and total energy increments for the four bar mechanism.}
    \label{fig:results_four_bar_energy}
\end{figure}
Fig. \ref{fig:newton_iterations_four_bar} reports the number of Newton iterations for the proposed solver.
In Fig. \ref{fig:comparison_newton_iterations} the number of Newton iterations at each time step is shown for different methods: the port-Hamiltonian energy momentum preserving scheme (PH-EM) and its irreducible implementation (PH-EMi) presented in \cite{kinon2025energy},  the classical energy momentum scheme (cl-EM) proposed in \cite{stander1996energy} and the proposed approach based on the intrinsic formulation (labeled PH intrinsic). To establish a fair comparison the Newton's method tolerance is set to $10^{-5}$ for all methods and a central difference approach is used to calculate the Jacobian matrix. The proposed formulation requires less iterations than all other methods. Furthermore, the number of degrees of freedom for each beam scales as $O(5N_{el})$ for the intrinsic method because one needs to solve for five fields ($\theta, \; \mat{v}, \; \omega, \; \mat{n}, \; m$). In contrast, the method proposed in \cite{kinon2025energy} requires $O(6N_{el})$ degrees of freedom for each beam, since in the spatial-material formulation the dynamics of the centerline position is coupled to the other variables. The numbers of variables for further methods from the literature are reported in Table~\ref{tab:placeholder}.
\begin{table}[]
    \centering
    \begin{tabular}{c|cccc}
     & cl-EM  & PH-EM & PH-EMi & PH intrinsic\\
     \hline
     $\#$Variables/beam   & $O(4N_{el})$ & $O(6N_{el})$ & $O(4N_{el})$ & $O(5N_{el})$ \\
     Fields  & ($\bm{r}, \; \theta, \; \mat{v}, \; \omega$) & ($\bm{r}, \; \theta, \; \mat{v}, \; \omega, \; \mat{n}, \; m$) & ($\bm{r}, \; \theta, \; \mat{v}, \; \omega$) & ($\theta, \; \mat{v}, \; \omega, \; \mat{n}, \; m$)
    \end{tabular}
    \caption{Number and nature of variables for different schemes}
    \label{tab:placeholder}
\end{table}

\section{Conclusion}
In this work, a structure-preserving discretization for multibody systems of geometrically exact beams has been proposed. The strategy relies on an intrinsic port-Hamiltonian beam model with different mixed finite element representations, that allow for different natural boundary conditions. Via feedback interconnections, multibody system can be assembled without the need for Lagrange multipliers. 
\rewthree{Since the intrinsic form of the equations relocates the nonlinearity to the algebraic part of the interconnection operator, the energy is a quadratic form in its variables. Thanks to this feature, the implicit midpoint integrator allows an exact discrete-time representation of the total energy balance.} \rewthree{Here,
exact denotes equality of the actual change of the discrete energy with respect to the qudrature approximation of the supplied. It does not signify zero approximation error against the true
trajectory --- except in the unforced case, where the balance trivializes
to an exactly conserved quadratic invariant.
}

To make the presented framework applicable to generic systems, the procedure can be automated to construct a library commonly found multibody assembly (for instance three rods combined in a triangular shape \cite{finozzi2022parametric}). \rewone{For some topologies, however, it may not be possible to interconnect systems with different boundary conditions and avoid Lagrange multipliers. For instance, when multiple beams converge at a single pivot node, one beam must be treated as free while the others are treated as pinned. However, this choice may introduce inconsistencies with the remaining interconnections in the network. In general, to avoid the introduction of Lagrange multipliers, regularization approaches—such as the addition of artificially small inertial terms in connecting nodes—can be employed.}

To improve the computational efficiency of the proposed methodology, the integration approach presented in \cite{brugnoli2025nonlinear} can be applied to the present formulation. In this integrator, the rotation and coenergy variables are staggered in time and therefore decoupled. Since the nonlinear coupling terms between different multibody components are bilinear in the rotation and coenergy variables (cf. Eq. \eqref{eq:discrete_nonlinear_multibody}), this temporal decoupling renders their integration linear. The only remaining nonlinearity arises from the intrinsic dynamics, which exhibits a block-diagonal structure. As a result, the number of Newton iterations required at each time step may be further reduced. 

\rewone{An extension of the present approach to three-dimensional beam systems is a promising direction for future work; see \cite{kinon2026mixed} for the corresponding spatial-material formulation. However, the extension is not straightforward, since rotations belong to \(SO(3)\) and arbitrary spatial joint orientations lead to additional complexity in the interconnection and compatibility conditions. In particular, a rotation parametrization that avoids introducing Lagrange multipliers while preserving the geometric exactness of the beam formulation is required. One possible route is to employ exponential coordinates for finite rotations, or incremental rotation updates, which avoid the explicit enforcement of orthonormality or unit-length constraints \cite{bruls2012lie,yu2013attitude,muller2018lie,todesco2023exp}. Developing such a parametrization together with a consistent treatment of spatial joint orientations and their interconnections constitutes an important part of the required three-dimensional extension.} \\

\rewfour{Standard control, estimation, and optimization algorithms are natively designed to interface with explicit state-space ordinary differential equations (ODEs) \cite{dopico2014ode,barfoot2024robotics}. When multibody systems rely on differential-algebraic equations (DAEs) to enforce kinematic constraints, applying these algorithms requires complex mathematical workarounds—such as index reduction, manifold projection, or specialized DAE solvers—which limit their applicability. By replacing algebraic constraints with power-preserving feedback interconnections, our framework yields a pure ODE structure that allows the direct use of several computational tools for nonlinear control and estimation.}

\section*{Code Availability}
The code used to generate the results is available at the GitHub repository \url{https://github.com/a-brugnoli/intrinsic-geometrically-exact-beams}.

\appendix

\section{Construction of the $\mathbf{P}({\theta})$ matrix}
\label{sec:p_matrix}

The $\mathbf{P}$ matrix, discussed in Sec. \ref{sec:nonlinear_formulation}, takes different algebraic realizations depending on the formulation. The overall structure is the following
$$
\mathbf{P}(\bm{\theta}) = \begin{bmatrix}
    0 & 0 & \mathbf{R}(\bm{\theta}) & 0 \\
    0 & 0 & 0 & \mathbf{A}
\end{bmatrix}
$$
where $\mathbf{R}(\bm{\theta})$ is associated with the term $\int_\Omega \bm{\psi}_v \bm{\Lambda}(\theta_h) \bm{v}_h \d{s}$ and $\mathbf{A}$ is associated with the term $\int_\Omega \psi_\omega \omega \d{s}$. This is because positions and angular rotations are discretized using the same space of velocities and angular velocities respectively. $\mathbf{A}$ can be either $\mathbf{M}$ (the mass matrix obtained using linear Lagrange finite elements) when the angular velocity is discretized via a continuous space or $h\mathbf{I}$ if the angular velocity is discretized using a discontinuous space.

In the following, the matrices 
$$\mathbf{\Phi} = \begin{bmatrix}
    \phi_1 & \phi_2 & \dots & \phi_{N_{el} + 1}
\end{bmatrix} \in \mathbb{R}^{1\times N_{el+1}}
\qquad\qquad \mathbf{\Xi} = \begin{bmatrix}
    \xi_1 & \xi_2 & \dots & \xi_{N_{el}}
\end{bmatrix} \in \mathbb{R}^{1\times N_{el}}
$$
collect the finite element bases in a single row vector. The $\mathbf{R}$ and $\mathbf{A}$ matrices
for the different formulations are as follows
\begin{itemize}
\item Free case:
$$
\mathbf{R}(\bm{\theta}) = \int_\Omega \mathbf{\Phi}_{\otimes 2}^\top \bm{\Lambda}(\mathbf{\Phi}\bm{\theta}) \mathbf{\Phi}_{\otimes 2} \d{s}, \qquad\qquad \mathbf{A} = \mathbf{M}.
$$
\item Clamped case:
$$
\mathbf{R}(\bm{\theta}) = \int_\Omega \mathbf{\Xi}_{\otimes 2}^\top \bm{\Lambda}(\mathbf{\Xi}\bm{\theta}) \mathbf{\Xi}_{\otimes 2} \d{s}, \qquad\qquad \mathbf{A} = h\mathbf{I}.
$$
\item Pinned case:
$$
\mathbf{R}(\bm{\theta}) = \int_\Omega \mathbf{\Xi}_{\otimes 2}^\top \bm{\Lambda}(\mathbf{\Phi}\bm{\theta}) \mathbf{\Xi}_{\otimes 2} \d{s}, \qquad\qquad \mathbf{A} = \mathbf{M}.
$$
\item Guided case:
$$
\mathbf{R}(\bm{\theta}) = \int_\Omega \mathbf{\Phi}_{\otimes 2}^\top \bm{\Lambda}(\mathbf{\Xi}\bm{\theta}) \mathbf{\Phi}_{\otimes 2} \d{s}, \qquad\qquad \mathbf{A} = h\mathbf{I}.
$$
\end{itemize}

\bibliographystyle{elsarticle-num}
\bibliography{biblio}

\end{document}